\documentclass{article}
\usepackage{arxiv}
\RequirePackage{amsthm,amsmath,amsfonts,amssymb}

\usepackage{float}
\usepackage[caption = false,farskip=0pt]{subfig}

\usepackage{natbib}

\RequirePackage{graphicx}
\usepackage{algorithmic}
\usepackage[linesnumbered,ruled,vlined]{algorithm2e}
\usepackage{comment}

\newtheorem{theorem}{Theorem}
\newtheorem{lemma}{Lemma}
\newtheorem{corollary}{Corollary}

\newtheorem{assumption}{Assumption}

\usepackage{xcolor}

\usepackage{inputenc}

\title{Kurtosis-based projection pursuit for matrix-valued data}

\author{
Una Radojicic  \\
  Vienna University of Technology\\
  \texttt{una.radojicic@tuwien.ac.at} \\
\And
Klaus Nordhausen \\
  University of Jyv\"askyl\"a\\
  \texttt{klaus.k.nordhausen@jya.fi} \\
  \And
Joni Virta  \\
  University of Turku\\
  \texttt{joni.virta@utu.fi} \\

  }

\date{\today}

\begin{document}

\maketitle
\begin{abstract}
We develop projection pursuit for data that admit a natural representation in matrix form. For projection indices we propose extensions of the classical kurtosis and Mardia's multivariate kurtosis. The first index estimates projections for both sides of the matrices simultaneously, while the second index finds the two projections separately. Both indices are shown to recover the optimally separating projection for two-group Gaussian mixtures in the full absence of any label information. We further establish the strong consistency of the corresponding sample estimators. Simulations and a real data example on hand-written postal code data are used to demonstrate the method. 
\end{abstract}


\section{Introduction}\label{sec:intro}

Projection pursuit (PP) is a classical multivariate method of compressing information contained in a $ p $-dimensional random vector $ \textbf{x} $ into an easy-to-visualize, low-dimensional form. At its simplest, PP involves choosing a \textit{projection index} $ Q $ (that maps a univariate random variable to a real number) and finding a projection direction $ \textbf{u} \in \mathbb{S}^{p - 1} $ that maximizes $ Q(\textbf{u}' \textbf{x}) $, where  $ \mathbb{S}^{p - 1} $ denotes the unit sphere in $ \mathbb{R}^p $. The index $ Q $ is typically thought to measure ``information'', in one sense or another, and several examples of indices are discussed in the seminal survey article by \cite{huber1985projection}.

Arguably, the need for methods such as projection pursuit that extract information from data that are unfeasible to visualize in their original form has only increased in the previous years, with the ever-increasing output of data from all sources. However, applying projection pursuit to data with large $ n, p $ (and not necessarily $ n < p $)  is less than straightforward due to the following three challenges that PP faces in the high-dimensional context.

\begin{enumerate}
\item[i)] As shown by \cite{diaconis1984asymptotics}, in regimes where $ n, p \rightarrow \infty $ most univariate projections (in the sense of drawing uniformly random projection directions from the unit sphere) of high-dimensional point clouds are approximately normal. This result in conjunction with the standard maxim of projection pursuit that ``Gaussianity is noise'' \citep{huber1985projection} thus implies that in high-dimensional settings there is little information to be found by projection pursuit. Further results in this spirit can be found in \cite{bickel2018projection}.
\item[ii)] \cite{pires2019high} show that when $p \geq n - 1$, a bivariate projection equaling (up to an affine transformation) any given two-dimensional configuration of points can always be found from the data. Note that there is no contradiction involved between points i) and ii) as the former makes an asymptotical statement (as $ n, p \rightarrow \infty $, most projections are normal) while the latter relates to finite samples (for finite samples, any arbitrary configuration can be found).
\item[iii)] Most projection pursuit algorithms scale badly with the dimension $ p $. For example, methods in the classical FastICA-family \citep[see e.g.,][]{Miettinen2017TheSS} involve the inversion of a $ p \times p $ covariance matrix, quickly leading to numerical issues with a growing number of dimensions. The standard evolutionary algorithms~\citep{EvolutionaryAlgorithms2020} suffer even worse as the curse of dimensionality inflates the search space, necessitating the use of a substantial initial population to even adequately cover the whole range of possible solutions~\citep{Salomon2004}.
\end{enumerate}

The above points make applying PP in the high-dimensional setting seem like a rather futile exercise. However, all three can be circumvented by imposing a suitable structure, or a form of \textit{regularization}, on the data. In this work, the regularizing assumption we make is that the observable $ pq $-variate random vector $ \textbf{x} $ has a natural representation as a random $ p \times q $ matrix $ \textbf{X} $, such that $ \mathrm{vec}(\textbf{X}) = \textbf{x} $, where the $ \mathrm{vec} $-operator stacks the columns of its argument from left to right into a tall column vector. Such data are abundant nowadays, the most notable example being image data where the elements of the matrix represent the gray-scale intensities of the individual pixels of an image. Other examples include magnetic resonance imaging (MRI) data or biological abundance data where each observed matrix contains the abundances of a single species in $p$ regions (the rows) over $q$ time points (the columns).

To make use of the matrix structure, we work with projections of the form $ \textbf{u}' \textbf{X} \textbf{v} $ where $ \textbf{u} \in \mathbb{S}^{p - 1} $, $ \textbf{v} \in \mathbb{S}^{q - 1} $. Due to the unit length constraint, such projections have a total of $ p + q - 2 $ degrees of freedom, a stark contrast with the standard approach of treating the matrix $ \textbf{X} $ in the vector form $ \mathrm{vec}(\textbf{X}) $ and working with the usual projections $ \textbf{w}'\mathrm{vec}(\textbf{X}) $ involving $ pq - 1 $ degrees of freedom. As an example, consider a single horizontal slice of an fMRI image, typically $ 64 \times 64 $ pixels in size \citep{lindquist2008statistical}. The dual projection $ \textbf{u}' \textbf{X} \textbf{v} $ on $ \textbf{u} \in \mathbb{S}^{63} $, $ \textbf{v} \in \mathbb{S}^{63} $ involves a total of $ 126 $ parameters, whereas a standard projection after vectorization has a total $ 64^2 - 1 = 4095 $ parameters, with a difference of more than one order of magnitude. 
Of course, this reduction in the number of parameters would be meaningless unless the simpler model provided good fits to real data and, indeed, numerous case studies illustrate that modeling a random matrix through the simultaneous transformation of its rows and columns (as in $ \textbf{u}' \textbf{X} \textbf{v} $) provides interpretable and efficient results, see, e.g., \cite{zhang20052d, beckmann2005tensorial} and the references in \cite{lu2011survey}.

Projections of the form $ \textbf{u}' \textbf{X} \textbf{v} $ are commonly known as (orthogonal) rank-1 tensor projections, and they have been long used in the machine learning literature, most often in the context of classification using projection indices based on second moments, see, e.g., \cite{hua2007face, liu2011discriminant, wu2011local,wu2011rank,zhong-2015}. The term ``rank-1 tensor projection'' stems from the representation $ \textbf{u}' \textbf{X} \textbf{v} = \mathrm{tr}(\textbf{v} \textbf{u}' \textbf{X}) = \langle \textbf{u} \textbf{v}', \textbf{X} \rangle$ of projecting $ \textbf{X} $ onto the rank-1 matrix $ \textbf{u} \textbf{v}' $ (w.r.t. the Euclidean inner product on matrices) and ``orthogonal'' refers to the fact that if several projections are sought, they should be in mutually orthogonal directions (similarly as in PCA). Interestingly, the structure-ignoring projection $ \textbf{w}'\mathrm{vec}(\textbf{X}) $ also has an equivalent representation as a ``rank-$ p $ projection'' (assuming $ p \leq q $)  as $ \textbf{w}'\mathrm{vec}(\textbf{X}) =  \langle \textbf{W}, \textbf{X} \rangle$, where $ \mathrm{vec} (\textbf{W}) = \textbf{w} $. This shows that the two approaches, i.e., rank-1 tensor projections and projection after vectorization, are actually the two extreme cases of a range of projections of different ranks. 

The viewpoint we take for our analysis of matrix projection pursuit is that of clustering, an objective often associated with projection pursuit, see \cite{friedman1974projection, pena2001cluster, bolton2003projection, loperfido2015vector,radojicic2021}. Our work is especially in the spirit of: i) \cite{pena2001cluster} who showed that projection pursuit with the classical kurtosis as the projection index can recover the optimal projection (in the sense of LDA) for separating the classes of a mixture distribution, in absence of the class label information, and of ii) \cite{radojicic2021} who continued the work of \cite{pena2001cluster} by investigating the asymptotic efficiencies of various standard indices under the two-component Gaussian mixture model.

In the following, we consider two novel objective functions (projection indices) that are based on the fourth moments of the matrix~$ \textbf{X} $,
\begin{align}\label{eq:objective_functions}
\frac{\mathrm{E} \left( [ \textbf{u}' \{ \textbf{X} - \mathrm{E}(\textbf{X}) \} \textbf{v} ]^4 \right) }{ \left\{ \mathrm{E} \left( [ \textbf{u}' \{ \textbf{X} - \mathrm{E}(\textbf{X}) \} \textbf{v} ]^2 \right) \right\}^2} \quad \mbox{and} \quad \mathrm{E}\left[ \left\{ \textbf{u}' \tilde{\textbf{X}} \left[ \mathrm{E} \left( \tilde{\textbf{X}}' \textbf{u} \textbf{u}' \tilde{\textbf{X}} \right) \right]^{-1} \tilde{\textbf{X}}' \textbf{u} \right\}^2  \right],
\end{align}
where $ \tilde{\textbf{X}} := \textbf{X} - \mathrm{E}(\textbf{X}) $, see Section \ref{sec:method} for their motivation as extensions of the classical kurtosis. Analogously to \cite{pena2001cluster}, our Theorems~\ref{theo:matrix_lda_optimal_higher_rank} and~\ref{theo:matrix_lda_optimal_mode_wise} in Section~\ref{sec:lda} show that both indices in \eqref{eq:objective_functions} can be used to recover the direction that is optimal for separating the classes of a mixture of matrix normal distributions (in the sense of LDA) without any knowledge on the class memberships. Moreover, similarly to \cite{radojicic2021}, we explore the asymptotic behavior of the resulting estimators. However, we stress that our contributions are more than simple generalizations of the vectorial results of \cite{pena2001cluster} and \cite{radojicic2021} to the matrix case and go beyond these works for the following reasons: 
\begin{itemize}
	\item[i)] Unlike in the vector case, the projections $ \textbf{u}' \textbf{X} \textbf{v} $ involve two projection directions, $ \textbf{u} $ and $ \textbf{v} $, and we propose two competing indices for determining them: optimizing the first index in \eqref{eq:objective_functions} determines both directions simultaneously (the index involves both $ \textbf{u} $ and $ \textbf{v} $) while the second index in \eqref{eq:objective_functions} can be used to determine just the direction $ \textbf{u} $, in isolation of $ \textbf{v} $ (after which $ \textbf{v} $ can be determined by the same index after replacing $ \textbf{X} $ with its transpose due to the symmetric nature of the projection $ \textbf{u}' \textbf{X} \textbf{v} $).
	\item[ii)] The optimal projection direction $\textbf{W}_{\mathrm{LDA}}$ is a matrix (see Lemma \ref{lem:matrix_lda_optimal} in Section~\ref{sec:lda} for its exact form) and to fully recover it we need to extract a total of $ d $ pairs of projections $(\textbf{u}_1, \textbf{v}_1), \ldots , (\textbf{u}_d, \textbf{v}_d)$ where $ d $ is the rank of $\textbf{W}_{\mathrm{LDA}}$. Theorems~\ref{theo:matrix_lda_optimal_higher_rank} and~\ref{theo:matrix_lda_optimal_mode_wise} in Section~\ref{sec:lda} show that these pairs actually recover (in the specific sense described in Section~\ref{sec:lda}) the singular value decomposition of the optimal projection direction $\textbf{W}_{\mathrm{LDA}}$. Hence, to alleviate computational burden (which, while significantly lighter than in standard projection pursuit, can still be somewhat heavy for large data sets), accurate approximations of $\textbf{W}_{\mathrm{LDA}}$ can be obtained by extracting only a small number of pairs of directions. The theorems also reveal that, of the two indices in \eqref{eq:objective_functions}, the former is always a Fisher consistent estimator of the optimal projection direction $\textbf{W}_{\mathrm{LDA}}$ while the latter requires a mild (but in practice difficult to verify) assumption on the singular values of $\textbf{W}_{\mathrm{LDA}}$ to achieve Fisher consistency, see Section \ref{sec:lda} for details and intuition. There are no counterparts for these results in the context of \cite{pena2001cluster} and \cite{radojicic2021} where the optimal projection direction is a vector.
\end{itemize}

As mentioned already earlier, the majority of the literature combining projections with matrix-valued data focuses on using second-order projection indices for classification purposes. Unlike the ones in \eqref{eq:objective_functions}, these indices are often \textit{supervised}, in the sense that they involve knowledge on the class labels of the data, restricting their use strictly to the classification setting where a training sample with known labels is available. Two of the most popular \textit{unsupervised} projection methods (i.e., ones that do not use require label information) are known as MPCA \citep{ye2005generalized} and (2D)$ ^2 $PCA \citep{zhang20052d}, and they can be seen as second-order counterparts for our proposed fourth-order indices in \eqref{eq:objective_functions}. 
Namely, in MPCA one searches for projection directions  $ \textbf{u} \in \mathbb{S}^{p - 1} $, $ \textbf{v} \in \mathbb{S}^{q - 1} $ which maximize the quantity
\[
\mathrm{E} \left( [ \textbf{u}' \{ \textbf{X} - \mathrm{E}(\textbf{X}) \} \textbf{v} ]^2 \right).
\]
Whereas, analogously to the second index in \eqref{eq:objective_functions}, (2D)$ ^2 $PCA involves separately  directions $ \textbf{u} \in \mathbb{S}^{p - 1} $ and $ \textbf{v} \in \mathbb{S}^{q - 1} $, and maximizes both
\[
\mathrm{E} \left[ \textbf{u}' \{ \textbf{X} - \mathrm{E}(\textbf{X}) \}  \{ \textbf{X} - \mathrm{E}(\textbf{X}) \}' \textbf{u} \right] \quad \mbox{and} \quad \mathrm{E} \left[ \textbf{v}' \{ \textbf{X} - \mathrm{E}(\textbf{X}) \}'  \{ \textbf{X} - \mathrm{E}(\textbf{X}) \} \textbf{v} \right].
\]
In Section \ref{sec:lda} we will compare the indices in \eqref{eq:objective_functions} to 2D$ ^2 $PCA and MPCA and show that the latter are, in general, unable to recover the optimally separating direction under a matrix normal mixture. This result is analogous to the inability of PCA to recover Fisher's linear discriminant in standard LDA. In addition, in the examples of Section \ref{sec:simulations} we will compare our unsupervised proposal \eqref{eq:objective_functions} to matrix-valued LDA, which can be seen as the most direct supervised solution to the problem. Finally, note that the term ``matrix LDA'' does not have an agreed-upon standard definition in the literature (see, e.g., \cite{Hu2020} for a recent proposal) and the version we use is based on the comparison of density functions, see Lemma \ref{lem:matrix_lda_optimal}.


The rest of the paper is organized as follows. In Section \ref{sec:method} we motivate the two indices in \eqref{eq:objective_functions} and lay the basis on our methodology by showing that the optimization of the indices is well-defined under a certain mild condition on the data distribution. Section \ref{sec:lda} studies the behavior of matrix projection pursuit under the matrix normal mixture model, establishing several results on Fisher consistency and concludes with a theoretical comparison to 2D$ ^2 $PCA and MPCA. In Section \ref{sec:limiting} we further show that both indices yield strongly consistent estimators of the optimal direction under matrix normal mixtures. Afterward, we shift our focus towards the first index in \eqref{eq:objective_functions} and provide an algorithm for optimizing it (Section \ref{sec:algorithm}). The reason for excluding the second index in \eqref{eq:objective_functions} from the further study is that we found optimizing it to be computationally more demanding than optimizing the former index, while at the same time one of the motivations for introducing it was to ease the computational burden. Thus, we have   left it for future work. 
Finally, in Section \ref{sec:simulations}, we apply the method first to simulated data to investigate its finite-sample performance, and later to a hand-written digit data set, including also a comparison to competing methods. The proofs of the technical results are collected in Appendix~\ref{sec:proofs}.


\section{Notation}\label{sec:notation}

Throughout the paper, we work in a probability space $(\Omega, \mathcal{F}, \mathbb{P})$. We let $\mathbb{S}^{p-1}$ denote the unit sphere in $\mathbb{R}^p$. We assume that the dimensions $p, q \in \mathbb{N}$ are fixed throughout and denote $\mathcal{U}_0 := \mathbb{S}^{p-1} \times \mathbb{S}^{q-1}$. Given a function $g: \mathcal{U}_0 \rightarrow \mathbb{R}$ and collections of matrices, $\mathcal{G}_{1} = \{ \textbf{G}_{11}, \ldots , \textbf{G}_{1(d-1)} \} \in (\mathbb{R}^{p \times p})^{d-1}$ and $\mathcal{G}_{2} = \{ \textbf{G}_{21}, \ldots , \textbf{G}_{2(d-1)} \} \in (\mathbb{R}^{q \times q})^{d-1}$, we say that a collection of pairs of vectors, $ (\textbf{u}_1, \textbf{v}_1), \ldots , (\textbf{u}_d, \textbf{v}_d) \in \mathcal{U}_0 $, $d \leq \min\{ p , q \}$, is a sequence of $(\mathcal{G}_{1}, \mathcal{G}_{2})$-minimizers ($(\mathcal{G}_{1}, \mathcal{G}_{2})$-maximizers) of $ g $ if the following conditions hold:
\begin{itemize}
	\item[i)] The pair $ (\textbf{u}_1, \textbf{v}_1) $ minimizes (maximizes) $ g $ in $ \mathcal{U}_0 $.
	\item[ii)] For $ j = 2, \ldots , d $, the pair $ (\textbf{u}_j, \textbf{v}_j) $ minimizes (maximizes) $ g $ in $ \mathcal{U}_0 $ under the constraints that $ \textbf{u}_j' \textbf{G}_{1k} \textbf{u}_k = 0 $ and $ \textbf{v}_j' \textbf{G}_{2k} \textbf{v}_k = 0 $ for all $ k = 1, \ldots , j - 1 $
\end{itemize}
While it is not explicit in the above notation, we also allow the matrices $\textbf{G}_{1k}, \textbf{G}_{2k}$, $k = 1, \ldots d - 1 $, to depend on the earlier optimizers. For example, the matrix $\textbf{G}_{11}$ might depend on the first stage optimizers $\textbf{u}_1$ and $\textbf{v}_1$.

Analogously, in the case of a pair of single-argument functions, $g_1: \mathbb{S}^{p-1} \rightarrow \mathbb{R}$, $g_2: \mathbb{S}^{q-1} \rightarrow \mathbb{R}$, and the collections of matrices, $\mathcal{G}_{1} = \{ \textbf{G}_{11}, \ldots , \textbf{G}_{1(d-1)} \} \in (\mathbb{R}^{p \times p})^{d-1}$ and $\mathcal{G}_{2} = \{ \textbf{G}_{21}, \ldots , \textbf{G}_{2(d-1)} \} \in (\mathbb{R}^{q \times q})^{d-1}$, we say that a collection of pairs of vectors, $ (\textbf{u}_1, \textbf{v}_1), \ldots , (\textbf{u}_d, \textbf{v}_d) \in \mathcal{U}_0 $, $d \leq \min\{ p , q \}$, is a sequence of $(\mathcal{G}_{1}, \mathcal{G}_{2})$-minimizers ($(\mathcal{G}_{1}, \mathcal{G}_{2})$-maximizers) of $ (g_1, g_2) $ if the following conditions hold:
\begin{itemize}
	\item[i)] The vector $ \textbf{u}_1 $ minimizes (maximizes) $ g_1 $ in $ \mathbb{S}^{p - 1} $ and the vector $ \textbf{v}_1 $ minimizes (maximizes) $ g_2 $ in $ \mathbb{S}^{q - 1} $.
	\item[ii)] For $ j = 2, \ldots , d $, the vector $ \textbf{u}_1 $ minimizes (maximizes) $ g_1 $ in $ \mathbb{S}^{p - 1} $ and the vector $ \textbf{v}_1 $ minimizes (maximizes) $ g_2 $ in $ \mathbb{S}^{q - 1} $, respectively, under the constraints that $ \textbf{u}_j' \textbf{G}_{1k} \textbf{u}_k = 0 $ and $ \textbf{v}_j' \textbf{G}_{2k} \textbf{v}_k = 0 $ for all $ k = 1, \ldots , j - 1 $.
\end{itemize} 

\section{Projection pursuit for matrix-valued data}\label{sec:method}

\subsection{Index for simultaneous estimation of the directions}

As described in the introduction, we propose projection indices for determining the directions $ \textbf{u}, \textbf{v} $ in the projection $ \textbf{u}' \textbf{X} \textbf{v} $ both simultaneously and one-by-one. Beginning with the former, recall that the kurtosis of a non-degenerate univariate random variable $ x $ (with finite fourth moment) is defined as
\[
\kappa_x := \frac{\mathrm{E} \left[ \{ x - \mathrm{E}(x) \}^4 \right] }{ \left\{ \mathrm{E} \left[ \{ x- \mathrm{E}(x) \}^2 \right] \right\}^2}.
\]
(some authors subtract 3 from $ \kappa_x $ to make the kurtosis of normal distribution equal zero, but this change plays no role in the context of maximizing/minimizing kurtosis). Hence, given a $ p \times q $ random matrix $ \textbf{X} $ (with finite fourth moments) and the projection direction $ (\textbf{u}, \textbf{v}) \in \mathcal{U}_0 $, the kurtosis of the projection $ \textbf{u}' \textbf{X} \textbf{v} $ is,
\[
\kappa_{\textbf{X}}(\textbf{u}, \textbf{v}) := \frac{\mathrm{E} \left( [ \textbf{u}' \{ \textbf{X} - \mathrm{E}(\textbf{X}) \} \textbf{v} ]^4 \right) }{ \left\{ \mathrm{E} \left( [ \textbf{u}' \{ \textbf{X} - \mathrm{E}(\textbf{X}) \} \textbf{v} ]^2 \right) \right\}^2}.
\]
The scale invariance of $\kappa_{\textbf{X}}$ guarantees that it is indeed sufficient to restrict its domain to the ``unit sphere'' $\mathcal{U}_0$. Note also that, for $\kappa_{\textbf{X}}$ to be well-defined in the whole of $\mathcal{U}_0$, it is necessary that the random variable $\textbf{u}' \{ \textbf{X} - \mathrm{E}(\textbf{X}) \} \textbf{v}$ is, for all $(\textbf{u}, \textbf{v}) \in \mathcal{U}_0$, not almost surely a constant. In Appendix \ref{sec:well_defined}, we discuss this condition more closely and derive, in two specific contexts, more easily verifiable forms for it.



\subsection{Index for separate estimation of the directions}

As an alternative to $ \kappa_{\textbf{X}} $, we provide an index that is a function of $ \textbf{u} $ only. Such an index can be useful when only the rows of $ \textbf{X} $ are of interest (consider, e.g., a situation where the columns of $ \textbf{X} $ denote different time points during which the row variables are measured, as in \cite{pfeiffer2012sufficient}, and assume that we are interested solely in the relationships between the row variables) or when, e.g., due to computational reasons, one wants to estimate $ \textbf{u} $ and $ \textbf{v} $ separately. To measure the interestingness of the ``partial'' projection $ \textbf{u}' \textbf{X} $, we use the famous Mardia's measure of multivariate kurtosis \cite{mardia1970measures}, defined for a $ q $-dimensional random vector $ \textbf{x} $ as
\[
\psi(\textbf{x}) := \mathrm{E}\left( [ \{ \textbf{x} -  \mathrm{E}(\textbf{x}) \}' \mathrm{Cov}(\textbf{x})^{-1} \{ \textbf{x} -  \mathrm{E}(\textbf{x}) \} ]^2  \right),
\]
Hence, given a $ p \times q $ random matrix $ \textbf{X} $ (with finite fourth moments) and the projection direction $ \textbf{u} \in \mathbb{S}^{p - 1} $ Mardia's kurtosis of the projection $ \textbf{u}' \textbf{X} $ is,
\begin{align}\label{eq:psi_definition}
\psi_{\textbf{X}}( \textbf{u} ) := \mathrm{E}\left[ \left\{ \textbf{u}' \tilde{\textbf{X}} \left[ \mathrm{E} \left( \tilde{\textbf{X}}' \textbf{u} \textbf{u}' \tilde{\textbf{X}} \right) \right]^{-1} \tilde{\textbf{X}}' \textbf{u} \right\}^2  \right],
\end{align}
where $ \tilde{\textbf{X}} = \textbf{X} - \mathrm{E}(\textbf{X}) $. Note that, to estimate the projection direction $ \textbf{v} $, the right-hand side analogue of $ \psi_{\textbf{X}} $ is naturally needed. However, as the roles of  $ \textbf{u} $ and $ \textbf{v} $ are fully symmetric in $ \textbf{u}' \textbf{X} \textbf{v} $ under the transposition of $\textbf{X}$, everything that we say about $ \textbf{u} $ applies equally to $ \textbf{v} $ (after transposition) and, thus, we will for the remainder of this section formulate our results on the mode-wise index $ \psi_{\textbf{X}} $ for the $ \textbf{u} $-side of the projection only. 
For $\psi_{\textbf{X}}$ to be well-defined in $\mathbb{S}^{p - 1}$, the matrix $\mathrm{E} [ \{ \textbf{X} - \mathrm{E}(\textbf{X}) \}' \textbf{u} \textbf{u}' \{ \textbf{X} - \mathrm{E}(\textbf{X}) \} ]$ needs to be invertible for all $\textbf{u} \in \mathbb{S}^{p-1}$ and this condition (which is in fact equivalent to the corresponding condition for $\kappa_\textbf{X}$) is further discussed in Appendix \ref{sec:well_defined}. Finally note that both $ \kappa_{\textbf{X}}(\textbf{u}, \textbf{v}) $ and $ \psi_{\textbf{X}}(\textbf{u}) $ are true generalizations of the classical kurtosis in the sense that if the column dimension is degenerate $( q = 1 )$, then both indices are equal to the kurtosis of the univariate projection $ \textbf{u}' \textbf{X}$. 

\subsection{General results}
The main focus of this work is to study the use of $ \kappa_{\textbf{X}} $ and $ \psi_{\textbf{X}} $ as projection indices in the projection pursuit of matrix-valued data. As our first result, we establish that the optimization of the two indices is indeed a well-defined procedure in the sense that a minimizing/maximizing direction always exists. Note that we include both minimization and maximization as, analogous to standard projection pursuit with kurtosis, the choice of the optimization direction (minimization/maximization) affects what kind of structures we can find. The usual heuristic is that minimization of kurtosis finds clusters of roughly equal proportions and maximization finds outliers,  see Theorems \ref{theo:matrix_lda_optimal_higher_rank} and \ref{theo:matrix_lda_optimal_mode_wise} later in Section \ref{sec:lda}.

\begin{lemma}\label{lem:existence}
	Let $ \textbf{X} $ be a $ p \times q $ random matrix having finite fourth moments and assume that $\mathrm{E} \left( [ \textbf{u}' \{ \textbf{X} - \mathrm{E}(\textbf{X}) \} \textbf{v} ]^2 \right) > 0$ for all $(\textbf{u}, \textbf{v}) \in \mathcal{U}_0$. Then,
	\begin{itemize}
		\item[i)] there exist both a pair $ (\textbf{u}_0, \textbf{v}_0) $ that minimizes $ \kappa_{\textbf{X}}(\textbf{u}, \textbf{v}) $ and a pair $ (\textbf{u}_1, \textbf{v}_1) $ that maximizes $ \kappa_{\textbf{X}}(\textbf{u}, \textbf{v}) $ in $ \mathcal{U}_0 $,
		\item[ii)] there exist both a direction $ \textbf{u}_0 $ that minimizes $ \psi_{\textbf{X}}(\textbf{u}) $ and a direction $ \textbf{u}_1 $ that maximizes $ \psi_{\textbf{X}}(\textbf{u}) $ in $ \mathbb{S}^{p - 1} $.
	\end{itemize}
\end{lemma}

Having established existence in Lemma \ref{lem:existence}, we note that the uniqueness of an optimizer is unobtainable in the case of general $ \textbf{X} $. This follows instantly by considering any $ \textbf{X} $ which is \textit{spherical} in the sense that $ \textbf{U} \textbf{X} \textbf{V}' \sim \textbf{X} $ for all orthogonal matrices $ \textbf{U} \in \mathbb{R}^{p \times p} $ and $ \textbf{V} \in \mathbb{R}^{q \times q} $, and noting that for such $ \textbf{X} $ all projections $\textbf{u}' \textbf{X} \textbf{v}$ have identical distributions. See \cite{gupta2018matrix} for examples of spherical matrix distributions. However, uniqueness of the optimizer can naturally be established for some particular families of $ \textbf{X} $, see Section \ref{sec:lda} for an example. Finally, alternative forms for the condition required in Lemma \ref{lem:existence} are discussed in Appendix \ref{sec:well_defined}.


In the next section we will investigate the theoretical properties of the projection indices $ \kappa_{\textbf{X}} $ and $ \psi_{\textbf{X}} $ under the matrix normal distribution mixture model. 


\section{Group separation with matrix projection pursuit}\label{sec:lda}

\subsection{Optimal projections under matrix normal mixture}

Throughout Section \ref{sec:lda}, we assume that the $ p \times q $ random matrix $ \textbf{X} $ obeys a mixture of matrix normal distributions,
\begin{align}\label{eq:matrix_normal_mixture}
\textbf{X} \sim \alpha_1 \mathcal{N}_{p \times q}(\textbf{T}_1, \textbf{A}, \textbf{B}) +  \alpha_2 \mathcal{N}_{p \times q}(\textbf{T}_2, \textbf{A}, \textbf{B}),
\end{align}
where $ \textbf{A} \in \mathbb{R}^{p \times p} $ and $ \textbf{B} \in \mathbb{R}^{q \times q} $ are positive definite, $ \alpha_1 + \alpha_2 = 1 $ and the mean matrices $ \textbf{T}_1 \in \mathbb{R}^{p \times q} $, $ \textbf{T}_2 \in \mathbb{R}^{p \times q} $ are not equal. The matrix normal distribution $ \mathcal{N}_{p \times q}(\textbf{T}, \textbf{A}, \textbf{B}) $ is defined to be the distribution of the random matrix $ \textbf{T} + \textbf{A}^{1/2} \textbf{Z} \textbf{B}^{1/2} $, where the elements of the $ p \times q $ random matrix $ \textbf{Z} $ are i.i.d. standard normal, see \cite{gupta2018matrix}. As discussed in the introduction, this is a typical context to apply projection pursuit (for vectorial data), with the objective of finding a low-dimensional projection that separates the two components of the mixture. To quantify our target, we begin by deriving an expression for the optimal projection direction (in the sense of LDA) for separating the components of the mixture \eqref{eq:matrix_normal_mixture}. 
The following lemma gives closed-form expression for the optimal linear discriminant projection under model~\eqref{eq:matrix_normal_mixture}.

\begin{lemma}\label{lem:matrix_lda_optimal}
	Under model \eqref{eq:matrix_normal_mixture}, the optimal projection for separating the parts of the mixture in the sense of LDA is
	\[
	\langle \textbf{W}_{\mathrm{LDA}}, \textbf{X} \rangle,
	\]
	where the projection direction is
	\[
	\textbf{W}_{\mathrm{LDA}} := \textbf{A}^{-1} (\textbf{T}_2 - \textbf{T}_1) \textbf{B}^{-1}.
	\]	
\end{lemma}
Inspection of the proof of Lemma~\ref{lem:matrix_lda_optimal} also reveals that the optimal Bayes classifier depends on the data $\textbf{X}$ only through this particular projection $\langle \textbf{W}_{\mathrm{LDA}}, \textbf{X} \rangle$.

We still establish some additional notation. Turns out that in each case, the correct optimization direction (minimization/maximization) is fully determined by the value of the mixing proportion $ \alpha_1 $. If $ \alpha_1 \in \mathcal{A}_{\mathrm{min}} := \{ \alpha \in (0, 1) \mid | \alpha - 1/2 | < 1/\sqrt{12} \} $, then one should minimize, and if $ \alpha_1 \in \mathcal{A}_{\mathrm{max}} := \{ \alpha \in (0, 1) \mid | \alpha - 1/2 | > 1/\sqrt{12} \} $, then one should maximize the corresponding objective function. In the edge case, $ \alpha_1 \in \mathcal{A}_{\mathrm{0}} := \{ \alpha \in (0, 1) \mid | \alpha - 1/2 | = 1/\sqrt{12} \} $, the kurtosis of the projection does not depend on the projection direction and in this case projection pursuit (with kurtosis as the index) carries no information on the group separation. Note that what makes the values $ \alpha_0 = 1/2 \pm 1/\sqrt{12} $ special, is that any univariate normal mixture $ \alpha_0 \mathcal{N}(-\mu, \sigma^2) + (1 - \alpha_0) \mathcal{N}(\mu, \sigma^2) $ has kurtosis equal to that of standard normal distribution, making it indistinguishable from noise w.r.t. kurtosis.
The dependency of the form of optimization on $ \alpha_1 $ naturally means that, in order to choose the correct optimization direction in practice, one has to know the (generally unknown) mixing proportion. However, this is not an issue in practice as one can simply estimate both a minimizing and maximizing  solution and use, e.g., the scatter plot between the respective projections to identify the optimal direction. More discussion on whether to maximize or to minimize is given in Section~\ref{sec:discussion} and an alternative, fully blind approach is described later in this section in Corollary~\ref{cor:alternative_1}. Finally, as the borderline case $ \alpha_0 $, while arguably a rare occurrence, can sometimes be a nuisance in practice, we discuss in Section \ref{sec:discussion} a way to get around it.

\subsection{Optimization of $ \kappa_{\textbf{X}} $}\label{subsec:kappa}

We begin by considering the optimization of the projection index $ \kappa_{\textbf{X}} $ involving both projection directions $ \textbf{u} $ and $ \textbf{v} $ simultaneously. Recall that the dual projection $ \textbf{u}' \textbf{X} \textbf{v} $ can be written as a rank-1 projection $ \langle \textbf{u} \textbf{v}', \textbf{X} \rangle $ and that, based on Lemma \ref{lem:matrix_lda_optimal}, the optimal projection has the $ \mathrm{rank}(	\textbf{W}_{\mathrm{LDA}} ) = \mathrm{rank}(\textbf{T}_2 - \textbf{T}_1) =: d $. This implies that, in order to recover $ \textbf{W}_{\mathrm{LDA}} $ through matrix projection pursuit, we need to extract at least $ d $ pairs of directions just to account for the degrees of freedom.  

The next result shows that extracting $ d $ pairs is, besides necessary, also sufficient for the projection index $ \kappa_{\textbf{X}} $ to reconstruct $ \textbf{W}_{\mathrm{LDA}} $. Unlike in standard projection pursuit, we do not enforce orthogonality of the successive optimizers in the usual sense, but w.r.t. a set of specific matrices that depend on the previous optimizers (inspection of the proof of the following theorem reveals that requiring regular orthogonality would not allow the estimation of $ \textbf{W}_{\mathrm{LDA}} $). Namely, for a fixed collection $ (\textbf{u}_1, \textbf{v}_1), \ldots , (\textbf{u}_d, \textbf{v}_d) \in \mathcal{U}_0 $, we let $\mathcal{G}_{1,\textbf{X}} = \{ \textbf{G}_{11,\textbf{X}}, \ldots , \textbf{G}_{1(d-1),\textbf{X}} \}$ and $\mathcal{G}_{2,\textbf{X}} = \{ \textbf{G}_{21,\textbf{X}}, \ldots , \textbf{G}_{2(d-1),\textbf{X}} \} $, where
\[ 
	\textbf{G}_{1k,\textbf{X}} := \mathrm{E} \left[ \{ \textbf{X} - \mathrm{E}(\textbf{X}) \} \textbf{v}_k \textbf{v}_k'  \{ \textbf{X} - \mathrm{E}(\textbf{X}) \}' \right],
	\]
	and
	\[ 
	\textbf{G}_{2k,\textbf{X}} := \mathrm{E} \left[ \{ \textbf{X} - \mathrm{E}(\textbf{X})' \} \textbf{u}_k \textbf{u}_k'  \{ \textbf{X} - \mathrm{E}(\textbf{X}) \} \right],
	\]
for $ k = 1, \ldots , d-1$ (recall Section \ref{sec:notation} for the definition of a sequence of optimizers under the orthogonality constraints given by $\mathcal{G}_{1,\textbf{X}}$ and $\mathcal{G}_{2,\textbf{X}}$). Additionally, we denote
\[ 
\lambda_j := \sqrt{ \max \left\{ \frac{\theta_j}{1 - \alpha_1 \alpha_2 \theta_j}, 0 \right\} }, \quad \mbox{where} \quad \theta_j := \sqrt{  \max \left\{ \frac{\kappa_{\textbf{X}}(\textbf{u}_j, \textbf{v}_j) - 3}{\alpha_1 \alpha_2 (1 - 6 \alpha_1 \alpha_2)}, 0 \right\} }.
\]
Note that the thresholding in $\lambda_j$ and $\theta_j$ is unnecessary on the population level (the quantities involved are non-negative under the model). However, we include it to make sure that the square roots are well-defined also in the sample version of the method.

\begin{theorem}\label{theo:matrix_lda_optimal_higher_rank}
	Assume that model \eqref{eq:matrix_normal_mixture} holds. Then,
	\begin{itemize}
		\item[i)] if $ \alpha_1 \in \mathcal{A}_{\mathrm{min}}\setminus\{\frac{1}{2}\} $, then any sequence of $(\mathcal{G}_{1,\textbf{X}}, \mathcal{G}_{2,\textbf{X}})$-minimizers of $ \kappa_{\textbf{X}}$ satisfies
		\[
		\sum_{j=1}^d \frac{s_j \lambda_j \sqrt{1 + \alpha_1 \alpha_2 \lambda_j^2} }{\sqrt{\mathrm{E} ( [ \textbf{u}_j' \{ \textbf{X} - \mathrm{E}(\textbf{X})  \} \textbf{v}_j  ]^2 )}} \textbf{u}_j \textbf{v}_j'  = \textbf{W}_{\mathrm{LDA}},
		\]
		where $ s_1, \ldots, s_d \in \{ -1, 1 \}$ are the signs of the quantities
		\[ 
		(\alpha_1 - \alpha_2)^{-1} \mathrm{E} ( [ \textbf{u}_j' \{ \textbf{X} - \mathrm{E}(\textbf{X})  \} \textbf{v}_j  ]^3 ).
		\]
		\item[ii)] if $ \alpha_1 \in \mathcal{A}_{0} $, then $ \kappa_{\textbf{X}}(\textbf{u}, \textbf{v}) = 3 $, regardless of $ \textbf{u}, \textbf{v} $.
		\item[iii)] if $  \alpha_1 \in \mathcal{A}_{\mathrm{max}} $, then any sequence of $(\mathcal{G}_{1,\textbf{X}}, \mathcal{G}_{2,\textbf{X}})$-maximizers of $ \kappa_{\textbf{X}} $ satisfies
		\[
		\sum_{j=1}^d  \frac{s_j \lambda_j \sqrt{1 + \alpha_1 \alpha_2 \lambda_j^2} }{\sqrt{\mathrm{E} ( [ \textbf{u}_j' \{ \textbf{X} - \mathrm{E}(\textbf{X})  \} \textbf{v}_j  ]^2 )}} \textbf{u}_j \textbf{v}_j'  = \textbf{W}_{\mathrm{LDA}},
		\]
		where $ s_1, \ldots, s_d \in \{ -1, 1 \}$ are the signs of the quantities
		\[ 
		(\alpha_1 - \alpha_2)^{-1} \mathrm{E} ( [ \textbf{u}_j' \{ \textbf{X} - \mathrm{E}(\textbf{X})  \} \textbf{v}_j  ]^3 ).
		\]
	\end{itemize}
\end{theorem}

\begin{figure}[h]
    \centering
    \includegraphics[width=1\linewidth]{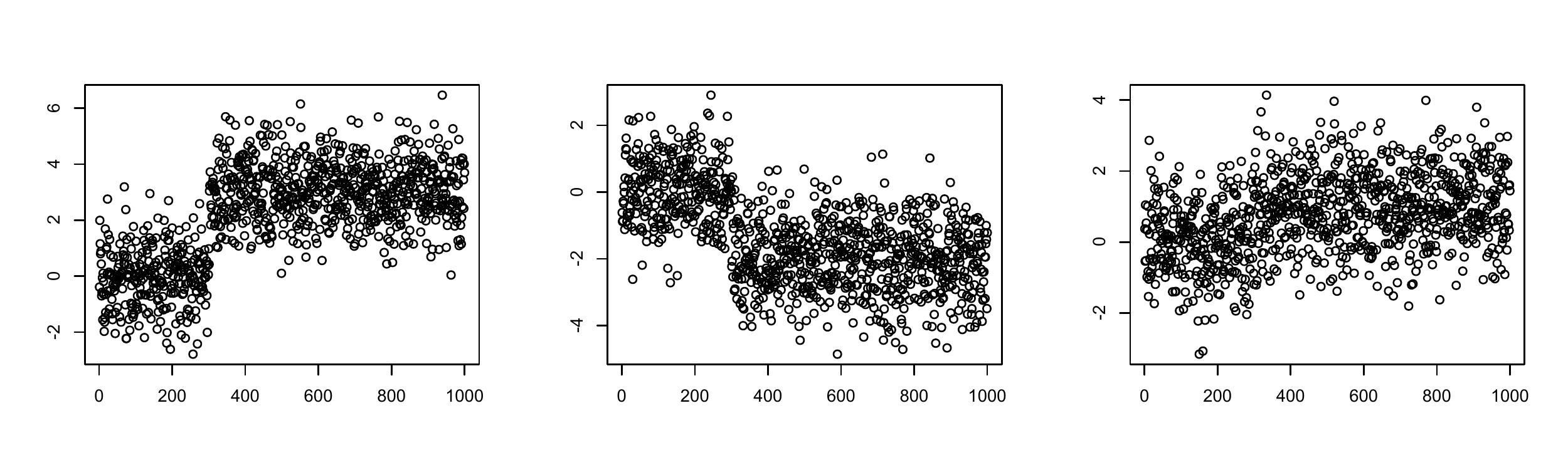}
    \caption{From left to right, the sub-plots show the scatter plots of $\textbf{u}_1' \textbf{X}_i \textbf{v}_1$, $\textbf{u}_2' \textbf{X}_i \textbf{v}_2$ and $\textbf{u}_3' \textbf{X}_i \textbf{v}_3$ versus the index $i = 1, \ldots , n$ in a sample of $n=1000$ observations from the model~\eqref{eq:matrix_normal_mixture} with $\alpha_1=0.3$, $\textbf{A}=\textbf{B}=\textbf{I}_3$ and $\textbf T_1=\textbf{0}$, $\textbf T_2=\mathrm{diag}(4,1,3)$. The first 300 observations correspond to the first group. No sign-correction has been applied to the projections and, due to this, if, e.g., the first two sub-plots would be superimposed the separation information contained in them would be partially cancelled out.}
    \label{fig:fig_1}
\end{figure}

Theorem~\ref{theo:matrix_lda_optimal_higher_rank} essentially says that the sequential optimization of $ \kappa_{\textbf{X}} $ allows the reconstruction of $ \textbf{W}_{\mathrm{LDA}} $, one rank-1 block at a time. Moreover, while the decomposition in Theorem \ref{theo:matrix_lda_optimal_higher_rank} is not the SVD of the matrix $ \textbf{W}_{\mathrm{LDA}} $ (since the vectors $ \textbf{u}_1, \ldots, \textbf{u}_d $ are not orthogonal in the usual sense, and similarly for the $ \textbf{v}_j $), inspecting the proof of the result reveals that the matrix 
\[ 
\textbf{A}^{1/2} \left\{ \sum_{j=1}^d \frac{s_j \lambda_j \sqrt{1 + \alpha_1 \alpha_2 \lambda_j^2} }{\sqrt{\mathrm{E} ( [ \textbf{u}_j' \{ \textbf{X} - \mathrm{E}(\textbf{X})  \} \textbf{v}_j  ]^2 )}} \textbf{u}_j \textbf{v}_j'  \right\} \textbf{B}^{1/2}
\]
is in fact the SVD of $ \textbf{A}^{1/2} \textbf{W}_{\mathrm{LDA}} \textbf{B}^{1/2} $, i.e., the matrix giving the optimal projection for the ``standardized'' observation $ \textbf{A}^{-1/2} \textbf{X} \textbf{B}^{-1/2} $. This gives us an interpretation for the result of Theorem \ref{theo:matrix_lda_optimal_higher_rank}; by the Eckart-Young Theorem, by optimizing $ \kappa_{\textbf{X}} $ once, we recover (in the previous sense) the best rank-1 approximation of the optimal projection direction, by optimizing it twice, we recover the best rank-2 approximation to the optimal projection direction and so on. In this spirit, the number of extracted directions can be seen as a tuning parameter giving a trade-off between lesser computational burden (smaller number of extracted directions) and increased accuracy (larger number of extracted directions), with the guarantee that the approximation of the optimal direction is always the best possible for any given rank.

Let us next demystify the role of the signs $s_1, \dots, s_d$ in Theorem \ref{theo:matrix_lda_optimal_higher_rank}. Intuitively, their role is to guarantee that the signs of the successive optimizers match in the sense that the same group is always projected onto the same side of the real line. In more detail, without loss of generality, let $\mathrm{E}(\textbf{X}) = \textbf{0}$ and consider the first optimizing pair $(\textbf{u}_1, \textbf{v}_1)$, denoting $\mu_k := \textbf{u}_1' \textbf{T}_k \textbf{v}_1$, $k = 1,2$, and $\sigma^2 := \mathrm{Var}(\textbf{u}_1'\textbf{X} \textbf{v}_1)$. Then, due to the zero-mean assumption, $\alpha_1 \mu_1+\alpha_2 \mu_2=0$ and
$$
\mathbb{E}\{ (\textbf{u}_1'\textbf{X} \textbf{v}_1 )^3 \} = \alpha_1 \mu_1^3+\alpha_2 \mu_2^3+3\sigma^2(\alpha_1 \mu_1+\alpha_2 \mu_2)=-\alpha_2 \mu_2 \mu_1^2+\alpha_2\mu_2^3=\alpha_2\mu_2(\mu_2-\mu_1)(\mu_2+\mu_1).
$$
Furthermore, assuming that $\alpha_1 \neq \alpha_2$, we have $\mu_1+\mu_2=(\alpha_1-\alpha_2)\mu_2/\alpha_1$, implying that 
$$
s_1=\mathrm{sign}[(\alpha_1 - \alpha_2)^{-1} \mathbb{E}\{ (\textbf{u}_1'\textbf{X} \textbf{v}_1 )^3 \}] = \mathrm{sign}( \mu_2 - \mu_1 ).
$$
This implies that the sign $s_1$ (and, analogously, the signs $s_2, \ldots , s_d$) is set such that the sign of the group whose projection is further away from the origin is taken to be positive. Fixing the signs in this way lets us avoid situations where two projections $\textbf{u}_j'\textbf{X}\textbf{v}_j$ and $\textbf{u}_k'\textbf{X}\textbf{v}_k$ would have opposite signs and would ``nullify'' each other, see the illustration in Figure~\ref{fig:fig_1}. The previous also reveals why the technique is unable to reconstruct $ \textbf{W}_{\mathrm{LDA}} $ when $\alpha_1 = \alpha_2 = 1/2$. Indeed, in the case of balanced groups, the two means are equally far away from the origin and the previous criterion does not let us identify the groups. In Section~\ref{sec:discussion} we discuss ways of working around this issue in practice.


We next illustrate an interesting property of $\kappa_{\bf X}$ under the normal mixture model \eqref{eq:matrix_normal_mixture}. Namely, Lemma~\ref{lemma:lemma3} shows that every member of any sequence of $(\mathcal{G}_{1,\textbf{X}},\mathcal{G}_{2,\textbf{X}})$-optimizers of $\kappa_{\bf X}$ is in fact a stationary point of the unconstrained objective function $\kappa_{\bf X}$. The result is given from the viewpoint of maximization but applies equally under minimization. 

\begin{lemma}\label{lemma:lemma3}
Assume that model \eqref{eq:matrix_normal_mixture} holds and that $ \alpha_1 \in \mathcal{A}_{0} $. Let $(\textbf{u}_{1}, \textbf{v}_{1}), \ldots , (\textbf{u}_{d}, \textbf{v}_{d}) \in \mathcal{U}_0$ be any sequence of $(\mathcal{G}_{1,\textbf{X}}, \mathcal{G}_{2,\textbf{X}})$-maximizers of $\kappa_{\bf X}$. Then, for all $k = 1, \ldots , d$, we have $\nabla\kappa_{\bf X}(\textbf{u}_k,\textbf{v}_k)=\textbf{0}$.
\end{lemma}

Lemma \ref{lemma:lemma3} essentially says that $\kappa_{\bf X}$ has several local optima/saddle points, implying that its optimization is likely to be difficult in practice, see the later Section \ref{sec:algorithm} on algorithms for more discussion on this.

We close the section by noting that an alternative estimator of $\textbf{W}_{\mathrm{LDA}}$ for which maximization is always sufficient can be obtained by considering excess kurtosis instead of kurtosis. Indeed, this is also what was done in the vectorial context by \cite{radojicic2021}. However, as knowledge on the mixing proportions $\alpha_1, \alpha_2$ is anyway required to reconstruct $\textbf{W}_{\mathrm{LDA}}$ (in the form of the signs $s_1, \ldots , s_d$), we have chosen not to take this approach in the current work. The proof of Corollary \ref{cor:alternative_1} follows by straightforward adaptation from that of Theorem \ref{theo:matrix_lda_optimal_higher_rank} and is thus omitted.

\begin{corollary}\label{cor:alternative_1}
   Assume that model \eqref{eq:matrix_normal_mixture} holds. If $ \alpha_1 \in (0, 1)\setminus\{1/2 - 1/\sqrt{12}, \frac{1}{2}, 1/2 + 1/\sqrt{12}\} $, then any sequence of $(\mathcal{G}_{1,\textbf{X}}, \mathcal{G}_{2,\textbf{X}})$-maximizers of the function $ (\textbf{u}, \textbf{v}) \mapsto \{ \kappa_{\textbf{X}}(\textbf{u}, \textbf{v}) - 3 \}^2$ satisfies
		\[
		\sum_{j=1}^d \frac{s_j \lambda_j \sqrt{1 + \alpha_1 \alpha_2 \lambda_j^2} }{\sqrt{\mathrm{E} ( [ \textbf{u}_j' \{ \textbf{X} - \mathrm{E}(\textbf{X})  \} \textbf{v}_j  ]^2 )}} \textbf{u}_j \textbf{v}_j'  = \textbf{W}_{\mathrm{LDA}},
		\]
		where $ s_1, \ldots, s_d \in \{ -1, 1 \}$ are as in Theorem \ref{theo:matrix_lda_optimal_higher_rank}.
\end{corollary}

\subsection{Optimization of $ \psi_{\textbf{X}} $}

We next consider the optimization of the projection index $ \psi_{\textbf{X}} $ in \eqref{eq:psi_definition} involving only the direction $ \textbf{u} $, along with its counterpart,
\[ 
\psi_{\textbf{X}'}( \textbf{v} ) := \mathrm{E}\left[ \left\{ \textbf{v}' \tilde{\textbf{X}}' \left[ \mathrm{E} \left( \tilde{\textbf{X}} \textbf{v} \textbf{v}' \tilde{\textbf{X}}' \right) \right]^{-1} \tilde{\textbf{X}} \textbf{v} \right\}^2  \right],
\]
depending only on $ \textbf{v} $. The next theorem proves a result analogous to Theorem \ref{theo:matrix_lda_optimal_higher_rank} for $ \psi_{\textbf{X}} $ and $ \psi_{\textbf{X}'} $, with the crucial difference that the latter can recover the direction of optimal separation under repeated optimization only if the non-zero singular values of $\textbf{A}^{-1/2} (\textbf{T}_2 - \textbf{T}_1) \textbf{B}^{-1/2}$ are distinct. The intuition behind this assumption will be discussed after the statement of the result.

In the following, for a fixed collection $ (\textbf{u}_1, \textbf{v}_1), \ldots , (\textbf{u}_d, \textbf{v}_d) \in \mathcal{U}_0 $, we define the sets of matrices $\mathcal{G}_{1,\textbf{X}}$ and  $\mathcal{G}_{2,\textbf{X}} $ similarly as prior to Theorem \ref{theo:matrix_lda_optimal_higher_rank}.
Additionally, we define
\[ 
\lambda_j := \sqrt{ \max \left\{ \frac{\theta_j}{1 - \alpha_1 \alpha_2 \theta_j}, 0 \right\} }, \quad \mbox{where } \theta_j := \sqrt{ \max \left\{ \frac{\psi_{\textbf{X}}(\textbf{u}_j) - q(q + 2)}{\alpha_1 \alpha_2 (1 - 6 \alpha_1 \alpha_2)}, 0 \right\} }.
\]
That is, $\lambda_j$ and $\theta_j$ are as in Section \ref{subsec:kappa}, apart from changing the constant~3 to $ q(q + 2) $.

\begin{theorem}\label{theo:matrix_lda_optimal_mode_wise}
	Assume that model \eqref{eq:matrix_normal_mixture} holds and that the non-zero singular values of $\textbf{A}^{-1/2} (\textbf{T}_2 - \textbf{T}_1) \textbf{B}^{-1/2}$ are distinct. Then,
	\begin{itemize}
		\item[i)] if $ \alpha_1 \in \mathcal{A}_{\mathrm{min}}\setminus\{\frac{1}{2}\} $, then any sequence of $(\mathcal{G}_{1,\textbf{X}}, \mathcal{G}_{2,\textbf{X}})$-minimizers of $ (\psi_{\textbf{X}}, \psi_{\textbf{X}'}) $ satisfies
		\[
		\sum_{j=1}^d \frac{s_j \lambda_j \sqrt{1 + \alpha_1 \alpha_2 \lambda_j^2} }{\sqrt{\mathrm{E} ( [ \textbf{u}_j' \{ \textbf{X} - \mathrm{E}(\textbf{X})  \} \textbf{v}_j  ]^2 )}} \textbf{u}_j \textbf{v}_j'  = \textbf{W}_{\mathrm{LDA}},
		\]
		where $ s_1, \ldots, s_d \in \{ -1, 1 \}$ are the signs of the quantities
		\[ 
		(\alpha_1 - \alpha_2)^{-1} \mathrm{E} ( [ \textbf{u}_j' \{ \textbf{X} - \mathrm{E}(\textbf{X})  \} \textbf{v}_j  ]^3 ).
		\]
		\item[ii)] if $ \alpha_1 \in \mathcal{A}_{0} $, then $ \psi_{\textbf{X}}(\textbf{u}) = q(q + 2) $ and $ \psi_{\textbf{X}}(\textbf{v}) = p(p + 2) $, regardless of $ \textbf{u}, \textbf{v} $.
		\item[iii)] if $  \alpha_1 \in \mathcal{A}_{\mathrm{max}} $, then any sequence of $(\mathcal{G}_{1,\textbf{X}}, \mathcal{G}_{2,\textbf{X}})$-maximizers of $ (\psi_{\textbf{X}}, \psi_{\textbf{X}'}) $ satisfies
		\[
		\sum_{j=1}^d  \frac{s_j \lambda_j \sqrt{1 + \alpha_1 \alpha_2 \lambda_j^2} }{\sqrt{\mathrm{E} ( [ \textbf{u}_j' \{ \textbf{X} - \mathrm{E}(\textbf{X})  \} \textbf{v}_j  ]^2 )}} \textbf{u}_j \textbf{v}_j'  = \textbf{W}_{\mathrm{LDA}},
		\]
		where $ s_1, \ldots, s_d \in \{ -1, 1 \}$ are the signs of the quantities
		\[ 
		(\alpha_1 - \alpha_2)^{-1} \mathrm{E} ( [ \textbf{u}_j' \{ \textbf{X} - \mathrm{E}(\textbf{X})  \} \textbf{v}_j  ]^3 ).
		\]
	\end{itemize}
\end{theorem}

The conclusion of Theorem \ref{theo:matrix_lda_optimal_mode_wise} is intuitively rather unsuprising: As $ \psi_{\textbf{X}} $ and $\psi_{\textbf{X}'} $ each ``see'' only one side of the model, they fail to recover $\textbf{W}_{\mathrm{LDA}}$, an object depending on both sides of the model, unless it carries a simple enough structure. On a more technical level, as discussed after Theorem \ref{theo:matrix_lda_optimal_higher_rank}, the reconstruction of $\textbf{W}_{\mathrm{LDA}}$ essentially boils down to the estimation of the singular vector pairs of the matrix $ \textbf{R} := \textbf{A}^{1/2} \textbf{W}_{\mathrm{LDA}} \textbf{B}^{1/2} $. Now, $\kappa_\textbf{X}$ succeeds in this by always extracting both $\textbf{u}$ and $\textbf{v}$ at the same time, forming complete singular pairs after each repeated optimization (pairs $(\textbf{u}_0, \textbf{v}_0)$ satisfying $\textbf{R} \textbf{v}_0 \propto \textbf{u}_0$). Whereas, inspection of the proof of Theorem~\ref{theo:matrix_lda_optimal_mode_wise} reveals that $ \psi_{\textbf{X}} $ and $\psi_{\textbf{X}'} $ extract, respectively, eigenvectors of the matrices $\textbf{R}\textbf{R}'$ and $\textbf{R}' \textbf{R}$, which are guaranteed to form a pair of singular vectors of $\textbf{R}$ only if its singular spaces are one-dimensional, i.e., its singular values are distinct.

As with the simultaneous index $\kappa_{\textbf{X}}$, also the mode-wise indices can be made non-dependent on the optimization direction by considering ``excess kurtosis'' instead of kurtosis (the involved quantity is not true excess kurtosis as it uses dimension-dependent constants in place of 3). The resulting Corollary \ref{cor:alternative_2} follows straightforwardly from Theorem~\ref{theo:matrix_lda_optimal_mode_wise} and its proof is omitted.

\begin{corollary}\label{cor:alternative_2}
   Assume that model \eqref{eq:matrix_normal_mixture} holds and that the non-zero singular values of $\textbf{A}^{-1/2} (\textbf{T}_2 - \textbf{T}_1) \textbf{B}^{-1/2}$ are distinct. If $ \alpha_1 \in (0, 1)\setminus\{1/2 - 1/\sqrt{12}, \frac{1}{2}, 1/2 + 1/\sqrt{12}\} $, then any sequence of $(\mathcal{G}_{1,\textbf{X}}, \mathcal{G}_{2,\textbf{X}})$-minimizers of the pair of functions $ \textbf{u} \mapsto \{ \psi_{\textbf{X}}(\textbf{u}) - q(q + 2) \}^2$, $ \textbf{v} \mapsto \{ \psi_{\textbf{X}'}(\textbf{v}) - p(p + 2) \}^2$ satisfies
		\[
		\sum_{j=1}^d \frac{s_j \lambda_j \sqrt{1 + \alpha_1 \alpha_2 \lambda_j^2} }{\sqrt{\mathrm{E} ( [ \textbf{u}_j' \{ \textbf{X} - \mathrm{E}(\textbf{X})  \} \textbf{v}_j  ]^2 )}} \textbf{u}_j \textbf{v}_j'  = \textbf{W}_{\mathrm{LDA}},
		\]
		where $ s_1, \ldots, s_d \in \{ -1, 1 \}$ are as in Theorem \ref{theo:matrix_lda_optimal_mode_wise}.
\end{corollary}

\subsection{Comparison to second-order projection methods}

Let $ \kappa_{2, \textbf{X}} $ denote the projection index used in MPCA and $ (\psi_{2, \textbf{X}}, \psi_{2, \textbf{X}'}) $ denote the indices used in (2D)$ ^2 $PCA (when restricting both methods to a single projection-pair), that is,
\[
\kappa_{2, \textbf{X}}(\textbf{u}, \textbf{v}) := \mathrm{E} \left( [ \textbf{u}' \{ \textbf{X} - \mathrm{E}(\textbf{X}) \} \textbf{v} ]^2 \right)
\]
and
\[
\psi_{2, \textbf{X}}(\textbf{u}) := \mathrm{E} \left[ \textbf{u}' \{ \textbf{X} - \mathrm{E}(\textbf{X}) \}  \{ \textbf{X} - \mathrm{E}(\textbf{X}) \}' \textbf{u} \right] \quad \mbox{and} \quad \psi_{2, \textbf{X}'}(\textbf{v}) := \mathrm{E} \left[ \textbf{v}' \{ \textbf{X} - \mathrm{E}(\textbf{X}) \}'  \{ \textbf{X} - \mathrm{E}(\textbf{X}) \} \textbf{v} \right].
\]

The next result shows that, even in the simplest case of rank-1 difference between the group means, $ \kappa_{2, \textbf{X}} $ and $ (\psi_{2, \textbf{X}}, \psi_{2, \textbf{X}'}) $ are able to recover the optimally separating direction under \eqref{eq:matrix_normal_mixture} only when very specific conditions are met. Thus, the leading projections extracted by the second-order methods MPCA and (2D)$ ^2 $PCA may fail to identify the cluster structure and it is more preferable to resort to our proposed fourth moment-based projection pursuit in group separation scenarios. In Theorems \ref{theo:matrix_lda_mpca_2d2pca} and \ref{theo:matrix_lda_mpca} we denote by $\textbf{u}_{\mathrm{LDA}} := \textbf{A}^{-1} \textbf{a}/ \| \textbf{A}^{-1} \textbf{a} \|$ and $\textbf{v}_{\mathrm{LDA}} := \textbf{B}^{-1} \textbf{v}/ \| \textbf{B}^{-1} \textbf{v} \|$ the optimal projection directions (up to scale) under the rank-1 assumption.  The proof of Theorem \ref{theo:matrix_lda_mpca_2d2pca} is omitted as it is exactly analogous to Lemma A.1 in \cite{radojicic2021}. 


\begin{theorem}\label{theo:matrix_lda_mpca_2d2pca}
Assume that model \eqref{eq:matrix_normal_mixture} holds such that $ \textbf{T}_2 - \textbf{T}_1 = \textbf{a} \textbf{b}' $ for some $ \textbf{a} \in \mathbb{R}^p $, $ \textbf{b} \in \mathbb{R}^q $. Then the following two are equivalent:
\begin{itemize}
	\item[i)] The unique unit length maximizers of $\psi_{2,\textbf X}$ are $\pm \textbf{u}_{\mathrm{LDA}}$.
	\item[ii)] The vector $\textbf{a}$ is an eigenvector of $\textbf{A}$ and, letting $\lambda$ stand for the corresponding eigenvalue, the second-to-largest eigenvalue $\phi_2$ of $\mathrm{tr}(\textbf{B}) \textbf{A} + \alpha_1 \alpha_2 \| \textbf{b} \|^2 \textbf{a} \textbf{a}'$ satisfies,
	\begin{align*}
	    \phi_2 < \lambda \{ \mathrm{tr}(\textbf{B}) + \alpha_1 \alpha_2 \textbf{a}' \textbf{A}^{-1} \textbf{a} \| \textbf{b} \|^2 \}.
	\end{align*}
\end{itemize}
\end{theorem}

\begin{theorem}\label{theo:matrix_lda_mpca}
Assume that model \eqref{eq:matrix_normal_mixture} holds such that $ \textbf{T}_2 - \textbf{T}_1 = \textbf{a} \textbf{b}' $ for some $ \textbf{a} \in \mathbb{R}^p $, $ \textbf{b} \in \mathbb{R}^q $. Then,
\begin{itemize}
	\item[i)]  For $(\pm \textbf{u}_{\mathrm{LDA}}, \pm \textbf{v}_{\mathrm{LDA}})$ to be the unique maximizers of $\kappa_{2,\textbf X}$ in $\mathcal{U}_0$, it is necessary for $\textbf a$ and $\textbf b$ to be eigenvectors of $\textbf A$ and $\textbf{B}$ respectively.
	\item[ii)] Assume that $\textbf a$ and $\textbf b$ are eigenvectors of $\textbf A$ and $\textbf B$, respectively, corresponding to the simple eigenvalues $\sigma_{\textbf a}$ and $\lambda_{\textbf{b}}$. Then $(\pm \textbf{u}_{\mathrm{LDA}}, \pm \textbf{v}_{\mathrm{LDA}})$ are the unique maximizers of $\kappa_{2,\textbf X}$ in $\mathcal{U}_0$ if and only if
	\begin{align*}
	    \sigma_{\textbf a}\lambda_{\textbf b} + \alpha_1 \alpha_2||\textbf a||^2||\textbf b||^2>\sigma_1\lambda_1,
	\end{align*}
	where $\sigma_1$ and $\lambda_1$ are the largest eigenvalues of $\textbf{A}$ and $\textbf{B}$, respectively.

\end{itemize}
\end{theorem}

Theorems~\ref{theo:matrix_lda_mpca_2d2pca} and \ref{theo:matrix_lda_mpca} show that for the second-order methods MPCA and (2D)$ ^2 $PCA to recover the optimal LDA direction in the rank-1 case, it is at the minimum necessary for $\textbf{a}$ and $\textbf{b}$ to be eigenvectors of $\bf A$ and $\bf B$, respectively. However, the corresponding eigenvalues do not necessarily have to be the largest ones, but a certain tolerance is allowed, depending both on how well the clusters are separated and on the mixing proportion and the two covariance matrices. 


\section{Large-sample properties}\label{sec:limiting}

Let $ \textbf{X}_1, \ldots, \textbf{X}_n $ be an i.i.d. sample from the model \eqref{eq:matrix_normal_mixture} with $\mathrm{rank}\{ \textbf{A}^{-1/2} (\textbf{T}_2 - \textbf{T}_1) \textbf{B}^{-1/2} \} =: d \leq \max\{p, q\}$. In pursuing the asymptotic properties of the method we make the following assumption.
\begin{assumption}\label{assu:distinct}
The $d$ non-zero singular values of the matrix $ \textbf{W}_{\mathrm{LDA}} = \textbf{A}^{-1/2} (\textbf{T}_2 - \textbf{T}_1) \textbf{B}^{-1/2} $ are distinct.
\end{assumption}
Assumption \ref{assu:distinct} is made for theoretical convenience. Namely, it ensures that each sequential optimizer is for both indices unique (up to sign), enabling us to approach the problem progressively by establishing the limiting properties of each sequential optimizer one-by-one, finally culminating in the construction of $\textbf{W}_{\mathrm{LDA}}$. In contrast, without Assumption \ref{assu:distinct} the derivation of the limiting properties would be significantly more difficult as, in the worst-case scenario, none of the population-level optimizers would be unique, the only-well defined part of the process being the matrix $\textbf{W}_{\mathrm{LDA}}$, which we would then have to target directly.


The sample versions of the two indices are,
\begin{align*}
\kappa_{n\textbf{X}}(\textbf{u}, \textbf{v}) := \frac{\frac{1}{n} \sum_{i=1}^n \left[ \{ \textbf{u}' ( \textbf{X}_i - \bar{\textbf{X}} ) \textbf{v} \}^4 \right] }{ \left( \frac{1}{n} \sum_{i=1}^n \left[ \{ \textbf{u}' ( \textbf{X}_i - \bar{\textbf{X}} ) \textbf{v} \}^2 \right] \right)^2},
\end{align*}
and
\begin{align*}
\psi_{n\textbf{X}}( \textbf{u} ) := \frac{1}{n} \sum_{i=1}^n\left( \left[ \textbf{u}' ( \textbf{X}_i - \bar{\textbf{X}} ) \left\{ \frac{1}{n} \sum_{i=1}^n ( \textbf{X}_i - \bar{\textbf{X}} )' \textbf{u} \textbf{u}' ( \textbf{X}_i - \bar{\textbf{X}} ) \right\}^{-1} ( \textbf{X}_i - \bar{\textbf{X}} )' \textbf{u} \right]^2  \right),
\end{align*}
and similarly for $\psi_{n\textbf{X}'}( \textbf{v} )$. Furthermore, given a fixed $n \in \mathbb{N}$ and a fixed collection of pairs $ (\textbf{u}_{n1}, \textbf{v}_{n1}), \ldots , (\textbf{u}_{nd}, \textbf{v}_{nd}) \in \mathcal{U}_0 $, the sample versions of the orthogonality constraint sets are $\mathcal{G}_{n1,\textbf{X}} := \{ \textbf{G}_{n11,\textbf{X}}, \ldots , \textbf{G}_{n1(d-1),\textbf{X}} \}$ and $\mathcal{G}_{n2,\textbf{X}} = \{ \textbf{G}_{n21,\textbf{X}}, \ldots , \textbf{G}_{n2(d-1),\textbf{X}} \} $, where
\[ 
	\textbf{G}_{n1k,\textbf{X}} := \frac{1}{n} \sum_{i=1}^n ( \textbf{X}_i - \bar{\textbf{X}} ) \textbf{v}_{nk} \textbf{v}_{nk}' ( \textbf{X}_i - \bar{\textbf{X}} )',
	\]
	and
	\[ 
	\textbf{G}_{n2k,\textbf{X}} := \frac{1}{n} \sum_{i=1}^n ( \textbf{X}_i - \bar{\textbf{X}} )' \textbf{u}_{nk} \textbf{u}_{nk}' ( \textbf{X}_i - \bar{\textbf{X}} ),
	\]
for $ k = 1, \ldots , d - 1$.

Assuming, without loss of generality, that $\alpha \in \mathcal{A}_{\mathrm{max}}$ (the opposite choice leads to minimization instead of maximization, and is treated analogously), we begin by establishing the strong consistency of sequences of $(\mathcal{G}_{n1,\textbf{X}}, \mathcal{G}_{n2,\textbf{X}})$-maximizers of $ \kappa_{n\textbf{X}} $ and $ (\psi_{n \textbf{X}}, \psi_{n \textbf{X}'})$. In the following, let $(\textbf{u}_{01}, \textbf{v}_{01}), \ldots , (\textbf{u}_{0d}, \textbf{v}_{0d}) $ denote any collection of first $d$ singular pairs of the matrix $ \textbf{A}^{-1/2} (\textbf{T}_2 - \textbf{T}_1) \textbf{B}^{-1/2} $ (which, by Assumption \ref{assu:distinct}, are unique up to sign). 

\begin{theorem}\label{theo:strong_consistency}
    Let $ (\textbf{u}_{n1}, \textbf{v}_{n1}), \ldots , (\textbf{u}_{nd}, \textbf{v}_{nd}) \in \mathcal{U}_0 $ be any sequence of $(\mathcal{G}_{n1,\textbf{X}}, \mathcal{G}_{n2,\textbf{X}})$-maximizers of $ \kappa_{n\textbf{X}} $ or $(\psi_{n \textbf{X}}, \psi_{n \textbf{X}'})$. Then there exists sequences of signs $s_{n{u}1}, \ldots , s_{n{u}d} \in \{ -1, 1 \}$ and $s_{n{v}1}, \ldots , s_{n{v}d} \in \{ -1, 1 \}$ such that
    \begin{align*}
        s_{n{u}j} \textbf{u}_{nj} \rightarrow \textbf{u}_j := \frac{\textbf{A}^{-1/2} \textbf{u}_{0j}}{\|\textbf{A}^{-1/2} \textbf{u}_{0j}\|} \quad \mbox{and} \quad s_{n{v}j} \textbf{v}_{nj} \rightarrow \textbf{v}_j := \frac{\textbf{B}^{-1/2} \textbf{v}_{0j}}{\|\textbf{B}^{-1/2} \textbf{v}_{0j}\|}
    \end{align*}
    almost surely, for all $j = 1, \ldots , d$.
\end{theorem}

The proof of Theorem \ref{theo:strong_consistency} is rather general and in no way tied to the current distribution of $\textbf{X}$ (mixture of matrix normals). The key requirements are simply that the population level sequence of $(\mathcal{G}_{1,\textbf{X}}, \mathcal{G}_{2,\textbf{X}})$-maximizers of the objective function is unique up to signs and that none of the successive orthogonality constraints imposed by $\mathcal{G}_{1,\textbf{X}}$ and $\mathcal{G}_{2,\textbf{X}}$ are implied by the earlier ones. 



Focus next on $\kappa_{n \textbf{X}}$ and define the sample counterparts of the quantities introduced in Theorems \ref{theo:matrix_lda_optimal_higher_rank} and \ref{theo:matrix_lda_optimal_mode_wise} as
\begin{align*}
     \lambda_{nj} := \sqrt{ \max \left\{ \frac{\theta_{nj}}{1 - \alpha_1 \alpha_2 \theta_{nj}}, 0 \right\} } \quad \mbox{and} \quad \theta_{nj} := \sqrt{  \max \left\{ \frac{\kappa_{n\textbf{X}}(\textbf{u}_{nj}, \textbf{v}_{nj}) - 3}{\alpha_1 \alpha_2 (1 - 6 \alpha_1 \alpha_2)}, 0 \right\} }
\end{align*}
Furthermore, denote $z_{njk} = (1/n) \sum_{i=1}^n \{ \textbf{u}_{nj}' ( \textbf{X}_i - \bar{\textbf{X}} ) \textbf{v}_{nj} \}^k$. Theorem \ref{theo:strong_consistency} now readily implies the existence of a strongly consistent estimator of the optimal projection $\textbf{W}_{\mathrm{LDA}}$.

\begin{corollary}\label{cor:strong_consistency}
     Let $ (\textbf{u}_{n1}, \textbf{v}_{n1}), \ldots , (\textbf{u}_{nd}, \textbf{v}_{nd}) \in \mathcal{U}_0 $ be any sequence of $(\mathcal{G}_{n1,\textbf{X}}, \mathcal{G}_{n2,\textbf{X}})$-maximizers of $ \kappa_{n\textbf{X}} $. Then
     	\begin{equation}\label{eq:Wlda_consistency_kappa}
		\textbf{W}_{n\mathrm{LDA}} := \sum_{j=1}^d s_{nj} z_{nj2}^{-1/2} \lambda_{nj} \sqrt{1 + \alpha_1 \alpha_2 \lambda_{nj}^2} \textbf{u}_j \textbf{v}_j' \rightarrow \textbf{W}_{\mathrm{LDA}},
\end{equation}
		almost surely, where $ s_{n1}, \ldots, s_{nd} \in \{ -1, 0, 1 \}$ are the signs of the quantities $(\alpha_1 - \alpha_2)^{-1} z_{nj3}$. 
\end{corollary}

A result equivalent to Corollary \ref{cor:strong_consistency} holds also for $\psi_\textbf{X}$ and is proven in the same manner, after changing the constant 3 in the definition of $\theta_{nj}$ above into $q(q+2)$.

\section{Algorithms}\label{sec:algorithm}
As described in Section \ref{sec:intro}, the optimization of the mode-wise index turned out to be unexpectedly computationally demanding. 
Hence, from here onward, we restrict our attention to the simultaneous index $\kappa_{n\textbf{X}}$ only, and for optimizing it we present two approaches. The first one is based on gradient descent with Barzilai-Borwein step size~\citep{Barzilai1988}. Gradient descent being a method of local optimization, we re-initialize a predefined number of times, and the most optimal candidate is then taken as the solution, see Algorithm \ref{alg::kappa_unconst}. The expression for the gradient of  $\kappa_{n\textbf{X}}$ that is needed in the algorithm can be found in the proof of Lemma \ref{lemma:lemma3}.

Both Algorithms \ref{alg::kappa_unconst} and \ref{alg::kappa_fixed_point} use a similar strategy for obtaining the successive optimizers after the first one. Namely, once the first $k-1$ pairs of $(\mathcal{G}_{n1,\textbf{X}},\mathcal{G}_{n2,\textbf{X}})$-optimizers $(\textbf{u}_1,\textbf{v}_1),\dots,(\textbf{u}_{k-1},\textbf{v}_{k-1})$ are found (regardless of the approach), the $k$th pair  $(\textbf{u}_k,\textbf{v}_k)$ is of the form $(\textbf{G}_{\textbf{u},k}^{\perp}\textbf{u},\textbf{G}_{\textbf{v},k}^{\perp}\textbf{v})$, for some $\textbf{u}\in\mathbb{R}^{p-{k-1}}$, $\textbf{v}\in\mathbb{R}^{q-{k-1}}$, where $\textbf{G}_{\textbf{u},k}^{\perp}$ and $\textbf{G}_{\textbf{v},k}^{\perp}$ are \textcolor{black}{arbitrary bases for the} orthogonal complements of $(\textbf{G}_{n11,\textbf{X}}\textbf{u}_1,\dots,\textbf{G}_{n1(k-1), \textbf{X}}\textbf{u}_{k-1})$ and $(\textbf{G}_{n21,\textbf{X}}\textbf{v}_1,\dots,\textbf{G}_{n2(k-1), \textbf{X}}\textbf{v}_{k-1})$, respectively. 

\begin{algorithm}[t]
\caption{Gradient-based algorithm for the optimization of $\kappa_{n\textbf{X}}$.}\label{alg::kappa_unconst}
\SetKwInOut{Input}{Input}
        \Input{$\textbf{X}_1,\dots \textbf{X}_n\in\mathbb{R}^{p \times q}$ centered observations;}
        \BlankLine
	Initialize $\textbf{u}_0$, $\|\textbf{u}_0\|=1$, $\textbf{v}_0$, $\|\textbf{v}_0\|=1$;\\
 	Set the tolerance $\varepsilon>0$ and $e=\varepsilon+1$;\\
 	Set the number of initializations $r>0$;\\
 	Initialize step size $\gamma>0$;\\
 	Initialize $\textbf{G}_{\textbf{u}}^{\perp}\leftarrow\textbf{I}_p$, $\textbf{G}_{\textbf{v}}^{\perp}\leftarrow\textbf{I}_q$;\\ 
    	
	\While{$\min\{p,q\}>1$}{
	$p\leftarrow p-1$; $q\leftarrow q-1$;\\
     \For{$i=1$; $i\leq r$; $i++$}{
	Initialize $\kappa_{min}\leftarrow 100000$;\\
	Calculate $\nabla\kappa_{n(\textbf{G}_{\textbf{u}}^{\perp})'\textbf{X}\textbf{G}_{\textbf{v}}^{\perp}}(\textbf{u}_0,\textbf{v}_0)$;\\

    \While{$e>\varepsilon$}{

	    $(\textbf{u}_1,\textbf{v}_1)\leftarrow (\textbf{u}_0,\textbf{v}_0)-\gamma\nabla\kappa_{n,\textbf{X}}(\textbf{u}_0,\textbf{v}_0)$;\\
	    
	    Calculate $\nabla\kappa_{n(\textbf{G}_{\textbf{u}}^{\perp})'\textbf{X}\textbf{G}_{\textbf{v}}^{\perp}}(\textbf{u}_1,\textbf{v}_1)$;\\
	    $\displaystyle\gamma\leftarrow\frac{|(\textbf{u}_1-\textbf{u}_0,\textbf{v}_1-\textbf{v}_0)'(\nabla\kappa_{n(\textbf{G}_{\textbf{u}}^{\perp})'\textbf{X}\textbf{G}_{\textbf{v}}^{\perp}}(\textbf{u}_1,\textbf{v}_1)-\nabla\kappa_{n(\textbf{G}_{\textbf{u}}^{\perp})'\textbf{X}\textbf{G}_{\textbf{v}}^{\perp}}(\textbf{u}_0,\textbf{v}_0))|}{\|\nabla\kappa_{n\textbf{X}}(\textbf{u}_1,\textbf{v}_1)-\nabla\kappa_{n\textbf{X}}(\textbf{u}_0,\textbf{v}_0)\|^{-2}}$;\\
	    
	    $e\leftarrow\|\nabla\kappa_{n(\textbf{G}_{\textbf{u}}^{\perp})'\textbf{X}\textbf{G}_{\textbf{v}}^{\perp}}(\textbf{u}_1,\textbf{v}_1)\|$;\\
	    $(\textbf{u}_0,\textbf{v}_0)\leftarrow(\textbf{u}_1,\textbf{v}_1)$;
	    }
	   Calculate $\kappa_{n(\textbf{G}_{\textbf{u}}^{\perp})'\textbf{X}\textbf{G}_{\textbf{v}}^{\perp}}(\textbf{u}_1,\textbf{v}_1)$;\\
	    \If{$\kappa_{n(\textbf{G}_{\textbf{u}}^{\perp})'\textbf{X}\textbf{G}_{\textbf{v}}^{\perp}}(\textbf{u}_1,\textbf{v}_1)<\kappa_{min}$}{
	        $\kappa_{min}\leftarrow\kappa_{n,(\textbf{G}_{\textbf{u}}^{\perp})'\textbf{X}\textbf{G}_{\textbf{v}}^{\perp}}(\textbf{u}_1,\textbf{v}_1)$;\\
	        $(\textbf{u},\textbf{v})\leftarrow(\textbf{u}_1,\textbf{v}_1)$;
	        }
	         }
	$(\textbf{u}_1,\textbf{v}_1)\leftarrow(\textbf{u},\textbf{v})$;\\    
	
    Append $\textbf{U}\leftarrow[\textbf{U},\textbf{G}_{\textbf{u}}^{\perp}\textbf{u}_1]$,  $\textbf{V}\leftarrow[\textbf{V},\textbf{G}_{\textbf{v}}^{\perp}\textbf{v}_1]$;\\
    
   Append $\textbf{G}_\textbf{u}\leftarrow[\textbf{G}_\textbf{u},\frac{1}{n}\sum_{i=1}^n\textbf{u}_1'\textbf{X}_i\textbf{v}_1 \cdot \textbf{X}_i\textbf{v}_1]$, $\textbf{G}_\textbf{v}\leftarrow[\textbf{G}_\textbf{v},\frac{1}{n}\sum_{i=1}^n\textbf{u}_1'\textbf{X}_i\textbf{v}_1 \cdot \textbf{X}_i'\textbf{u}_1]$;\\
    
    Calculate the orthogonal complements $\textbf{G}_\textbf{u}^\perp$, $\textbf{G}_\textbf{v}^\perp$;
    }
     
	Return $(\textbf{U},\,\textbf{V})$;
\end{algorithm}

Our second proposed algorithm is based on the following ``flip-flop'' idea: First, we sample a uniformly random unit vector $\textbf{v}_0$ from $\mathbb{S}^{q - 1}$. Next, for this fixed $\textbf{v}_0$, we search for $ \textbf{u}_1 \in \mathbb{S}^{p - 1} $ that optimizes the kurtosis of the linear combination $\textbf{u}_1' \{ (\textbf{X}_i - \bar{\textbf{X}}) \textbf{v}_0 \}$. Afterwards, we hold $\textbf{\textbf{u}}_1$ fixed and take $\textbf{v}_1 \in \mathbb{S}^{q - 1}$ to be the optimizer of the kurtosis of $\textbf{v}_1' \{ (\textbf{X}_i - \bar{\textbf{X}})' \textbf{u}_1 \}$ and so on. This strategy allows reducing the problem to a sequence of optimization problems for vector-valued observations, for which efficient algorithms exist, e.g., the function \texttt{NGPP} in the R-package \texttt{ICtest}~\citep{R_ICTEst}, designed for the maximization of the square excess kurtosis. Thus, no information on mixing proportion $\alpha_1$ is required to extract optimizers this way. Based on our experiments, the success of this approach, formalized as Algorithm~\ref{alg::kappa_fixed_point} in Section \ref{sec:algorithm_appendix}, is highly dependent on the initial projection direction $\textbf{v}_0$ and, as such, we leave its study as part of future work. Thus, instead, in Section \ref{sec:simulations} our simulations and data example will be conducted using Algorithm \ref{alg::kappa_unconst}.

\section{Simulations and  example}\label{sec:simulations}

\subsection{Simulations}


In the simulation study, we consider two
homoscedastic Gaussian mixture models 
with two classes:
$$
\text{Model 1:}\quad\textbf{X}\sim \alpha_1\mathcal{N}_{5\times 3}(\textbf{0},\textbf{A},\textbf{B})+ \alpha_2\mathcal{N}_{5\times 3}(\textbf{T}_{2,1},\textbf{A},\textbf{B}),
$$ 
$$
\text{Model 2:}\quad\textbf{X}\sim \alpha_1\mathcal{N}_{5\times 3}(\textbf{0},\textbf{A},\textbf{B})+ \alpha_2\mathcal{N}_{5\times 3}(\textbf{T}_{2,2},\textbf{A},\textbf{B}),
$$
where $\textbf{A}\in\mathbb{R}^{5\times 5}$ and $\textbf{B}\in\mathbb{R}^{3\times 3}$ have $\mathrm{AR}(1)$-structures with unit variances and autocorrelations $\rho=0.6$ and $\rho=0.3$, respectively. The matrices $\textbf{T}_{2,1}\in\mathbb{R}^{5\times 3}$ and $\textbf{T}_{2,2}\in\mathbb{R}^{5\times 3}$ are chosen randomly such that 
non-zero singular value of $\textbf{A}^{-1/2}\textbf{T}_{2,1}\textbf{B}^{-1/2}$ is $4$ and non-zero singular values of $\textbf{A}^{-1/2}\textbf{T}_{2,2}\textbf{B}^{-1/2}$ are $5$ and $3$. More precisely, $\textbf{T}_{2,1}=\textbf{A}^{1/2}\textbf U_1 \boldsymbol\Lambda_{1,2}\textbf{V}_1\textbf{B}^{1/2}$,  $\textbf{T}_{2,2}=\textbf{A}^{1/2}\textbf U_2 \boldsymbol\Lambda_{2,2}\textbf{V}_2\textbf{B}^{1/2}$, 
where 
$$
\Lambda_{1,2}=\begin{pmatrix}
4&0&0\\
0&0&0\\
0&0&0\\
0&0&0\\
0&0&0
\end{pmatrix},\quad \Lambda_{2,2}=\begin{pmatrix}
5&0&0\\
0&3&0\\
0&0&0\\
0&0&0\\
0&0&0
\end{pmatrix},
$$
$\textbf{U}_1,\,\textbf{U}_2$ and $\textbf{V}_1,\,\textbf{V}_2$ are randomly generated  $5\times 5$ and $3\times 3$ orthogonal matrices, respectively.  

For every combination of the mixing proportion $\alpha\in\{0.1,0.2,0.3,0.4,0.49\}$ and sample size $n\in\{500,1000,2000,4000,8000,16000\}$ we then independently generate $m=1000$ samples from the models and compute for each the optimizers $(\textbf{u}_{ni},\textbf{v}_{ni})$, $i=1,\,2$, using  Algorithm~\ref{alg::kappa_unconst} with $5$ random initializations. Performance was measured in the simulation using the \textit{Maximal similarity index} (MSI) between two unit length vectors, where for $\textbf{x},\textbf{y}\in\mathbb{S}^{p-1}$, MSI$(\textbf{x},\textbf{y}):=|\textbf{x}'\textbf{y}|\in [0,1]$, where the value $0\,(1)$ of MSI correspond to $\textbf{x}$ being orthogonal (parallel) to $\textbf{y}$.
\begin{figure}[h]
    \centering
    \includegraphics[width=\linewidth]{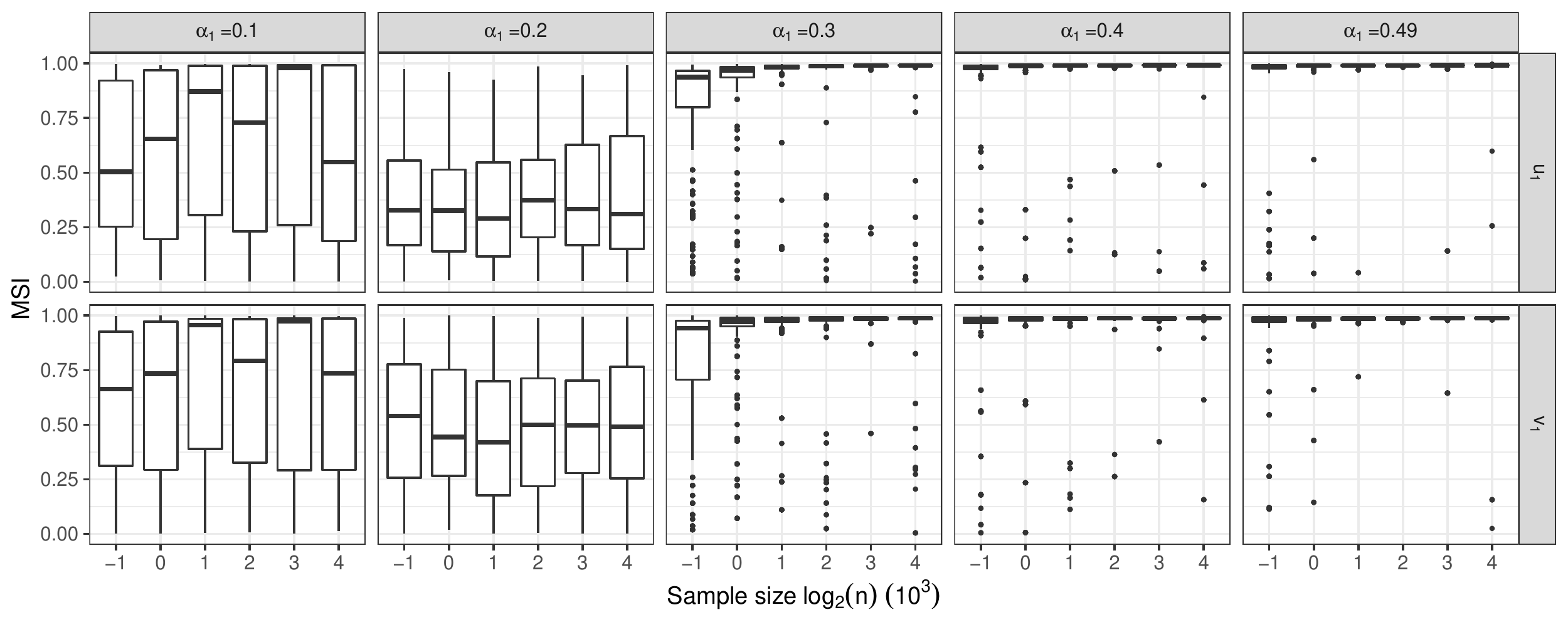}
    \caption{Boxplots of the $m=1000$ MSI-values $|\textbf{u}_{n1}'\textbf{u}_1|$ (left) and  $|\textbf{v}_{n1}'\textbf{v}_1|$ (right) for all combinations of the sample size $n$ and the mixing proportion $\alpha_1$, where $(\textbf{u}_{n1},\textbf{v}_{n1})$ are optimizers of $\kappa_{n,\textbf{X}}$ obtained by Algorithm~\ref{alg::kappa_unconst} under Model~$1$.}
    \label{fig:model_2_box}
\end{figure}
\begin{figure}[ht]
    \centering
    \includegraphics[width=\linewidth]{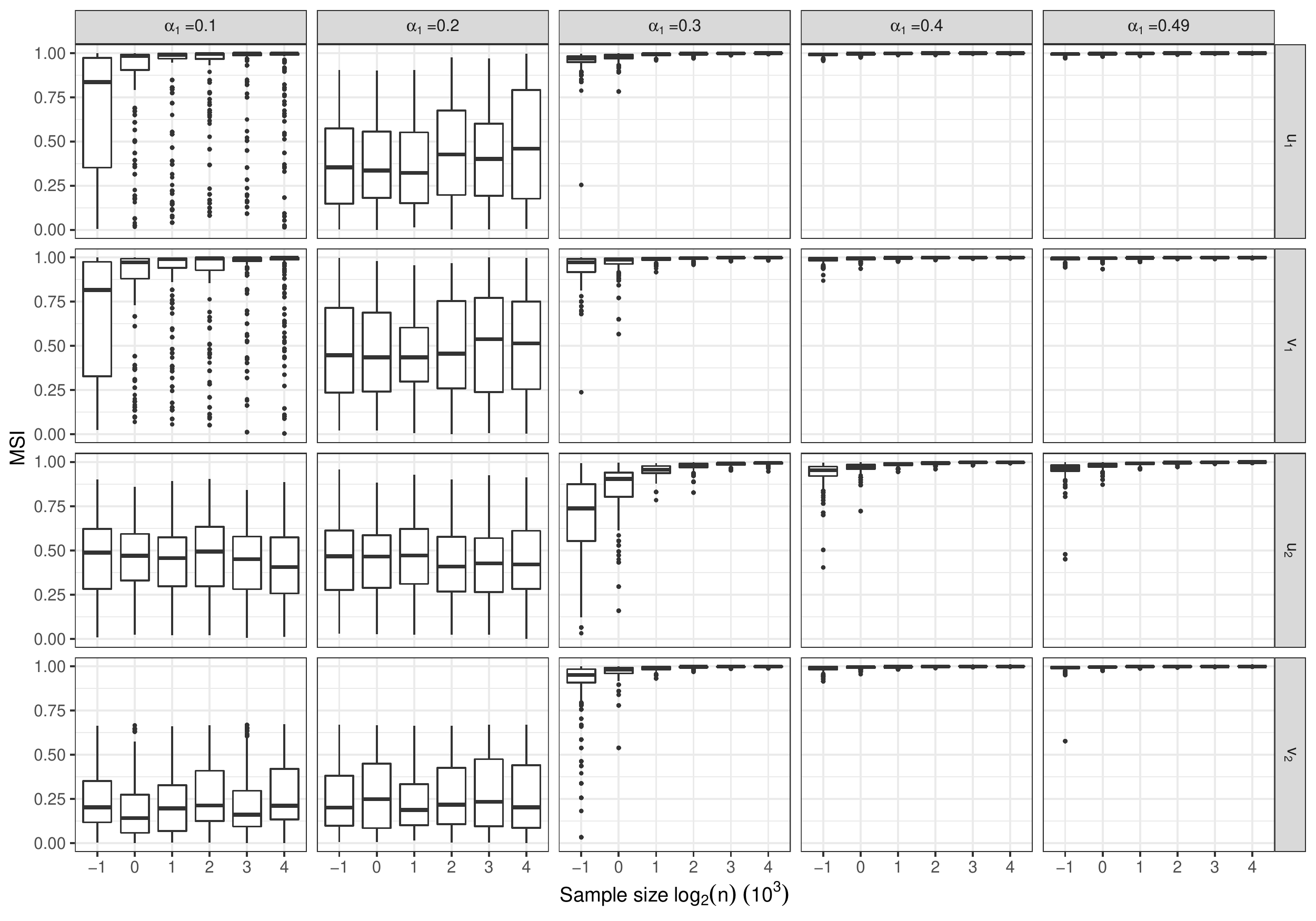}
    \caption{Boxplots of the $m=1000$ MSI-values $|\textbf{u}_{n1}'\textbf{u}_1|$,  $|\textbf{v}_{n1}'\textbf{v}_1|$ (left) and $|\textbf{u}_{n2}'\textbf{u}_2|$,  $|\textbf{v}_{n2}'\textbf{v}_2|$ (right) for all combinations of the sample size $n$ and the mixing proportion $\alpha_1$, where $(\textbf{u}_{ni},\textbf{v}_{ni})$, $i=1,2$, are optimizers of $\kappa_{n,\textbf{X}}$ obtained by Algorithm~\ref{alg::kappa_unconst} under Model $2$.}
    \label{fig:model_1_box}
\end{figure}
The resulting MSI-values are presented in Figures~\ref{fig:model_2_box} and~\ref{fig:model_1_box} and show that Algorithm~\ref{alg::kappa_unconst} estimates the optimal LDA projection with very high accuracy if the groups are moderately balanced. The bad performance for $\alpha_1 = 0.2$ comes with no surprise when we recall that $0.2$ is very close to the value $1/2-1/\sqrt{12} \approx 0.211$ for which the kurtosis of every projection is always $3$ under the model. Similar behaviour was observed in the simulation study performed by~\cite{radojicic2021} in the vector setting. As was to be expected, the estimation accuracy is somewhat worse for the second pair of optimizers in Figure \ref{fig:model_1_box} under the rank-2 Model $2$.

Figures~\ref{fig:W_model_1} and \ref{fig:W_model_2} give boxplots of the logarithmized squared Frobenius norm between the estimated and true $\textbf{W}_\mathrm{LDA}$. It is again visible that if the groups are well balanced and the sample size is large enough, the difference in norms is small, under both models. Finally, Figure~\ref{fig:W_model_2} reveals that the estimation of $\textbf{W}_\mathrm{LDA}$ is less efficient for $\alpha_1=0.49$ than for $\alpha_1\in\{0.3, 0.4\}$ and we recall from Section \ref{sec:lda} that the reason for this is the difficulty of the estimation of the signs $s_1, \ldots , s_d$ of the individual projections for nearly balanced mixtures, see the discussion after Theorem \ref{theo:matrix_lda_optimal_higher_rank}.

Unsurprisingly, Figures~\ref{fig:model_1_box}-\ref{fig:W_model_2} indicate that the accuracy of estimation of the whole mixing matrix $\textbf{W}_\mathrm{LDA}$ is significantly lower than the one of the individual optimizers $(\textbf{u}_k,\textbf{v}_k)$, where the observed behaviour in Model 1 indicates that it is due to poor estimation of coefficients multiplying $\textbf{u}_k\textbf{v}_k$ in the decomposition of $\textbf{W}_\mathrm{LDA}$; see Theorem~\ref{theo:matrix_lda_optimal_higher_rank}. 
Therefore, besides drawing conclusions on the data based solely on the scores obtained by rank-d projection of the data $\langle\textbf{X}_i,\textbf{W}_\mathrm{nLDA}\rangle$, $i=1,\dots,n$, we advise to inspect several rank-1 projections $\textbf{u}_k'\textbf{X}_i\textbf{v}_k'$, for $k=1,\,2,\dots,k_0$, as well. For $k_0\in\mathbb{N}$ small enough, e.g. $k_0=3$, the dimension of the transformed data is small enough so that visualizations and cluster identification are rather straightforward, but the estimates are  more accurate. This strategy bears even more benefits in the examples where the assumption of GMM is violated by e.g. presence of outliers; see for example Figure~\ref{fig:project}. Finally, note that the low MSI (high Frobenius norm) outliers observed in  Figures~\ref{fig:model_1_box}-\ref{fig:W_model_2} are mostly due to poor, randomly generated, initial value used in Algorithm~\ref{alg::kappa_unconst}. Namely, due to the large number of settings considered in the simulation study as well as the $m=1000$ repetitions of each setting, the number of re-initialization used in the simulation study is ``only'' 5. That number should in practice be substantially larger.  

\begin{figure}[h]
\centering
    \includegraphics[width=\linewidth]{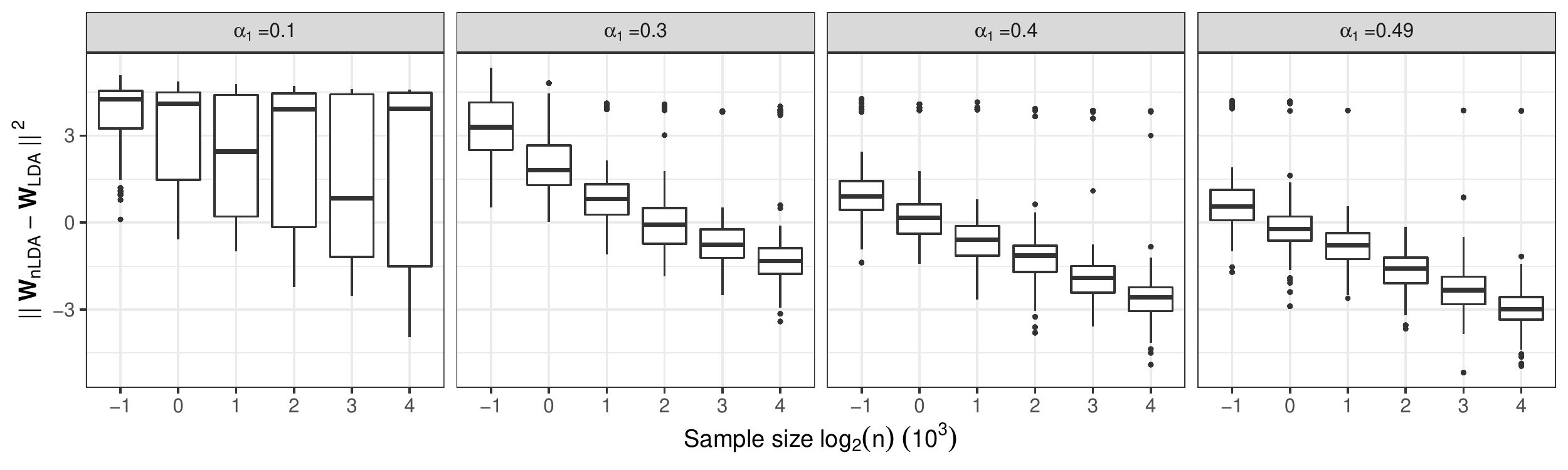}
    \caption{Boxplots of the $m=1000$ logarithmized squared Frobenius norm between the estimated and true $\textbf{W}_\mathrm{LDA}$ for all combinations of the sample size $n$ and the mixing proportion $\alpha_1$, where the estimate is obtained by Algorithm~\ref{alg::kappa_unconst} under Model $1$. The mixing proportion $\alpha_1=0.2$ is excluded due to it inflating the scales of the plot.}
    \label{fig:W_model_1}
\end{figure}

\begin{figure}[ht]
\centering
    \includegraphics[width=\linewidth]{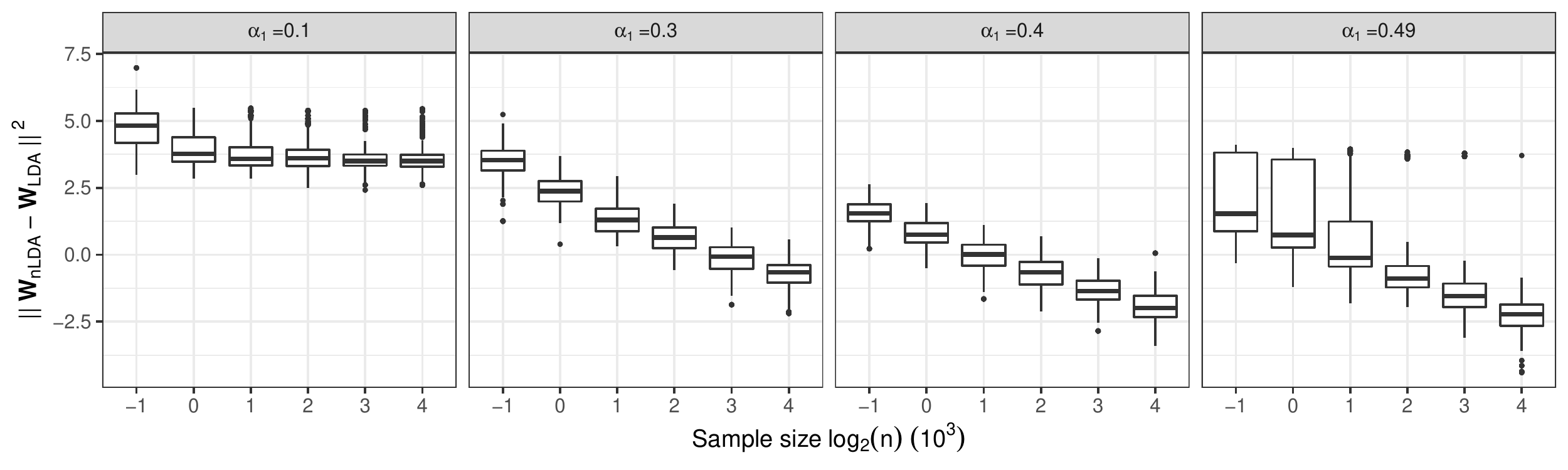}
    \caption{Boxplots of the $m=1000$ logarithmized squared Frobenius norm between the estimated and true $\textbf{W}_\mathrm{LDA}$ for all combinations of the sample size $n$ and the mixing proportion $\alpha_1$, where the estimate is obtained by Algorithm~\ref{alg::kappa_unconst} under Model $2$. The mixing proportion $\alpha_1=0.2$ is excluded due to it inflating the scales of the plot.}
    \label{fig:W_model_2}
\end{figure}



\subsection{Real data example}


\begin{figure}[ht]
    \centering
    \includegraphics[width=0.6\linewidth]{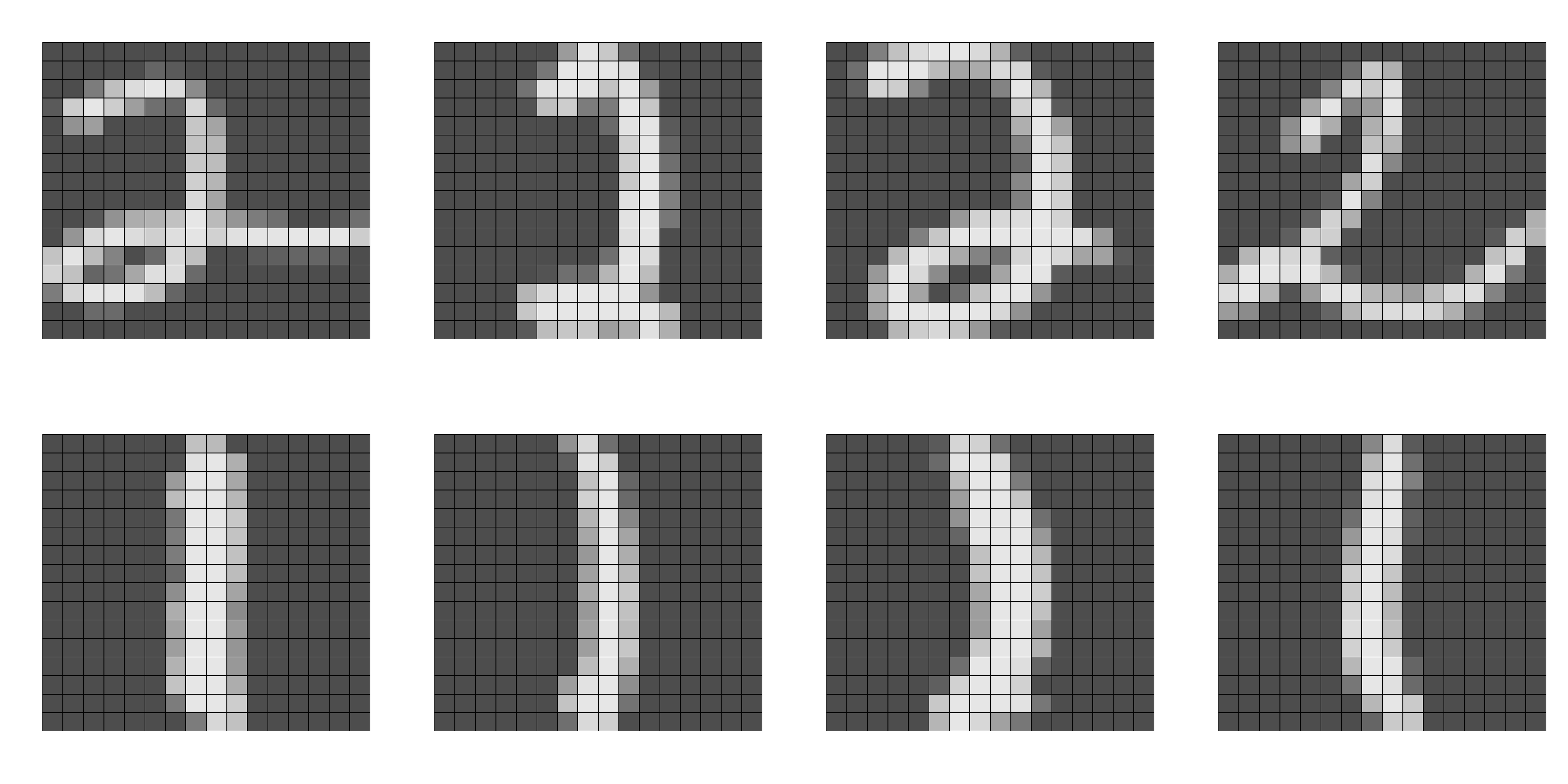}
    \caption{A sample of images of digit $1$ and digit $2$ from the \texttt{digits} data set. Each image is $16 \times 16$ matrix of grayscale intensities.}
    \label{fig:digits2}
\end{figure}

To evaluate the performance of our proposed method in a real data set we consider the data set \texttt{digits}, available freely in the R package \texttt{tensorBSS} \citep{rtensorBSS}. The data consist of $16 \times 16$ grayscale images of normalized handwritten digits ($0,\dots,9$) automatically scanned from envelopes by the U.S. Postal Service. For simplicity, we restrict ourselves to the training subset of $1732$ pictures, $1005$ of which correspond to digit $1$ and $731$ to digit $2$. A sample of the included images is shown in Figure~\ref{fig:digits2}.

We computed the estimate of $\textbf{W}_\mathrm{LDA}$ obtained by Algorithm~\ref{alg::kappa_unconst} using 15 initializations and projected the data both on the direction of the estimate and on the rank-1 direction corresponding to the first pair of found optimizers. As a reference, we computed for the data set the projection given by LDA (in the sense of Lemma \ref{lem:matrix_lda_optimal}) and the projection on the first pair of solutions found by MPCA. 
MPCA was computed using the function \texttt{tTucker} from the R-package \texttt{tensorBSS}~\citep{rtensorBSS}. The four resulting projections are shown in Figure~\ref{fig:project}.

\begin{figure}[ht]
\subfloat[Full rank projection]{\includegraphics[width = 0.25\linewidth]{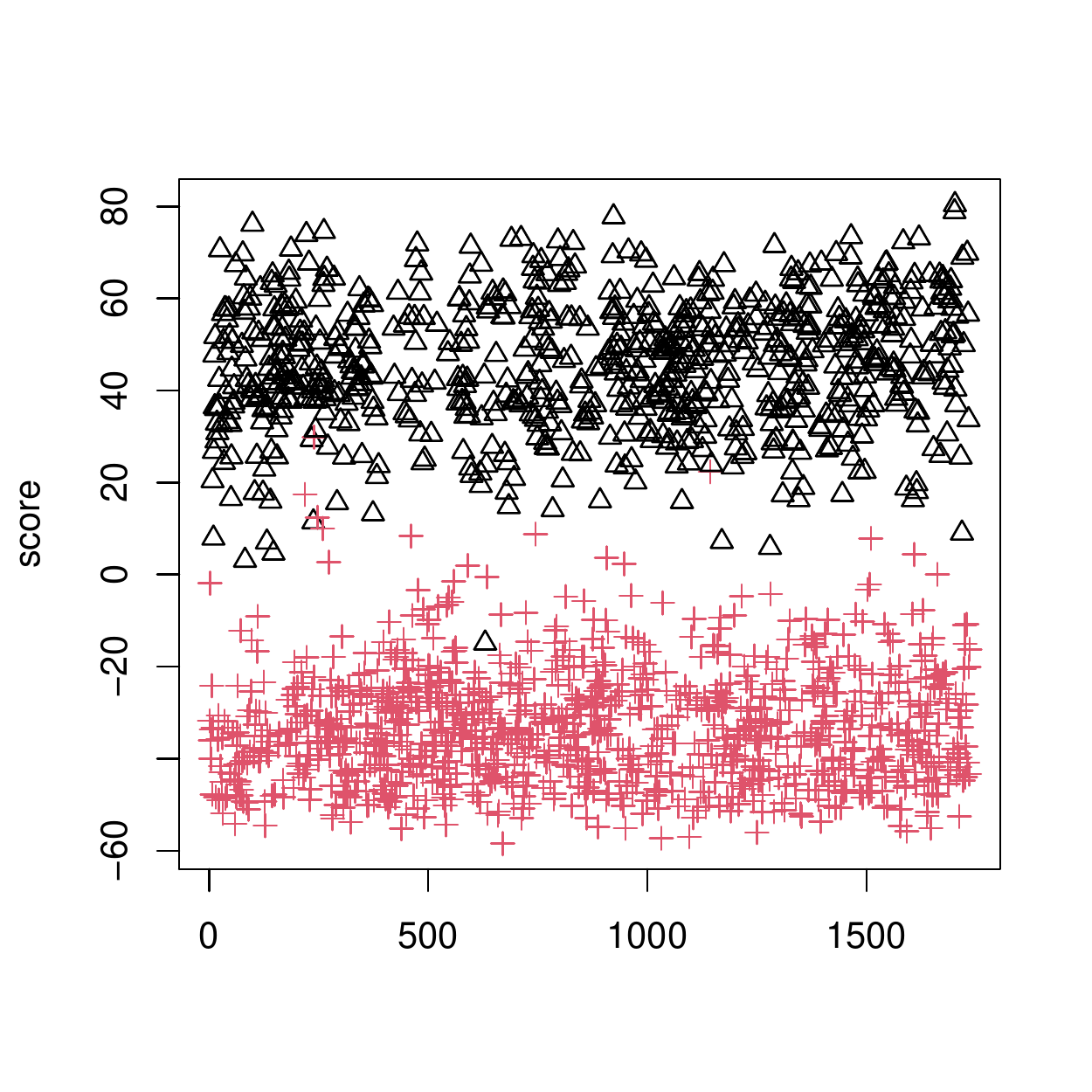}} 
\subfloat[Rank-1 projection]{\includegraphics[width = 0.25\linewidth]{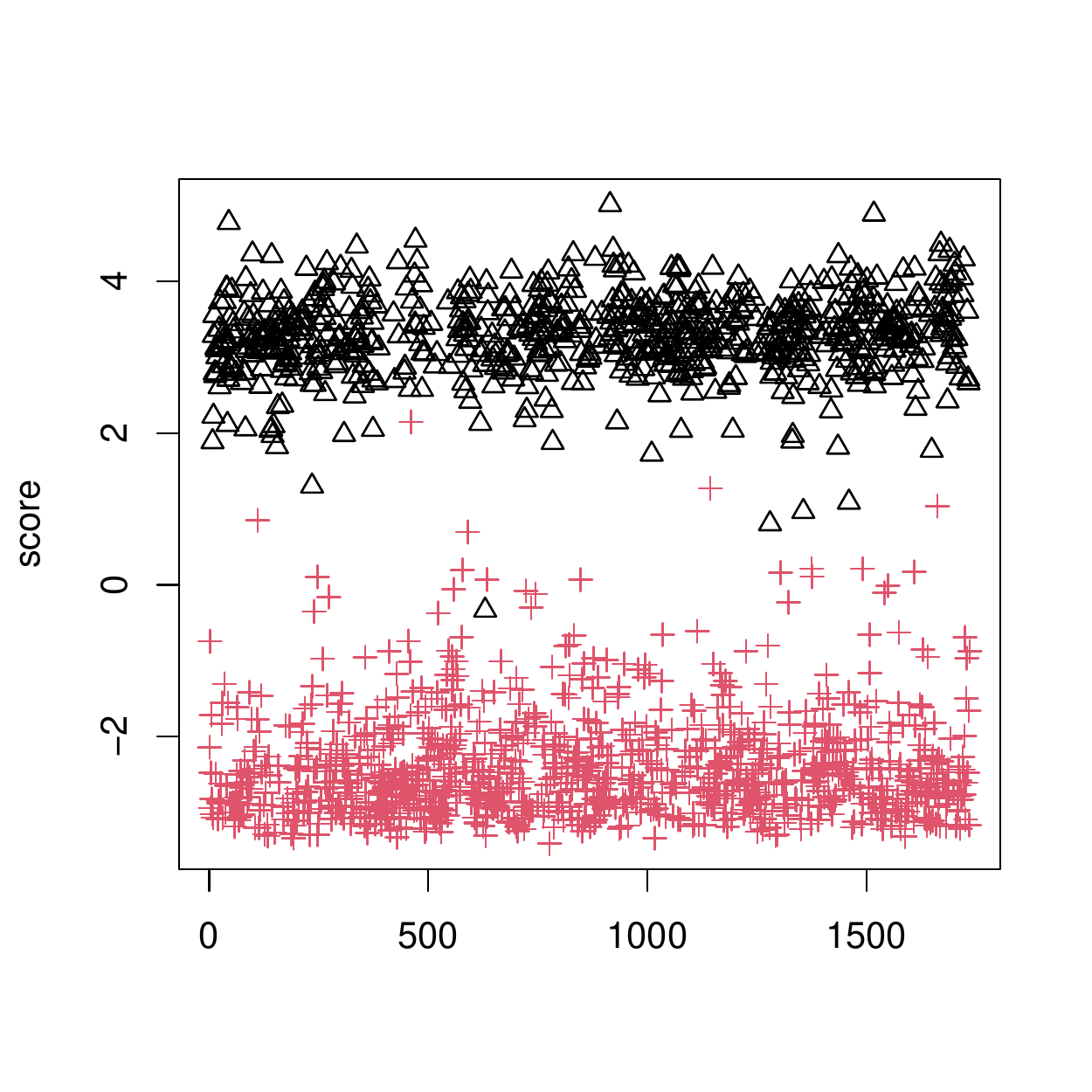}}
\subfloat[LDA]{\includegraphics[width = 0.25\linewidth]{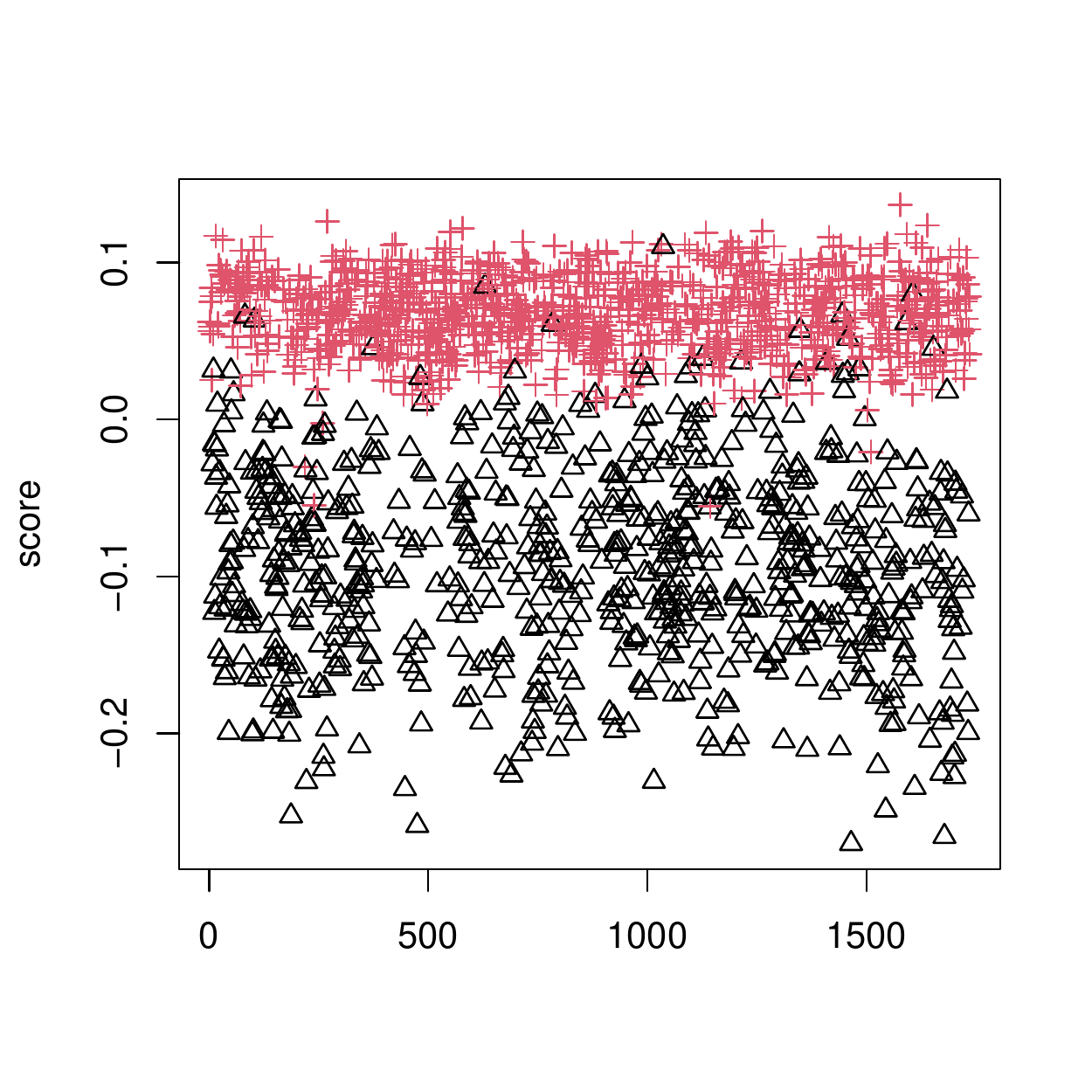}} 
\subfloat[MPCA]{\includegraphics[width = 0.25\linewidth]{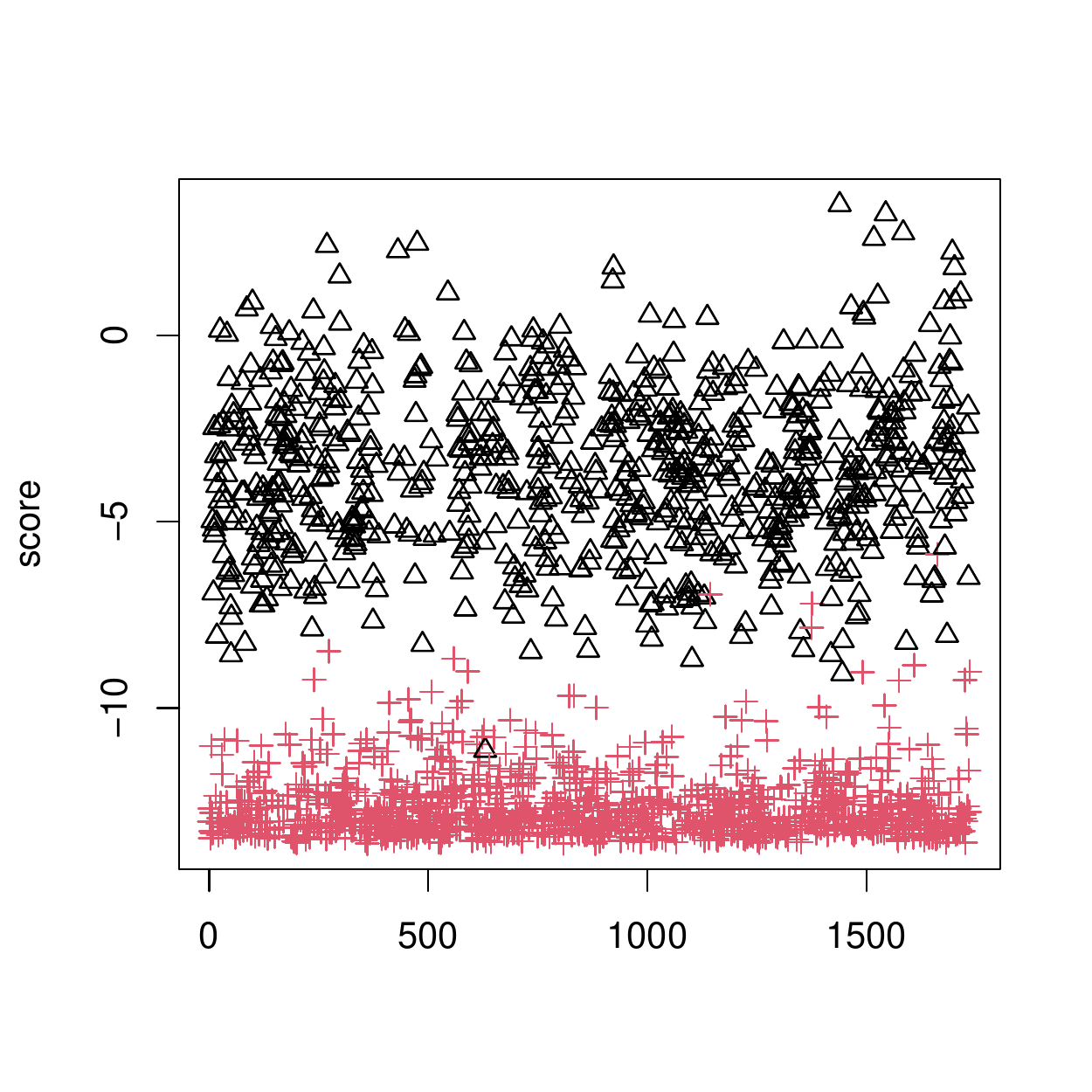}} 
 \caption{Left to right, univariate projections of \textit{digits} data  using the estimate of $\textbf{W}_\mathrm{LDA}$ obtained by Algorithm~\ref{alg::kappa_unconst}, first pair of optimizers of $\kappa_{n,\textbf{X}}$ obtained by Algorithm~\ref{alg::kappa_unconst}, LDA and MPCA, where the coloring and the point shape characterizes group membership (red color corresponds to digit $2$).}
    \label{fig:project}
\end{figure}

Figure~\ref{fig:project} shows that the projection on the first pair of optimizers of $\kappa_{n\textbf{X}}$ (second sub-plot) reveals two clear clusters corresponding to the true data labeling. The same is true also for the projection onto the full estimate of $\textbf{W}_\mathrm{LDA}$ (first sub-plot), but with a less clear cluster structure. However, the supervised LDA estimator (third sub-plot) gives clearly inferior results, thus partially explaining why using just the first pair of optimizers outperforms the use of the full estimate $\textbf{W}_\mathrm{LDA}$. It is worth mentioning that the estimated rank of $\textbf{W}_\mathrm{LDA}$ was $6$ (this estimate was obtained as the number of non-zero values among the estimates of $\lambda_1, \lambda_2, \ldots$).

Further, we clustered the observations along the obtained projections by fitting a two-component Gaussian mixture model via an EM-algorithm as implemented in the R-package \texttt{mclust}~\citep{mclust}. 
If no homoscedasticity restriction on the variances of the model was posed, the misclassification rate was 0.12\% for the rank-1 projection, 1.09\% for the projection onto the full estimate of $\textbf{W}_\mathrm{LDA}$, 3.17\% for the MPCA-based projection and 17.05\% for the LDA-based projection. If, on the other hand, equal variances were assumed, the misclassification rate was 0.23\% for the rank-1 projection, 0.17\% for the projection onto the full estimate of $\textbf{W}_\mathrm{LDA}$, 23.10\% for the MPCA-based projection, and 28.08\% for the LDA-based projection. Thus the matrix projection pursuit approach is clearly the best way here to reduce the dimension of the problem. The relative bad performance of LDA in this example indicates that the data does not follow a Gaussian mixture model.

\section{Discussion}\label{sec:discussion}

In traditional projection pursuit kurtosis is arguably the most popular projection index. There are however many other possibilities to measure what are interesting directions as for example discussed in \cite{huber1985projection,FischerBerroNordhausenRuizGazen2019,RadojicicNordhausnenOja2020}. However, how to generalize these indices to the matrix case is not always clear and is the topic of further research. 
In the vector case invariant coordinate selection (ICS) \cite{Tyler2009} can be seen as projection pursuit without the cost of the pursuit and is for example also able to recover the LDA direction(s) without knowing the class labels \citep{Tyler2009,pena2010eigenvectors}. A direction of further research will be if ICS can also be applied to matrix-variate data, where already some first steps are made in \cite{virta2017independent}. Taking the route via ICS might be also promising in the sense that it alleviates one of the biggest drawbacks of the approach suggested here, which is the high sensitivity to outliers (through our use of fourth moments). In an ICS framework, as a workaround, one could replace the kurtosis as an objective function with one of its robust alternatives. For example, given two \textit{scatter matrices} $\textbf{V}_1, \textbf{V}_2$, \cite{Tyler2009} considers the ratio $\textbf{h}'\textbf{V}_2\textbf{h}/\textbf{h}'\textbf{V}_1\textbf{h}$ as measuring the ``generalized kurtosis'' in the direction $\textbf{h}$ (the regular kurtosis is obtained by a specific choice of $\textbf{V}_1$ and $\textbf{V}_2)$. In general, \cite{Tyler2009} suggest using either a combination of class II and class III scatters, or two class II scatters, also commenting on which combinations can be used to recover the optimal LDA direction in a vector setting. Class II scatters refer to those which are moderately robust, but whose breakdown points are no greater than $1/(p+1)$, where $p$ is the dimension of the data, with multivariate $M$-estimates of the covariance matrix~\citep{maronna1976robust} being maybe the most prominent members of this class~\citep{Tyler2009}. Class III scatters are those with very high breakdown points, with some members of the class being $S$-estimates~\citep{Davies1987}, $\tau$-estimates~\citep{tau_estimates} and the minimum volume ellipsoid~\citep{Rousseeuw1985}. Note that all these estimators are defined for vector data and would have to be first extended to matrix-valued data to apply in our scenario, possibly using a ``flip-flop''-style idea similar to the one in Section \ref{sec:algorithm}. Finally, note that robust scatter matrices are usually computationally very demanding, so further investigation is needed prior to implementing such methods in practice.  

As an alternative to knowing the mixing proportion $\alpha_1$ (and, subsequently, whether to minimize or maximize), we proposed in Section \ref{sec:lda} using the squared excess kurtosis in place of kurtosis, a change which guarantees that maximization is always sufficient. However, in our experiments (not shown here) we discovered that this approach ends up finding mostly directions dominated by outliers. This behaviour is essentially caused by the fact that (excess) kurtosis is bounded from below but not from above. Consider, e.g., data with two perfectly separated groups of equal size and a single outlier. The direction corresponding to the group separation has excess kurtosis roughly somewhere between -2 and 0, whereas there is no upper bound for the excess kurtosis of the direction of the outlier, the actual value depending on its level of outlyingness. Taking now the square of the excess kurtosis then masks the separation direction, making it impossible to find through maximization (naturally, this issue does not happen under the model where no outliers occur). This is in agreement with~\cite{Tyler2009} and~\cite{Ruiz1990}, where it is argued that, in the vector setting, non-robust generalized kurtosis measures should be used as PP indices if the data is expected not to contain any outliers, or if the objective itself is outlier detection, and the same suggestion can be seen to apply also to the use of squared excess kurtosis in our context.

Besides the squared excess kurtosis, another strategy in a situation where the value of $\alpha_1$ is unknown is to estimate several projection directions (by both minimizing and maximizing and using multiple initial values) and then project the data onto these directions separately. By then inspecting these projections individually, one can draw conclusions not only on the value of $\alpha_1$ but also on the presence of outliers and the validity of the model. This strategy seems particularly useful as we noticed in our experiments (and in the real data example in Section \ref{sec:simulations}) that often already the first few projections carry almost all of the information on the group membership.

Another practical issue with the method is that while on the population level, it is guaranteed that $\alpha_1\alpha_2\theta_j < 1$ (see the formula for $\lambda_j$ in Section \ref{sec:lda}), this is not necessarily the case for the sample estimate of the same quantity, especially if the groups are balanced and well-separated (i.e., the kurtosis of the corresponding projections is very small). A possible solution is obtained by replacing the maximum with absolute value in the definition of $\lambda_{j}$, $j=1,\ldots,d$.

Recall that, theoretically, we are unable to reconstruct the optimal projection in Theorems \ref{theo:matrix_lda_optimal_higher_rank} and \ref{theo:matrix_lda_optimal_mode_wise} when the groups are exactly balanced, i.e., $\alpha_1 = 0.5$.  A possible way around this is to ensure that the correlations of the sequential projections are positive. Therefore, in practice, after the first two pairs $(\textbf{u}_{n1},\textbf{v}_{n1})$ and $(\textbf{u}_{n2},\textbf{v}_{n2})$ of optimizers are estimated, one should calculate the correlation of the projections $\textbf{u}_{nk}’\textbf{X}_i\textbf{v}_{nk}$, $k=1,\,2$, $i=1,\dots,n$ of data onto the first two optimizing directions, respectively. If the correlation is negative, sign of the second projection should be changed.  One proceeds in the same manner for further optimizers as well.  Based on our experiments, this idea seems to work very well in practice.


Finally, we note that both proposed projection pursuit indices are \textit{affine equivariant} in the sense that if the data are subjected to the transformation $\textbf{X}_i \mapsto \textbf{M}' \textbf{X}_i \textbf{N} + \textbf{C}$, for some $\textbf{C} \in \mathbb{R}^{p \times q}$ and some full rank $\textbf{M} \in \mathbb{R}^{p \times p}$, $\textbf{N} \in \mathbb{R}^{q \times q}$, then the members $ (\textbf{u}_{n1}, \textbf{v}_{n1}), \ldots , (\textbf{u}_{nd}, \textbf{v}_{nd}) \in \mathcal{U}_0 $ of any sequence of $(\mathcal{G}_{n1,\textbf{X}}, \mathcal{G}_{n2,\textbf{X}})$-maximizers of $ \kappa_{n\textbf{X}} $ or $(\psi_{n \textbf{X}}, \psi_{n \textbf{X}'})$ transform as $(\textbf{u}_{nk}, \textbf{v}_{nk}) \mapsto (\textbf{M}^{-1} \textbf{u}_{nk}/\| \textbf{M}^{-1} \textbf{u}_{nk} \|, \textbf{N}^{-1} \textbf{v}_{nk}/\| \textbf{N}^{-1} \textbf{v}_{nk} \|)$, $k = 1, \ldots, d$ (assuming of course that the constraint matrices are subjected to the transformation). This fact will likely simplify the derivation of the limiting distributions of the estimators, a task we have left for future work, as we conjecture that $\kappa_{n\textbf{X}}$ based estimator $\textbf{W}_\mathrm{nLDA}$ of $\textbf{W}_\mathrm{LDA}$ has limiting normal distribution, with standard $\sqrt{n}$ convergence rate.

A natural extension of the proposed methods is to the general tensors, to accommodate e.g. color images and videos. One could proceed by taking the kurtosis of the rank-1 tensor projections as the projection index. However, in the presented matrix setting, sequential optimizers of $\kappa_\textbf{X}$ are not orthogonal, but satisfy modified orthogonality constraints. Therefore, the extension to the general tensor setting is not so straightforward, thus making it a topic of future research.

\section*{Appendix}
\appendix

\setcounter{lemma}{0}
\renewcommand{\thelemma}{\Alph{section}\arabic{lemma}}

\section{Well-definedness of $\kappa_\textbf{X}$ and $\psi_\textbf{X}$}\label{sec:well_defined}

Assume, without loss of generality, that $\mathrm{E}(\textbf{X}) = \textbf{0} $. Then, the simultaneous index
\[
\kappa_{\textbf{X}}(\textbf{u}, \textbf{v}) = \frac{\mathrm{E} \left( [ \textbf{u}' \textbf{X}  \textbf{v} ]^4 \right) }{ \left\{ \mathrm{E} \left( [ \textbf{u}' \textbf{X} \textbf{v} ]^2 \right) \right\}^2}
\]
is well-defined as soon as the random variable $\textbf{u}' \textbf{X} \textbf{v}$ is, for all $\textbf{u} \in \mathbb{S}^{p -1}$ and $ \textbf{v} \in \mathbb{S}^{q - 1}$,  not almost surely a constant. This condition is equivalent to any of the following:
\begin{itemize}
    \item[a)] The matrix $\mathrm{E}(\textbf{X} \textbf{v} \textbf{v}' \textbf{X}')$ is positive definite for all $\textbf{v} \in \mathbb{S}^{q - 1}$.
    \item[b)] The matrix $\mathrm{E}(\textbf{X}' \textbf{u} \textbf{u}' \textbf{X})$ is positive definite for all $\textbf{u} \in \mathbb{S}^{p - 1}$.
    \item[c)] The random variable $(\textbf{v} \otimes \textbf{u})' \mathrm{vec}(\textbf{X})$ is not almost surely a constant for all $\textbf{u} \in \mathbb{S}^{p -1} $ and $ \textbf{v} \in \mathbb{S}^{q - 1}$.
\end{itemize}

A sufficient, but not necessary, condition for c) to hold is that $\mathrm{Cov}\{ \mathrm{vec}(\textbf{X})\}$ is positive definite which is actually what one needs to assume when applying regular kurtosis-based projection pursuit to the vectorized matrix $\mathrm{vec}(\textbf{X})$. Thus, matrix projection pursuit directly on $\textbf{X}$ (with the simultaneous index $\kappa_\textbf{X}$) requires weaker assumptions than what would be needed if one first converted $\textbf{X}$ to $\mathrm{vec}(\textbf{X})$. In particular, matrix PP allows having perfectly correlated elements in $\textbf{X}$. For example, let the $2 \times 2$ matrix $\textbf{X}$ have the structure
\begin{align*}
    \textbf{X} = \begin{pmatrix}
    x_1 & x_2 \\
    x_2 & x_3 \\
    \end{pmatrix}
\end{align*}
where $x_1, x_2, x_3$ are independent zero-mean random variables with unit variances. Then $\textbf{u}' \mathrm{E}(\textbf{X} \textbf{v} \textbf{v}' \textbf{X}') \textbf{u} = 1 + 2 u_1 u_2 v_1 v_2$ is strictly positive as $|u_1 u_2|, |v_1 v_2| < 1/\sqrt{2}$, showing that we satisfy condition a) above.

However, the conditions a)-c) all involve the parameters $\textbf{u}$ and $\textbf{v}$, making them rather nonintuitive. A parameter-free version is obtained, for example, if $\textbf{X}$ can be expressed as $\textbf{X} = \textbf{C} \textbf{Z} \textbf{D}'$ for some invertible $\textbf{C} \in \mathbb{R}^{p \times p}$, $\textbf{D} \in \mathbb{R}^{q \times q}$ and a $p \times q$ random matrix $\textbf{Z}$ having independent elements with zero means and equal variances (this is the so-called matrix independent component model, see \cite{virta2017independent}). Namely, in this special case, the above conditions are equivalent to
\begin{itemize}
    \item[d)] The matrices $\mathrm{E}(\textbf{X}\textbf{X}')$ and $\mathrm{E}(\textbf{X}'\textbf{X})$ are positive definite.
\end{itemize}
Similarly, under the matrix normal mixture utilized in Section \ref{sec:lda}, the condition d) (which is implied by our assumption in Section \ref{sec:lda} that the covariance parameters $\textbf{A}$ and $\textbf{B}$ are positive-definite) guarantees that $\kappa_\textbf{X}$ is well-defined.

Finally, the mode-wise index 
\begin{align*}
    \psi_{\textbf{X}}( \textbf{u} ) = \mathrm{E}\left[ \left\{ \textbf{u}' \textbf{X} \left[ \mathrm{E} \left( \textbf{X}' \textbf{u} \textbf{u}' \textbf{X} \right) \right]^{-1} \textbf{X}' \textbf{u} \right\}^2  \right]
\end{align*}
is well-defined as soon as the matrix $\mathrm{E}(\textbf{X}' \textbf{u} \textbf{u}' \textbf{X})$ is positive definite for all $\textbf{u} \in \mathbb{S}^{p - 1}$, i.e., under condition b), showing that the above discussion applies to it as well.



\section{Proofs}\label{sec:proofs}

\begin{proof}[Proof of Lemma \ref{lem:existence}]
Starting with the claim \textit{i}), our assumption guarantees that $\kappa_\textbf{X}$ is a continuous function with a compact domain. Hence, there exist $(\textbf{u}_0, \textbf{v}_0),(\textbf{u}_1,\textbf{v}_1)\in\mathcal{U}_0$ such that 
$$
\kappa_\textbf{X}(\textbf{u}_0,\textbf{v}_0)=\sup_{(\textbf{u},\textbf{v})\in\mathcal{U}_0}\kappa_\textbf{X}(\textbf{u},\textbf{v}),\quad \kappa_\textbf{X}(\textbf{u}_1,\textbf{v}_1)=\inf_{(\textbf{u},\textbf{v})\in\mathcal{U}_0}\kappa_\textbf{X}(\textbf{u},\textbf{v}),
$$
proving the claim \textit{i}). To establish the second one, we first observe that, for a fixed $\textbf{u} \in \mathbb{S}^{p - 1}$, the matrix $\mathrm{E} 
[ \{ \textbf{X} - \mathrm{E}(\textbf{X}) \}' \textbf{u} \textbf{u}' \{ \textbf{X} - \mathrm{E}(\textbf{X}) \} ]$ is invertible if and only if it is positive-definite, i.e., when
\begin{align*}
    0 < \textbf{v}' \mathrm{E} [ \{ \textbf{X} - \mathrm{E}(\textbf{X}) \}' \textbf{u} \textbf{u}' \{ \textbf{X} - \mathrm{E}(\textbf{X}) \} ] \textbf{v} = \mathrm{E} \left( [ \textbf{u}' \{ \textbf{X} - \mathrm{E}(\textbf{X}) \} \textbf{v} ]^2 \right), \quad \mbox{for all } \textbf{v} \in \mathbb{S}^{q-1}.
\end{align*}
But, by our assumption the above holds for all $\textbf{u} \in \mathbb{S}^{p-1}$, making the function $\textbf{u} \mapsto (\mathrm{E} 
[ \{ \textbf{X} - \mathrm{E}(\textbf{X}) \}' \textbf{u} \textbf{u}' \{ \textbf{X} - \mathrm{E}(\textbf{X}) \} ])^{-1}$ continuous. Hence, also $\psi_\textbf{X}$ is continuous, and arguing now as in part \textit{i}) establishes the second claim.
\end{proof}

\begin{proof}[Proof of Lemma \ref{lem:matrix_lda_optimal}]
	Recall that the density function of the matrix normal distribution $ \mathcal{N}_{p \times q}(\textbf{T}, \textbf{A}, \textbf{B}) $ writes
	\[
	f_{\textbf{T}, \textbf{A}, \textbf{B}}(\textbf{X}) = \frac{1}{(2 \pi)^{pq/2} | \textbf{A} |^{q/2} | \textbf{B} |^{p/2}} \exp\left\{  -\frac{1}{2} \mathrm{tr}\left[ \textbf{B}^{-1} (\textbf{X} - \textbf{T})' \textbf{A}^{-1} (\textbf{X} - \textbf{T}) \right] \right\}.
	\]
	It is straightforward to verify that Fisher's linear discriminant rule,
	\[
	f_{\textbf{T}_2, \textbf{A}, \textbf{B}}(\textbf{X}) > f_{\textbf{T}_1, \textbf{A}, \textbf{B}}(\textbf{X}),
	\]
	of classifying an observation $ \textbf{X} \in \mathbb{R}^{p \times q} $ to class 2 equals:
	\begin{align*}
	\mathrm{tr}\left[ \textbf{A}^{-1} (\textbf{T}_2 - \textbf{T}_1) \textbf{B}^{-1} \textbf{X}' \right] &> \frac{1}{2} \left( \mathrm{tr}\left[ \textbf{A}^{-1} \textbf{T}_2 \textbf{B}^{-1} \textbf{T}_2' \right] -  \mathrm{tr}\left[ \textbf{A}^{-1} \textbf{T}_1 \textbf{B}^{-1} \textbf{T}_1' \right] \right).
	\end{align*}
	I.e. the classification depends on $ \textbf{X} $ only through its projection
	\[
	\mathrm{tr}\left[ \textbf{A}^{-1} (\textbf{T}_2 - \textbf{T}_1) \textbf{B}^{-1} \textbf{X}' \right] = \langle \textbf{A}^{-1} (\textbf{T}_2 - \textbf{T}_1) \textbf{B}^{-1}, \textbf{X} \rangle,
	\]
	onto $ 	\textbf{W}_{\mathrm{LDA}} = \textbf{A}^{-1} (\textbf{T}_2 - \textbf{T}_1) \textbf{B}^{-1} $, proving the claim.
\end{proof}

Before proving Theorem \ref{theo:matrix_lda_optimal_higher_rank}, we first establish an auxiliary lemma.

\begin{lemma}\label{lem:rational_function}
	Let $ \alpha \in (0, 1) $. The function $ f: [0, \infty) \rightarrow \mathbb{R} $ defined as
	\[
	f(x) = \frac{3 + 6 \alpha (1 - \alpha) x + \alpha (1 - \alpha) [\alpha^3 + (1 - \alpha)^3 ] x^2 }{[1 + \alpha (1 - \alpha) x]^2}
	\]
	\begin{itemize}
		\item[i)] is strictly decreasing if $ | \alpha - 1/2 | < 1/\sqrt{12} $,
		\item[ii)] is a constant function, $ f(x) = 3 $, if $ | \alpha - 1/2 | = 1/\sqrt{12} $,
		\item[iii)] is strictly increasing if $ | \alpha - 1/2 | > 1/\sqrt{12} $.
	\end{itemize}
\end{lemma}

\begin{proof}[Proof of Lemma \ref{lem:rational_function}]
	The denominator of $ f'(x) $ is always positive, meaning that the sign of $ f'(x) $ is determined by the sign of its numerator,
	\begin{align*}
	& (6 \beta + 2 \beta \gamma x)(1 + \beta x)^2 - 2 \beta (3 + 6 \beta x+ \beta \gamma x^2 )(1 + \beta x) \\
	=& 2 \beta (1 + \beta x) \left[  (3 + \gamma x)(1 + \beta x) - (3 + 6 \beta x+ \beta \gamma x^2 ) \right] \\
	=& 2 \beta (1 + \beta x) x \left( \gamma - 3 \beta \right)
	\end{align*}
	where we have used the shorthand $ \beta := \alpha (1 - \alpha) $ and $ \gamma := \alpha^3 + (1 - \alpha)^3 $. Now,
	\[
	\gamma - 3 \beta = 1 - 6 \alpha (1 - \alpha) =: h(\alpha),
	\]
	and $ h(\alpha) $ can be verified to have $ h(\alpha) < 0 $ if $ | \alpha - 1/2 | < 1/\sqrt{12} $ and $ h(\alpha) > 0 $ if $ | \alpha - 1/2 | > 1/\sqrt{12} $, establishing parts \textit{i)} and \textit{iii)} of the claim. The part \textit{ii)} follows by observing that the values $ \alpha = 1/2 \pm 1/\sqrt{12} $ satisfy $ \gamma = 3 \beta $, yielding
	\[
	f(x) = \frac{3 + 6 \beta x + 3 \beta^2 x^2 }{(1 + \beta x)^2} = 3.
	\]
\end{proof}

\begin{proof}[Proof of Theorem \ref{theo:matrix_lda_optimal_higher_rank}]
	
		For arbitrary $ \textbf{u} \in \mathbb{R}^p $ and $ \textbf{v} \in \mathbb{R}^q $, the projection $ \textbf{u}' \textbf{X} \textbf{v} $ has
	\begin{align*}
	\textbf{u}' \textbf{X} \textbf{v} &\sim \alpha_1 \mathcal{N}( \textbf{u}' \textbf{T}_1 \textbf{v}, \textbf{u}' \textbf{A} \textbf{u} \textbf{v}' \textbf{B} \textbf{v}) + \alpha_2 \mathcal{N}( \textbf{u}' \textbf{T}_2 \textbf{v}, \textbf{u}' \textbf{A} \textbf{u} \textbf{v}' \textbf{B} \textbf{v})\\
	&=  \alpha_1 \mathcal{N}( m_1(\textbf{u}, \textbf{v}), s^2(\textbf{u}, \textbf{v}) ) +  \alpha_2 \mathcal{N}( m_2(\textbf{u}, \textbf{v}), s^2(\textbf{u}, \textbf{v}) ),
	\end{align*}
	where $ m_i(\textbf{u}, \textbf{v}) := \textbf{u}' \textbf{T}_i \textbf{v} $ and $ s^2(\textbf{u}, \textbf{v}) := \textbf{u}' \textbf{A} \textbf{u} \textbf{v}' \textbf{B} \textbf{v} $. The kurtosis of the projection $ \textbf{u}' \textbf{X} \textbf{v} $ is now, by the proof of Theorem 1 in \cite{pena2017clustering},
	\[
	\kappa_{\textbf{X}}(\textbf{u}, \textbf{v}) = \frac{3 s^4(\textbf{u}, \textbf{v}) + 6 a_2(\textbf{u}, \textbf{v}) s^2(\textbf{u}, \textbf{v}) + a_4(\textbf{u}, \textbf{v})}{[s^2(\textbf{u}, \textbf{v}) + a_2(\textbf{u}, \textbf{v})]^2},
	\]
	where
	\[
	a_d(\textbf{u}, \textbf{v}) := \sum_{i = 1}^2 \alpha_i \left( m_i(\textbf{u}, \textbf{v}) - \sum_{j = 1}^2 \alpha_j m_j(\textbf{u}, \textbf{v}) \right)^d.
	\]
	Denoting next $ \textbf{H} := \textbf{T}_2 - \textbf{T}_1 $ and  $ \bar{\textbf{T}} = \sum_{i = 1}^2 \alpha_i \textbf{T}_i $, we have $ \textbf{T}_1 - \bar{\textbf{T}} = -\alpha_2 \textbf{H} $ and $ \textbf{T}_2 - \bar{\textbf{T}} = \alpha_1 \textbf{H} $, and the terms $ a_2(\textbf{u}, \textbf{v}) $, $ a_4(\textbf{u}, \textbf{v}) $ can be written as
	\begin{align*}
	a_2(\textbf{u}, \textbf{v}) &= \alpha_1 \alpha_2 (\textbf{u}' \textbf{H} \textbf{v})^2 \\
	a_4(\textbf{u}, \textbf{v}) &= \alpha_1 \alpha_2 (\alpha_1^3 + \alpha_2^3) (\textbf{u}' \textbf{H} \textbf{v})^4.
	\end{align*}
	Let then $ t := \textbf{u}' \textbf{A} \textbf{u} \textbf{v}' \textbf{B} \textbf{v} $, $ \beta := \alpha_1 \alpha_2 $, $ \gamma := (\alpha_1^3 + \alpha_2^3) $ and, additionally, $ w := (\textbf{u}' \textbf{H} \textbf{v})^2 $. Under this notation, the kurtosis of the projection reads,
	\begin{align}\label{eq:wt_form}
	\kappa_{\textbf{X}}(\textbf{u}, \textbf{v}) = \frac{3 (t + \beta w)^2 + \beta (\gamma - 3 \beta) w^2}{(t + \beta w)^2} = 3 + \beta (\gamma - 3 \beta) \left\{ \frac{(w/t)}{1 + \beta (w/t)} \right\}^2.
	\end{align}
	Assume now that $ \alpha_1 \in \mathcal{A}_0 $. Then, by the proof of Lemma \ref{lem:rational_function}, $ \gamma - 3 \beta = 0 $, yielding  $ \kappa_{\textbf{X}}(\textbf{u}, \textbf{v}) = 3 $, regardless of the directions $ \textbf{u} $ and $ \textbf{v} $, and showing part \textit{ii)} of the claim.
	
	Assume next that $ \alpha_1 \in \mathcal{A}_{\mathrm{min}} \setminus \{ \frac{1}{2} \} $ and let $ z := w/t \geq 0$. Then, again by the proof of Lemma \ref{lem:rational_function}, $ \gamma - 3 \beta < 0 $ and, consequently, $ \kappa_{\textbf{X}}(\textbf{u}, \textbf{v}) $ is minimized when the function $ g: [0, \infty) \rightarrow [0, \infty) $, acting as
	\[
	g(z) = \frac{z}{1 + \beta z},
	\]
	achieves its maximal value. Now, $ g'(z) = (1 + \beta z)^{-2} > 0 $, for all $ z \geq 0 $, showing that $ g $ is a strictly increasing function. The direction with the minimal kurtosis is thus the one giving the largest value of $ z $.
	
	Now, we have
	\begin{align}\label{eq:variational_inequality}
	z = \frac{w}{t} = \left\{ \left( \frac{\textbf{A}^{1/2} \textbf{u}}{\|\textbf{A}^{1/2} \textbf{u}\|}  \right)' \textbf{A}^{-1/2} \textbf{H} \textbf{B}^{-1/2} \left( \frac{\textbf{B}^{1/2} \textbf{v}}{\|\textbf{B}^{1/2} \textbf{v}\|}  \right) \right\}^2. 
	\end{align}
	By the variational characterization of singular values, the expression \eqref{eq:variational_inequality} is maximized if and only if $ \textbf{A}^{1/2} \textbf{u} $ and $ \textbf{B}^{1/2} \textbf{v} $ are proportional, respectively, to members of any pair of unit length singular vectors, $ \textbf{u}_{01} $ and $ \textbf{v}_{01} ,$ of $ \textbf{A}^{-1/2} \textbf{H} \textbf{B}^{-1/2} $ corresponding to its largest singular value $ \sigma_1 $ (that is, $ \textbf{A}^{-1/2} \textbf{H} \textbf{B}^{-1/2} \textbf{v}_{01} = \sigma_1 \textbf{u}_{01} $). Hence,
	\begin{align}\label{eq:solution_1_kappa}
	(\textbf{u}_1, \textbf{v}_1) = \left( s_{\textbf{u}1} \frac{\textbf{A}^{-1/2} \textbf{u}_{01}}{\|\textbf{A}^{-1/2} \textbf{u}_{01}\|}, s_{\textbf{v}1} \frac{\textbf{B}^{-1/2} \textbf{v}_{01}}{\|\textbf{B}^{-1/2} \textbf{v}_{01}\|} \right),
	\end{align} 
	for some signs $ s_{\textbf{u}1}, s_{\textbf{v}1} \in \{ -1, 1 \} $. Having obtained the first pair, the second one is found by maximizing $ z $ under the constraints $ \textbf{u}_2' \textbf{G}_{11,\textbf{X}} \textbf{u}_1 = 0 $ and $ \textbf{v}_2' \textbf{G}_{21,\textbf{X}} \textbf{v}_1 = 0 $, where
	\[ 
	\textbf{G}_{11,\textbf{X}} =  \mathrm{E} \left[ \{ \textbf{X} - \mathrm{E}(\textbf{X}) \} \textbf{v}_1 \textbf{v}_1'  \{ \textbf{X} - \mathrm{E}(\textbf{X}) \}' \right],
	\]
	and
	\[ 
	\textbf{G}_{21,\textbf{X}} =  \mathrm{E} \left[ \{ \textbf{X} - \mathrm{E}(\textbf{X}) \}' \textbf{u}_1 \textbf{u}_1'  \{ \textbf{X} - \mathrm{E}(\textbf{X}) \} \right].
	\]
	We next show that the condition $ \textbf{u}_2' \textbf{G}_{11,\textbf{X}} \textbf{u}_1 = 0 $ is equivalent to $ \textbf{u}_2' \textbf{A}^{1/2} \textbf{u}_{01} = 0 $. Observing that,
	\[ 
	\{ \textbf{X} - \mathrm{E}(\textbf{X}) \} \textbf{v}_1 \sim \alpha_1 \mathcal{N}_p(-\alpha_2 \textbf{H} \textbf{v}_1, \textbf{v}_1' \textbf{B} \textbf{v}_1 \cdot \textbf{A}) + \alpha_2 \mathcal{N}_p(\alpha_1 \textbf{H} \textbf{v}_1, \textbf{v}_1' \textbf{B} \textbf{v}_1 \cdot \textbf{A}),
	\]
	we get
	\[ 
	\textbf{G}_{11,\textbf{X}} = \textbf{v}_1' \textbf{B} \textbf{v}_1 \cdot \textbf{A} + \beta  \textbf{H} \textbf{v}_1 \textbf{v}_1' \textbf{H}'. 
	\]
	Hence, plugging in \eqref{eq:solution_1_kappa} and by the properties of SVD,
	\begin{align*}
	\textbf{u}_2' \textbf{G}_{11,\textbf{X}} \textbf{u}_1 =& \|\textbf{B}^{-1/2} \textbf{v}_{01}\|^{-2} \|\textbf{A}^{-1/2} \textbf{u}_{01}\|^{-1}  \left( 1 + \beta \sigma_1^2 \right) s_{\textbf{u}1} \textbf{u}_2' \textbf{A}^{1/2} \textbf{u}_{01}
	\end{align*}
	The above shows that indeed $ \textbf{u}_2' \textbf{G}_{11,\textbf{X}} \textbf{u}_1 = 0 $ if and only if $ \textbf{u}_2'  \textbf{A}^{1/2} \textbf{u}_{01} = 0 $. One can similarly show that $ \textbf{v}_2' \textbf{G}_{21,\textbf{X}} \textbf{v}_1 = 0 $ holds if and only if $ \textbf{v}_2' \textbf{B}^{1/2} \textbf{v}_{01} = 0 $. Thus, again by the variational characterization of singular values, \eqref{eq:variational_inequality} is maximized under the constraints that $ \textbf{A}^{1/2} \textbf{u} $ be orthogonal to $ \textbf{u}_{01} $ and that $ \textbf{B}^{1/2} \textbf{v} $ be orthogonal to $ \textbf{v}_{01} $ if and only if $ \textbf{A}^{1/2} \textbf{u} $ and $ \textbf{B}^{1/2} \textbf{v} $ are proportional, respectively, to members of any pair of unit length singular vectors, $ \textbf{u}_{02} $ and $ \textbf{v}_{02} ,$ of $ \textbf{A}^{-1/2} \textbf{H} \textbf{B}^{-1/2} $ corresponding to its second largest singular value $ \sigma_2 $ (which may be equal to $ \sigma_1 $). That is,
	\[ 
	(\textbf{u}_2, \textbf{v}_2) = \left( s_{\textbf{u}2} \frac{\textbf{A}^{-1/2} \textbf{u}_{02}}{\|\textbf{A}^{-1/2} \textbf{u}_{02}\|}, s_{\textbf{v}2} \frac{\textbf{B}^{-1/2} \textbf{v}_{02}}{\|\textbf{B}^{-1/2} \textbf{v}_{02}\|} \right),
	\]
	for some signs $ s_{\textbf{u}2}, s_{\textbf{v}2} \in \{ -1, 1 \} $.
	Continuing analogously, we finally get that, for all $ j = 1, \ldots , d $,
	\begin{align}\label{eq:kappa_uv_solution}
	(\textbf{u}_j, \textbf{v}_j) = \left( s_{\textbf{u}j} \frac{\textbf{A}^{-1/2} \textbf{u}_{0j}}{\|\textbf{A}^{-1/2} \textbf{u}_{0j}\|}, s_{\textbf{v}j} \frac{\textbf{B}^{-1/2} \textbf{v}_{0j}}{\|\textbf{B}^{-1/2} \textbf{v}_{0j}\|} \right),
	\end{align}
	for some signs $ s_{\textbf{u}j}, s_{\textbf{v}j} \in \{ -1, 1 \} $, where $ ( \textbf{u}_{0j},  \textbf{v}_{0j} ) $ is a pair of unit length singular vectors of $ \textbf{A}^{-1/2} \textbf{H} \textbf{B}^{-1/2} $ corresponding to its $ j $th largest singular value $ \sigma_j $.
	
	Consider next the quantities $ \mathrm{E} ( [ \textbf{u}_j' \{ \textbf{X} - \mathrm{E}(\textbf{X})  \} \textbf{v}_j  ]^2 ) $ in the denominator of the decomposition in part \textit{i)} of the lemma. By \eqref{eq:kappa_uv_solution}, the projection $ \textbf{u}_j' \{ \textbf{X} - \mathrm{E}(\textbf{X})  \} \textbf{v}_j $ has the distribution
	\begin{align}\label{eq:uXv_distribution} 
	\textbf{u}_j' \{ \textbf{X} - \mathrm{E}(\textbf{X})  \} \textbf{v}_j \sim \alpha_1 \mathcal{N}(-\alpha_2 s_{\textbf{u}j} s_{\textbf{v}j} c_j \sigma_j, c_j^2) + \alpha_2 \mathcal{N}(\alpha_1 s_{\textbf{u}j} s_{\textbf{v}j} c_j \sigma_j, c_j^2),
	\end{align}
	where $ c_j = \|\textbf{B}^{-1/2} \textbf{v}_{0j}\|^{-1} \|\textbf{A}^{-1/2} \textbf{u}_{0j}\|^{-1} $. This gives
	\[ 
	\mathrm{E} ( [ \textbf{u}_j' \{ \textbf{X} - \mathrm{E}(\textbf{X})  \} \textbf{v}_j  ]^2 ) = c_j^2 (1 + \beta \sigma_j^2).
	\]
	
	Next, we simplify $ \lambda_j $ and $ \theta_j $ (defined just before the statement of Theorem \ref{theo:matrix_lda_optimal_higher_rank}). Given the $ j $th pair of optimizers $ (\textbf{u}_j, \textbf{v}_j) $,
	\[ 
	\kappa_{\textbf{X}}(\textbf{u}_j, \textbf{v}_j) = 3 + \beta (\gamma - 3 \beta) \left( \frac{z}{1 + \beta z} \right)^2,
	\] 
	where $ z $ is computed by substituting $ (\textbf{u}_j, \textbf{v}_j) $ to \eqref{eq:variational_inequality} and equals, by the variational characterization of singular values, $ \sigma_j^2 $. Hence, using the identity $\gamma = 1 - 3 \beta$, we get $ \theta_j = \sigma_j^2/(1 + \beta \sigma_j^2) $, and plugging in to the definition of $ \lambda_j $ shows that $ \lambda_j = \sigma_j $.
	
	Combining all the previous gives,
	\begin{align}\label{eq:kappa_final_equation}
	\begin{split}
	& \sum_{j=1}^d \frac{s_j \lambda_j \sqrt{1 + \alpha_1 \alpha_2 \lambda_j^2} }{\sqrt{\mathrm{E} ( [ \textbf{u}_j' \{ \textbf{X} - \mathrm{E}(\textbf{X})  \} \textbf{v}_j  ]^2 )}} \textbf{u}_j \textbf{v}_j'\\
	=& \textbf{A}^{-1/2} \left( \sum_{j=1}^d \frac{s_j  \sigma_j\sqrt{1 + \beta \sigma_j^2}}{ c_j \sqrt{1 + \beta \sigma_j^2}} s_{\textbf{u}j} s_{\textbf{v}j} c_j  \textbf{u}_{0j} \textbf{v}_{0j}' \right) \textbf{B}^{-1/2} \\
	=& \textbf{A}^{-1/2} \left( \sum_{j=1}^d s_j s_{\textbf{u}j} s_{\textbf{v}j} \sigma_j \textbf{u}_{0j} \textbf{v}_{0j}' \right) \textbf{B}^{-1/2},
	\end{split}
	\end{align}
	where $ s_1, \ldots, s_d \in \{ -1, 1 \}$ are the signs of the quantities 
	\begin{align}\label{eq:kappa_sign_choosing}
	(\alpha_1 - \alpha_2)^{-1} \mathrm{E} ( [ \textbf{u}_j' \{ \textbf{X} - \mathrm{E}(\textbf{X})  \} \textbf{v}_j  ]^3 ).
	\end{align} 
	Now, by \eqref{eq:uXv_distribution} and the moment formulas for normal distribution,
	\[ 
	(\alpha_1 - \alpha_2)^{-1} \mathrm{E} ( [ \textbf{u}_j' \{ \textbf{X} - \mathrm{E}(\textbf{X}) \} \textbf{v}_j ]^3 ) = \beta (\alpha_1 + \alpha_2) s_{\textbf{u}j} s_{\textbf{v}j} c_j^3 \sigma_j^3,
	\]
	which is non-zero and has the same sign as $ s_{\textbf{u}j} s_{\textbf{v}j} $, implying that $ s_j $ satisfies $ s_j s_{\textbf{u}j} s_{\textbf{v}j} = 1 $. Plugging this in into \eqref{eq:kappa_final_equation} gives,
	\[ 
	\sum_{j=1}^d \frac{s_j \lambda_j \sqrt{1 + \alpha_1 \alpha_2 \lambda_j^2} }{\sqrt{\mathrm{E} ( [ \textbf{u}_j' \{ \textbf{X} - \mathrm{E}(\textbf{X})  \} \textbf{v}_j  ]^2 )}} \textbf{u}_j \textbf{v}_j' = \textbf{A}^{-1/2} \left( \sum_{j=1}^d \sigma_j \textbf{u}_{0j} \textbf{v}_{0j}' \right) \textbf{B}^{-1/2} = \textbf{A}^{-1} \textbf{H} \textbf{B}^{-1} = \textbf{W}_{\mathrm{LDA}},
	\]
	concluding the proof for part \textit{i)} of the claim. The proof for part \textit{iii)} is exactly analogous and we omit it. 
\end{proof}

\begin{proof}[Proof of Lemma~\ref{lemma:lemma3}]
    
By \eqref{eq:wt_form}, we have that $ \kappa_{\textbf{X}}(\textbf{u}, \textbf{v}) = 3 + \beta (\gamma - 3 \beta) \{ h(\textbf{u}, \textbf{v}) + \beta \}^{-2} $, where $h(\textbf{u}, \textbf{v}) := \textbf{u}' \textbf{A} \textbf{u} \cdot \textbf{v}' \textbf{B} \textbf{v} \cdot (\textbf{u}' \textbf{H} \textbf{v})^{-2} $. By the chain rule, the gradient of $\kappa_{\textbf{X}}$ is proportional to the gradient of the function $ h $, whose $\textbf{u}$-part can, by a straightforward computation, be seen to equal,
\begin{align}\label{eq:gradient_population}
    2 \textbf{v}' \textbf{B} \textbf{v} (\textbf{u}' \textbf{H} \textbf{v})^{-3} (\textbf{u}' \textbf{H} \textbf{v} \cdot \textbf{A} \textbf{u} - \textbf{u}' \textbf{A} \textbf{u} \cdot \textbf{H} \textbf{v}).
\end{align}
By the proof of Theorem \ref{theo:matrix_lda_optimal_higher_rank}, $(\textbf{u}_k, \textbf{v}_k) = (c_{ku} \textbf{A}^{-1/2} \textbf{u}_{0k}, c_{kv} \textbf{B}^{-1/2} \textbf{v}_{0k})$, where $c_{ku}, c_{kv}$ are non-zero constants and $(\textbf{u}_{0k}, \textbf{v}_{0k})$ is a unit length singular pair of the matrix $\textbf{A}^{-1/2} \textbf{H} \textbf{B}^{-1/2}$. Plugging in to \eqref{eq:gradient_population} now shows that the $\textbf{u}$-part of the gradient vanishes and, by symmetry, this happens also the $\textbf{v}$-part, concluding the proof. 
\end{proof}

\begin{proof}[Proof of Theorem \ref{theo:matrix_lda_optimal_mode_wise}]
	
	Let again $ \bar{\textbf{T}} = \sum_{i = 1}^2 \alpha_i \textbf{T}_i $. Then, denoting $ \textbf{y} := \tilde{\textbf{X}}' \textbf{u} $, we have
	\[
	\textbf{y} \sim \alpha_1 \mathcal{N}_q(\textbf{m}_1, \boldsymbol{\Sigma}) + \alpha_2 \mathcal{N}_q(\textbf{m}_2, \boldsymbol{\Sigma}),
	\]
	where $ \textbf{m}_1 := (\textbf{T}_1 - \bar{\textbf{T}})' \textbf{u} = - \alpha_2 \textbf{H}' \textbf{u} $, $ \textbf{m}_2 := (\textbf{T}_2 - \bar{\textbf{T}})' \textbf{u} = \alpha_1 \textbf{H}' \textbf{u} $, $ \boldsymbol{\Sigma} := (\textbf{u}' \textbf{A} \textbf{u}) \textbf{B}  $ and $ \textbf{H} := \textbf{T}_2 - \textbf{T}_1 $. Let next
	\[
	\textbf{y}_1 \sim \mathcal{N}_q(\textbf{m}_1, \boldsymbol{\Sigma}) \quad \mbox{and} \quad \textbf{y}_2 \sim \mathcal{N}_q(\textbf{m}_2, \boldsymbol{\Sigma}).
	\]
	Then we have
	\begin{align*}
	\mathrm{E} \left( \tilde{\textbf{X}}' \textbf{u} \textbf{u}' \tilde{\textbf{X}} \right) =& \alpha_1 \mathrm{E}(\textbf{y}_1 \textbf{y}_1' ) + \alpha_2 \mathrm{E}(\textbf{y}_2 \textbf{y}_2' ) \\
	=& \boldsymbol{\Sigma} + \alpha_1 \textbf{m}_1 \textbf{m}_1' + \alpha_2 \textbf{m}_2 \textbf{m}_2'\\
	=& \boldsymbol{\Sigma} + \alpha_1 \alpha_2 \textbf{H}' \textbf{u} \textbf{u}' \textbf{H} \\
	=:& \textbf{G}
	\end{align*}
	Thus, the kurtosis $ \psi_{\textbf{X}}(\textbf{u}) $ in the direction $ \textbf{u} $ equals
	\begin{align*}
	\psi_{\textbf{X}}(\textbf{u}) =& \mathrm{E}\left\{ \left( \textbf{y}' \textbf{G}^{-1} \textbf{y} \right)^2  \right\} = \alpha_1 \mathrm{E}\left\{ \left( \textbf{y}_1' \textbf{G}^{-1} \textbf{y}_1 \right)^2  \right\} + \alpha_2 \mathrm{E}\left\{ \left( \textbf{y}_2' \textbf{G}^{-1} \textbf{y}_2 \right)^2  \right\}.
	\end{align*}
	Using next the formulas for the moments of quadratic forms of normal random vectors (see, e.g., Theorem 3.2b.2 in \cite{mathai1992quadratic}) and simplifying gives the following expression for the kurtosis $ \psi_{\textbf{X}}(\textbf{u}) $,
	\begin{align*}
	\psi_{\textbf{X}}(\textbf{u}) =& 2 \mathrm{tr}\{ (\textbf{G}^{-1} \boldsymbol{\Sigma} )^2 \} + 4 \beta \textbf{u}' \textbf{H} \textbf{G}^{-1} \boldsymbol{\Sigma} \textbf{G}^{-1} \textbf{H}' \textbf{u} + \{ \mathrm{tr}(\textbf{G}^{-1} \boldsymbol{\Sigma}) \}^2 \\
	&+ 2 \beta  \mathrm{tr}(\textbf{G}^{-1} \boldsymbol{\Sigma}) \textbf{u}' \textbf{H} \textbf{G}^{-1} \textbf{H}' \textbf{u} + \beta  \gamma (\textbf{u}' \textbf{H} \textbf{G}^{-1} \textbf{H}' \textbf{u})^2,
	\end{align*}
	where $ \beta := \alpha_1 \alpha_2 $ and $ \gamma := \alpha_1^3 + \alpha_2^3 $. Writing the same using the parametrization
	\[
	\textbf{u}_0 := \textbf{A}^{1/2} \textbf{u} \quad \mbox{and} \quad \textbf{H}_0 := \textbf{A}^{-1/2} \textbf{H} \textbf{B}^{-1/2},
	\]
	gives
	\begin{align}\label{eq:0_version_psi}
	\begin{split}
	\psi_{\textbf{X}}(\textbf{u}) =& 2 \| \textbf{u}_0 \|^4 \mathrm{tr}( \textbf{G}_0^{-2} ) + 4 \beta \| \textbf{u}_0 \|^2 \textbf{u}_0' \textbf{H}_0 \textbf{G}_0^{-2} \textbf{H}_0' \textbf{u}_0 + \| \textbf{u}_0 \|^4 \{ \mathrm{tr}( \textbf{G}_0^{-1} ) \}^2 \\
	&+ 2 \beta \| \textbf{u}_0 \|^2 \mathrm{tr}( \textbf{G}_0^{-1} ) \textbf{u}'_0 \textbf{H}_0 \textbf{G}_0^{-1} \textbf{H}_0' \textbf{u}_0 + \beta  \gamma (\textbf{u}_0' \textbf{H}_0 \textbf{G}_0^{-1}  \textbf{H}_0' \textbf{u}_0)^2,
	\end{split}
	\end{align}
	where $ \textbf{G}_0 := \| \textbf{u}_0 \|^2 \textbf{I}_q + \beta \textbf{H}_0' \textbf{u}_0 \textbf{u}_0' \textbf{H}_0 $. By the Sherman-Morrison formula, the inverse of $ \textbf{G}_0 $ is
	\[
	\textbf{G}_0^{-1} = \| \textbf{u}_0 \|^{-2} \left( \textbf{I}_q - \frac{\beta \| \textbf{u}_0 \|^{-2} \textbf{H}_0' \textbf{u}_0 \textbf{u}_0' \textbf{H}_0}{1 + \beta \| \textbf{u}_0 \|^{-2} \textbf{u}_0' \textbf{H}_0 \textbf{H}_0' \textbf{u}_0 } \right).
	\]
	Denoting $ z := \| \textbf{u}_0 \|^{-2} \textbf{u}_0' \textbf{H}_0 \textbf{H}_0' \textbf{u}_0 $ and plugging the inverse in to \eqref{eq:0_version_psi}, we get
	\begin{align*}
	\begin{split}
	\psi_{\textbf{X}}(\textbf{u}) &= 2 \left\{ q - \frac{2\beta z}{1 + \beta z} + \frac{\beta^2 z^2}{(1 + \beta z)^2} \right\} + 4 \beta z \left\{ 1 -  \frac{2 \beta z}{1 + \beta z} + \frac{\beta^2 z^2}{(1 + \beta z)^2} \right\} + \left( q - \frac{\beta z}{1 + \beta z} \right)^2 \\
	&+ 2 \beta z \left( q - \frac{\beta z}{1 + \beta z} \right) \left( 1 - \frac{\beta z}{1 + \beta z} \right) + \beta^{-1} \gamma \left(\beta z - \frac{\beta^2 z^2}{1 + \beta z} \right)^2,
	\end{split}
	\end{align*}
	which simplifies to,
	\[
	\psi_{\textbf{X}}(\textbf{u}) = \frac{q(q + 2)(1 + \beta z)^2 + \beta (\gamma - 3 \beta) z^2}{(1 + \beta z)^2} = q(q + 2) + \beta (\gamma - 3 \beta) \left( \frac{z}{1 + \beta z} \right)^2.
	\]
	The final part of the proof follows similarly as in the proof of Theorem~\ref{theo:matrix_lda_optimal_higher_rank} and, thus, we next go over only the main steps.
	
	Letting $\textbf{R} := \textbf{A}^{-1/2} \textbf{H} \textbf{B}^{-1/2} $, the SVD,
	\begin{align*}
	    \textbf{R} = \sum_{j = 1}^d \sigma_j \textbf{u}_{0j} \textbf{v}_{0j}',
	\end{align*}
	is unique up to signs of the pairs of singular vectors. Now, 
	\[
	z = \left( \frac{\textbf{A}^{1/2} \textbf{u}}{\|\textbf{A}^{1/2} \textbf{u}\|}  \right)' \textbf{R} \textbf{R}' \left( \frac{\textbf{A}^{1/2} \textbf{u}}{\|\textbf{A}^{1/2} \textbf{u}\|}  \right),
	\]
	and the eigenvectors of $\textbf{R}\textbf{R}'$ are the left singular vectors of $\textbf{R}$ and the $d$ non-zero eigenvalues of $\textbf{R}\textbf{R}'$ are the squared non-zero singular values of $\textbf{R}$ (and, hence, distinct). Hence, by the variational characterization of eigenvalues, the maximizers of $z$ are precisely all $\textbf{u}_1$ such that,
	\begin{align*}
	     \textbf{u}_1 = s_{\textbf{u}1} \frac{\textbf{A}^{-1/2} \textbf{u}_{01}}{\|\textbf{A}^{-1/2} \textbf{u}_{01}\|},
	\end{align*}
	where $ s_{\textbf{u}1} \in \{ -1, 1 \} $. Next, from the proof of Theorem \ref{theo:matrix_lda_optimal_higher_rank} we know that the condition $ \textbf{u}_2' \textbf{G}_{\textbf{X}, 1} \textbf{u}_1 = 0 $ is equivalent to $ \textbf{u}_2' \textbf{A}^{1/2} \textbf{u}_{01} = 0 $. Thus, continuing as in the proof of Theorem \ref{theo:matrix_lda_optimal_higher_rank} (but with the variational characterization of eigenvalues instead of singular values), we finally get that, for all $j = 1, \ldots , d$,
	\begin{align*}
	     \textbf{u}_j = s_{\textbf{u}j} \frac{\textbf{A}^{-1/2} \textbf{u}_{0j}}{\|\textbf{A}^{-1/2} \textbf{u}_{0j}\|},
	\end{align*}
	for some signs $ s_{\textbf{u}j} \in \{ -1, 1 \} $. Applying the previous reasoning to the transposed matrix $\textbf{X}$ then also gives that, for all $ j = 1, \ldots , d $,
	\begin{align*}
	\textbf{v}_j = s_{\textbf{v}j} \frac{\textbf{B}^{-1/2} \textbf{v}_{0j}}{\|\textbf{B}^{-1/2} \textbf{v}_{0j}\|},
	\end{align*}
	for some signs $ s_{\textbf{v}j} \in \{ -1, 1 \} $.
	
	The remainder of the proof follows now the same steps as the proof of Theorem \ref{theo:matrix_lda_optimal_higher_rank} (apart from the slight changes in the definition of $\theta_j$) and, as such, we omit it. 

\end{proof}

\begin{proof}[Proof of Theorem~\ref{theo:matrix_lda_mpca}]
\textit{i)} As discussed, $\textbf{u}' \{ \textbf{X}-\mathbb{E}(\textbf{X}) \} \textbf{v}\sim\alpha_1 \mathcal{N}\{ m_1(\textbf{u},\textbf{v}),s^2(\textbf{u},\textbf{v}) \}+\alpha_2 \mathcal{N} \{ m_2(\textbf{u},\textbf{v}),s^2(\textbf{u},\textbf{v}) \}$, where $s^2(\textbf{u},\textbf{v})=\textbf{u}'\textbf{A}\textbf{u} \cdot \textbf{v}'\textbf{B}\textbf{v}$, $m_1(\textbf{u},\textbf{v})=\alpha_2\textbf{u}'\textbf{H}\textbf{v}$, $m_2(\textbf{u},\textbf{v})=-\alpha_1 \textbf{u}'\textbf{H}\textbf{v}$ and $\textbf{H}=\textbf{T}_1-\textbf{T}_2$. Therefore,
$$
\kappa_{2,\textbf{X}}(\textbf{u},\textbf{v})=\textbf{u}'\textbf{A}\textbf{u} \cdot \textbf{v}'\textbf{B}\textbf{v}+\alpha_1\alpha_2(\textbf{u}'\textbf{H}\textbf{v})^2.
$$
Observe that $\kappa_{2,\textbf{X}}$ is continuously differentiable in $\textbf{u}$ and $\textbf{v}$. 
 Therefore, all minima and maxima of $\kappa_{2,\textbf{X}}$ are stationary points of the Lagrange function
$$
f(\textbf{u},\textbf{v};\theta_1,\theta_2)=\textbf{u}'\textbf{A}\textbf{u} \cdot \textbf{v}'\textbf{B}\textbf{v}+\alpha_1\alpha_2(\textbf{u}'\textbf{H}\textbf{v})^2-\theta_1(\textbf{u}'\textbf{u}-1)-\theta_2(\textbf{v}'\textbf{v}-1),
$$
whose partial derivatives with respect to $\textbf{u}$ and $\textbf{v}$ are given by
\begin{align*}
\partial_u f(\textbf{u},\textbf{v};\theta_1,\theta_2)&=2(\textbf{v}'\textbf{B}\textbf{v})\textbf{A}\textbf{u}+2\alpha_1\alpha_2(\textbf{u}'\textbf{H}\textbf{v})\textbf{H}\textbf{v}-2\theta_1\textbf{u},\\
\partial_v f(\textbf{u},\textbf{v};\theta_1,\theta_2)&=2(\textbf{u}'\textbf{A}\textbf{u})\textbf{B}\textbf{v}+2\alpha_1\alpha_2(\textbf{u}'\textbf{H}\textbf{v})\textbf{H}'\textbf{u}-2\theta_2\textbf{v}.
\end{align*}
In order to find the necessary conditions for the standardized LDA optimal directions $(\pm \textbf{u}_{\mathrm{LDA}}, \pm \textbf{v}_{\mathrm{LDA}})=(c_1\textbf A^{-1}\textbf a,c_2\textbf B^{-1}\textbf b)$, $c_1= \pm ||\textbf A^{-1}\textbf a||^{-1}$, $c_2= \pm ||\textbf B^{-1}\textbf b||^{-1}$ to be stationary points of $f$ we solve $\nabla f(\textbf{u}_{L},\textbf{v}_{L}) = \textbf{0}$. Hence, it is necessary that
\begin{align*}
c_1c_2^2 (\textbf b'\textbf B^{-1}\textbf B\textbf B^{-1}\textbf b)\textbf A\textbf A^{-1}\textbf a+c_1c_2^2\alpha_1\alpha_2(\textbf a'\textbf A^{-1}\textbf H\textbf B^{-1}\textbf{b})\textbf H\textbf B^{-1}\textbf b-c_1\theta_1\textbf A^{-1}\textbf a=&\textbf{0}\\
c_1^2c_2 (\textbf a'\textbf A^{-1}\textbf A\textbf A^{-1}\textbf a)\textbf B\textbf B^{-1}\textbf b+c_1^2c_2\alpha_1\alpha_2(\textbf a'\textbf A^{-1}\textbf H\textbf B^{-1}\textbf b)\textbf H' \textbf A^{-1}\textbf a-c_2\theta_2\textbf B^{-1}\textbf b=&\textbf{0}.
\end{align*}
Taking into account that $\textbf H=\textbf a\textbf b'$, the first part above then implies that
$
c_2^2 (\textbf{b}' \textbf{B}^{-1} \textbf{b}) \textbf a+c_2^2\alpha_1\alpha_2 (\textbf{b}' \textbf{B}^{-1} \textbf{b})^2 (\textbf{a}' \textbf{A}^{-1} \textbf{a}) \textbf a-\theta_1\textbf A^{-1}\textbf a=\textbf{0}$. Multiplication by $\textbf{A}$ now reveals that $\textbf{a}$ is an eigenvector of $\textbf{A}$. By symmetry, also $\textbf{b}$ then has to be an eigenvector of $\textbf B$.

\noindent\textit{ii)} Assume that $\textbf a$ and $\textbf b$ are eigenvectors of $\textbf A$ and $\textbf B$, respectively, corresponding to the simple eigenvalues $\sigma_{\textbf a}$ and $\lambda_{\textbf{b}}$. Hence, $\textbf a$ and $\textbf b$ are also eigenvectors of $\textbf A^{-1}$ and $\textbf B^{-1}$, respectively, implying that $\textbf{u}_{\mathrm{LDA}}= \frac{\textbf A^{-1}\textbf a}{||\textbf A^{-1}\textbf a||}=\frac{\textbf{a}}{||\textbf{a}||}$ and $\textbf{v}_{\mathrm{LDA}}= \frac{\textbf B^{-1}\textbf b}{||\textbf B^{-1}\textbf b||}=\frac{\textbf{b}}{||\textbf{b}||}$. Next, denote any sets of orthonormal eigenvectors of $\textbf{A}$ and $\textbf B$ as $\textbf u_1,\dots, \textbf u_p$ and $\textbf v_1,\dots, \textbf v_q$, respectively. Moreover, let the corresponding eigenvalues be $\sigma_1\geq\sigma_2\geq\cdots\geq\sigma_p> 0$ and $\lambda_1\geq\lambda_2\geq\cdots\geq\lambda_q> 0$, respectively. Therefore, for every $\textbf u\in\mathbb{R}^p$, $\|\textbf u\|=1$ there exists $\boldsymbol{\beta}=(\beta_1,\ldots, \beta_p)$, $\sum_{i=1}^p\beta_i^2=1$, such that $\textbf u=\sum_{i=1}^p\beta_i\textbf u_i$. Observe, furthermore, that since $\textbf u_{\textbf{a}}$ belongs to simple eigenvalue, we must have $\textbf{u}_{\mathrm{LDA}} = \pm \textbf u_k$ for some $k\in\{1,\ldots, p\}$ and the corresponding eigenvalue has $\sigma_{\textbf a} = \sigma_k$. Similarly, for every $\textbf v\in\mathbb{R}^q$, $\|\textbf v \| = 1$ there exists $\boldsymbol{\gamma}=(\gamma_1,\ldots, \gamma_q)$, $\sum_{j=1}^q\gamma_j^2=1$, such that $\textbf v=\sum_{j=1}^q\gamma_j\textbf v_j$. Moreover, $\textbf{v}_{\mathrm{LDA}} = \pm \textbf v_l$ for some $l\in\{1,\dots q\}$ and the corresponding eigenvalue has $\lambda_{\textbf b} =\lambda_l$.

Define now the function $\kappa:(\mathbb{S}^{p-1} \times \mathbb{S}^{q-1})\to\mathbb{R}_+$, such that $\kappa(\beta_1,\dots,\beta_p,\gamma_1,\dots,\gamma_q)=\kappa_{2,\textbf X}(\sum_i\beta_i\textbf u_i,\sum_j\gamma_j\textbf v_j)$ with
$$
\kappa(\boldsymbol{\beta},\boldsymbol{\gamma})=\sum_{i,j}\beta_i^2\gamma_j^2\sigma_i\lambda_j+\alpha_1\alpha_2||\textbf a||^2||\textbf b||^2\beta_{\textbf{a}}^2\gamma_{\textbf{b}}^2,
$$
where $\beta_{\textbf{a}}$ and $\gamma_{\textbf{b}}$ are the coefficients corresponding to $\textbf{u}_{\mathrm{LDA}}$ and $\textbf{v}_{\mathrm{LDA}}$, respectively. Denote now $\delta_{i,j}:=\beta_i\gamma_j$ and observe that $\sum_{i,j}\delta_{i,j}^2=1$. We then the define function $\tilde\kappa: \mathbb{S}^{pq - 1} \to \mathbb{R}_+$ (with the elements $\textbf{z}$ of $\mathbb{S}^{pq - 1}$ double-indexed as $z_{i, j}$, $i = 1, \ldots , p$, $j = 1, \ldots , q$), such that, 
$$
\tilde\kappa(z)=\sum_{i,j}z_{i,j}^2c_{i,j},
$$ 
where $c_{\textbf a,\textbf b}=\sigma_{\textbf a}\lambda_{\textbf b}+\alpha_1\alpha_2||\textbf a||^2||\textbf b||^2$ and $c_{i,j}=\sigma_i\lambda_j$ otherwise. The following is now true:
\begin{itemize}
    \item[1.] $\max_{\textbf z}\tilde\kappa(\textbf z)\geq \max_{\mathbb{\beta},\mathbb{\gamma}}\kappa(\boldsymbol{\beta},\boldsymbol{\gamma})$;
    \item[2.] $\max_{\textbf z}\tilde\kappa(\textbf z)=\max_{i,j}c_{i,j}=c_{i_0,j_0}$, which is obtained by setting $z_{i_0, j_0}=1$ and $z_{i,j}=0$ otherwise;
    \item[3.] $\max_{\textbf z}\tilde\kappa(\textbf z)=\max\{c_{1,1},c_{\textbf a,\textbf b}\}$.
\end{itemize}
Furthermore, taking $\boldsymbol{\beta}=\pm\textbf e_1$ and $\boldsymbol{\gamma}=\pm\textbf e_1$ we have $\kappa(\boldsymbol{\beta},\boldsymbol{\gamma})=\sigma_1\lambda_1$, while by taking $\boldsymbol{\beta}=\pm\textbf e_{\textbf{a}}$ and $\boldsymbol{\gamma}=\pm \textbf e_{\textbf{b}}$ we get $\kappa(\boldsymbol{\beta},\boldsymbol{\gamma})=\sigma_{\textbf{a}}\lambda_{\textbf{b}}+\alpha_1\alpha_2||\textbf a||^2||\textbf b||^2$, where $\textbf e_1$ denotes the first canonical basis vector and $\textbf e_{\textbf{a}},\textbf e_{\textbf{b}}$ are the canonical basis vectors corresponding to the positions of $\sigma_{\textbf{a}}$ and $\lambda_{\textbf{b}}$ in the ordered sequences of eigenvalues of $\textbf A$ and $\textbf B$, respectively. Now, due to the first point above, we conclude that $\max_{\boldsymbol{\beta},\boldsymbol{\gamma}}\kappa(\boldsymbol{\beta},\boldsymbol{\gamma})=\max\{\sigma_1\lambda_1,\sigma_{\textbf a}\lambda_{\textbf a}+\alpha_1\alpha_2||\textbf a||^2||\textbf b||^2\}$ which is obtained either with $(\boldsymbol{\beta},\boldsymbol{\gamma})=(\pm \textbf e_1, \pm \textbf e_1)$ or in $(\boldsymbol{\beta},\boldsymbol{\gamma})=(\pm \textbf{e}_{\textbf{a}}, \pm \textbf{e}_{\textbf{b}})$, depending on the value of the maximum. Thus, $( \pm \textbf{u}_{\mathrm{LDA}}, \pm \textbf{v}_{\mathrm{LDA}})$ are the unique maximizers of $\kappa_{2,\textbf X}$ if and only if $\sigma_1\lambda_1<\sigma_{\textbf a}\lambda_{\textbf b}+\alpha_1\alpha_2||\textbf a||^2||\textbf b||^2$.
\end{proof}

Before proving Theorem \ref{theo:strong_consistency}, we first present several auxiliary results, starting with the uniform consistency of the sample objective functions. We use the notation $f_n \rightrightarrows f$ to denote that the sequence of functions $f_n: \mathcal{X} \rightarrow \mathbb{R} $ defined on a common domain $\mathcal{X} $ converges to $f: \mathcal{X} \rightarrow \mathbb{R}$ uniformly in $x \in \mathcal{X}$, i.e.,
\begin{align*}
    \sup_{x \in \mathcal{X}} | f_n(x) - f(x) | \rightarrow 0
\end{align*}

\begin{lemma}\label{lem:uniform_convergence_objective_functions}
    We have $\kappa_{n\textbf{X}} \rightrightarrows \kappa_{\textbf{X}}$ a.s., and $\psi_{n\textbf{X}} \rightrightarrows \psi_{\textbf{X}}$ a.s.
\end{lemma}

\begin{proof}[Proof of Lemma \ref{lem:uniform_convergence_objective_functions}]
    Because of the centering, we may, without loss of generality, assume that $\mathrm{E}(\textbf{X}) = \textbf{0}$. Starting then with $\kappa_{n\textbf{X}}$, let $m_{nj}: \mathcal{U}_0 \rightarrow \mathbb{R}$ be defined as
    \begin{align*}
        m_{nj}(\textbf{u}, \textbf{v}) := \frac{1}{n} \sum_{i=1}^n \left[ \{ \textbf{u}' ( \textbf{X}_i - \bar{\textbf{X}} ) \textbf{v} \}^j \right],
    \end{align*}
    for $j = 2, 4$. Denote the corresponding population-level functions as $m_j: \mathcal{U}_0 \rightarrow \mathbb{R}$ where $m_j(\textbf{u}, \textbf{v}) := \mathrm{E}\{ (\textbf{u}' \textbf{X} \textbf{v})^j \}$. Now,
    \begin{align}\label{eq:uniform_convergence_objective_functions_1}
    \begin{split}
        & \sup_{(\textbf{u}, \textbf{v}) \in \mathcal{U}_0} | m_{nj}(\textbf{u}, \textbf{v}) - m_{j}(\textbf{u}, \textbf{v}) | \\
        \leq& \sup_{(\textbf{u}, \textbf{v}) \in \mathcal{U}_0} | m_{nj}(\textbf{u}, \textbf{v}) - m_{nj}^*(\textbf{u}, \textbf{v}) | + \sup_{(\textbf{u}, \textbf{v}) \in \mathcal{U}_0} | m_{nj}^*(\textbf{u}, \textbf{v}) - m_{j}(\textbf{u}, \textbf{v}) |,
    \end{split}
    \end{align}
    where $m_{nj}^*(\textbf{u}, \textbf{v}) := (1/n) \sum_{i=1}^n \{ ( \textbf{u}' \textbf{X}_i \textbf{v} )^j \} $ is the uncentered sample function. Now, 
    \begin{align*}
        ( \textbf{u}' \textbf{X}_i \textbf{v} )^j \leq \| \textbf{X}_i \|_2^j,
    \end{align*}
    for $j = 2, 4$ and for all $(\textbf{u}, \textbf{v}) \in \mathcal{U}_0$. As the matrix normal distribution has finite moments of all order, $\mathrm{E}(\| \textbf{X}_i \|_2^j) $ is finite and, hence, by the uniform law of large numbers, the second term on the right-hand side of \eqref{eq:uniform_convergence_objective_functions_1} converges a.s. to zero for $ j = 2, 4$. For the first term and assuming $j = 4$, we write
    \begin{align}\label{eq:centering_removal}
    \begin{split}
        & \sup_{(\textbf{u}, \textbf{v}) \in \mathcal{U}_0} | m_{nj}(\textbf{u}, \textbf{v}) - m_{j}(\textbf{u}, \textbf{v}) | \\
        \leq& \sup_{(\textbf{u}, \textbf{v}) \in \mathcal{U}_0} \left| \textbf{u}'  \bar{\textbf{X}} \textbf{v}  \cdot \frac{1}{n} \sum_{i = 1}^n \sum_{j = 0}^3 \{ \textbf{u}' ( \textbf{X}_i - \bar{\textbf{X}} ) \textbf{v} \}^j \{ \textbf{u}' \textbf{X}_i \textbf{v} \}^{3-j} \right|\\
        \leq& \| \bar{\textbf{X}} \|_2 \sup_{(\textbf{u}, \textbf{v}) \in \mathcal{U}_0} \sum_{j = 0}^3  \left| \frac{1}{n} \sum_{i = 1}^n \{ \textbf{u}' ( \textbf{X}_i - \bar{\textbf{X}} ) \textbf{v} \}^j \{ \textbf{u}' \textbf{X}_i \textbf{v} \}^{3-j} \right|\\
        \leq& \| \bar{\textbf{X}} \|_2 \sum_{j = 0}^3 \frac{1}{n} \sum_{i = 1}^n \| \textbf{X}_i - \bar{\textbf{X}} \|_2^j \| \textbf{X}_i \|_2^{3-j},
    \end{split}
    \end{align}
    which, by the strong law of large numbers, converges a.s. to zero. The convergence for the case $j = 2$ can be shown similarly. Thus, by \eqref{eq:uniform_convergence_objective_functions_1},  $m_{nj} \rightrightarrows m_j $ a.s., $j = 2, 4$.

    Now, by the proof of Theorem \ref{theo:matrix_lda_optimal_mode_wise}, $m_{2}(\textbf{u}, \textbf{v}) = (\textbf{u}' \textbf{A} \textbf{u}) (\textbf{v}' \textbf{B} \textbf{v}) + \alpha_1 \alpha_2 (\textbf{u}' \textbf{H} \textbf{v})^2 $. As $\textbf{A}$ and $\textbf{B}$ are positive-definite, there thus exists $c_1, C_1 > 0$ such that, for all $(\textbf{u}, \textbf{v}) \in \mathcal{U}_0$, we have $c_1 < m_{2}(\textbf{u}, \textbf{v}) < C_1 $. Furthermore, a uniform upper bound, $C_2 > 0$ say, for $m_4(\textbf{u}, \textbf{v})$ can be obtained through the Cauchy-Schwarz inequality. Restrict next to a set $N \subset \Omega$ with $\mathbb{P}(N) = 1$ such that $m_{n2} \rightrightarrows m_2 $ and $m_{n4} \rightrightarrows m_4 $ (that is, the uniform convergences are point-wise instead of almost sure). Then, for a fixed $\omega \in \Omega $ there exists $n_0$ such that, for $n > n_0$, we have $c_1/2 < m_{n2}(\textbf{u}, \textbf{v}) < 2 C_1$ for all $(\textbf{u}, \textbf{v}) \in \mathcal{U}_0$. Thus, for $ n > n_0 $ and abbreviating $m_{n2} \equiv m_{n2} (\textbf{u}, \textbf{v})$ etc., we bound,
    \begin{align}\label{eq:omega_by_omega}
    \begin{split}
         &\sup_{(\textbf{u}, \textbf{v}) \in \mathcal{U}_0} | \kappa_{n\textbf{X}}(\textbf{u}, \textbf{v}) - \kappa_{\textbf{X}}(\textbf{u}, \textbf{v}) | \\
         \leq& \sup_{(\textbf{u}, \textbf{v}) \in \mathcal{U}_0} \frac{1}{m_{n2}^2 m_2^2} \left( m_2^2 | m_{n4} - m_4 | + m_4 | m_{n2}^2 - m_2^2| \right) \\
         \leq& \frac{4}{c_1^4} \left( C_1^2 \sup_{(\textbf{u}, \textbf{v}) \in \mathcal{U}_0}  | m_{n4} - m_4 | + 3 C_1 C_2 \sup_{(\textbf{u}, \textbf{v}) \in \mathcal{U}_0}  | m_{n2} - m_2 | \right),
    \end{split}
    \end{align}
    which converges to zero. The same holds for all $\omega \in N$, finally proving that $\kappa_{n\textbf{X}} \rightrightarrows \kappa_{\textbf{X}}$ a.s.
    
    For $\psi_{n\textbf{X}}$, we first define $\boldsymbol{\Sigma}_n: \mathbb{S}^{p-1} \to \mathbb{R}^{q \times q}$ as
    \begin{align*}
        \boldsymbol{\Sigma}_n(\textbf{u}) = \frac{1}{n} \sum_{i=1}^n ( \textbf{X}_i - \bar{\textbf{X}} )' \textbf{u} \textbf{u}'( \textbf{X}_i - \bar{\textbf{X}} ),
    \end{align*}
     along with its population version $\boldsymbol{\Sigma}: \mathbb{S}^{p-1} \to \mathbb{R}^{q \times q}$, having $ \boldsymbol{\Sigma}(\textbf{u}) = \mathrm{E}(\textbf{X}' \textbf{u} \textbf{u}' \textbf{X})$. Then,
     \begin{align*}
         &\mathrm{sup}_{\textbf{u} \in \mathbb{S}^{p-1}} \| \boldsymbol{\Sigma}_n(\textbf{u}) - \boldsymbol{\Sigma}(\textbf{u}) \|\\
         =& \mathrm{sup}_{\textbf{u} \in \mathbb{S}^{p-1}} \left\| \left\{ \frac{1}{n} \sum_{i=1}^n (\textbf{X}_i - \bar{\textbf{X}})' \otimes (\textbf{X}_i - \bar{\textbf{X}})' - \mathrm{E}(\textbf{X}' \otimes \textbf{X}') \right\} (\textbf{u} \otimes \textbf{u}) \right\|\\
         \leq& \left\| \frac{1}{n} \sum_{i=1}^n (\textbf{X}_i - \bar{\textbf{X}})' \otimes (\textbf{X}_i - \bar{\textbf{X}})' - \mathrm{E}(\textbf{X}' \otimes \textbf{X}') \right\|_2,
     \end{align*}
     where the final step used the relation $\| \textbf{u} \otimes \textbf{u} \| = \| \textbf{u} \|^2$. Invoking now the strong law of large numbers yields $\boldsymbol{\Sigma}_n \rightrightarrows \boldsymbol{\Sigma}$ a.s. (in the Frobenius norm). Now, $\boldsymbol{\Sigma}(\textbf{u}) = (\textbf{u}' \textbf{A} \textbf{u}) \textbf{B} + \alpha_1 \alpha_2 \textbf{H}' \textbf{u} \textbf{u}' \textbf{H}$ is positive definite and there exists $M > 0$ such that $\| \boldsymbol{\Sigma}(\textbf{u})^{-1} \|_2 < M$ for all $\textbf{u} \in \mathbb{S}^{p-1}$. Using next $ \boldsymbol{\Sigma}_n(\textbf{u})^{-1} -  \boldsymbol{\Sigma}(\textbf{u})^{-1} =  \boldsymbol{\Sigma}(\textbf{u})^{-1} ( \boldsymbol{\Sigma}(\textbf{u}) - \boldsymbol{\Sigma}_n(\textbf{u}) ) \boldsymbol{\Sigma}(\textbf{u})_n^{-1}$ and reasoning as in \eqref{eq:omega_by_omega} further shows that $\boldsymbol{\Sigma}^{-1}_n \rightrightarrows \boldsymbol{\Sigma}^{-1}$ a.s. (in the Frobenius norm)
     
     Define then,
     \begin{align*}
         \psi_{n\textbf{X}}^1(\textbf{u}) &:= \frac{1}{n} \sum_{i=1}^n \left\{ \textbf{u}' ( \textbf{X}_i - \bar{\textbf{X}} ) \boldsymbol{\Sigma}(\textbf{u})^{-1} ( \textbf{X}_i - \bar{\textbf{X}} )' \textbf{u} \right\}^2,\\
         \psi_{n\textbf{X}}^2(\textbf{u}) &:= \frac{1}{n} \sum_{i=1}^n \left\{ \textbf{u}' (\textbf{X}_i - \bar{\textbf{X}}) \boldsymbol{\Sigma}(\textbf{u})^{-1} \textbf{X}_i' \textbf{u} \right\}^2,\\
         \psi_{n\textbf{X}}^3(\textbf{u}) &:= \frac{1}{n} \sum_{i=1}^n \left( \textbf{u}' \textbf{X}_i \boldsymbol{\Sigma}(\textbf{u})^{-1} \textbf{X}_i' \textbf{u} \right)^2.
     \end{align*}
With a techniques similar to \eqref{eq:centering_removal}, one can show that
\begin{align*}
    \sup_{\textbf{u} \in \mathbb{S}^{p-1}} | \psi_{n\textbf{X}}^2(\textbf{u}) - \psi_{n\textbf{X}}^3(\textbf{u}) | \rightarrow 0 \quad \sup_{\textbf{u} \in \mathbb{S}^{p-1}} | \psi_{n\textbf{X}}^1(\textbf{u}) - \psi_{n\textbf{X}}^2(\textbf{u}) | \rightarrow 0,
\end{align*}
almost surely. Now, $( \textbf{u}' \textbf{X} \boldsymbol{\Sigma}(\textbf{u})^{-1} \textbf{X}' \textbf{u} )^2$ is dominated by the integrable random variable $M^2 \| \textbf{X} \|_2^4 $ and the map $\textbf{X} \mapsto ( \textbf{u}' \textbf{X} \boldsymbol{\Sigma}(\textbf{u})^{-1} \textbf{X}' \textbf{u} )^2$ is continous (since $\boldsymbol{\Sigma}(\textbf{u})$ is positive definite with the uniform lower bound $M^{-1}$ on its smallest eigenvalue). Hence, the uniform law of large numbers gives $\psi_{n\textbf{X}}^3 \rightrightarrows \psi_\textbf{X}$ a.s. and the above chain of convergences further implies that $\psi_{n\textbf{X}}^1 \rightrightarrows \psi_\textbf{X}$ a.s. The claim is thus proven once we show that $ \sup_{\textbf{u} \in \mathbb{S}^{p-1}} | \psi_{n\textbf{X}}^1(\textbf{u}) - \psi_{n\textbf{X}}(\textbf{u}) | \rightarrow 0 $ a.s. To see this we write, with the notation $\textbf{Y}_i := \textbf{X}_i - \bar{\textbf{X}}$, that,
\begin{align*}
  &\sup_{\textbf{u} \in \mathbb{S}^{p-1}} | \psi_{n\textbf{X}}^1(\textbf{u}) - \psi_{n\textbf{X}}(\textbf{u}) | \\
  \leq& \sup_{\textbf{u} \in \mathbb{S}^{p-1}} \frac{1}{n}\sum_{i=1}^n \left| [ \textbf{u}' \textbf{Y}_i \{ \boldsymbol{\Sigma}(\textbf{u})^{-1} -  \boldsymbol{\Sigma}_n(\textbf{u})^{-1} \} \textbf{Y}_i' \textbf{u} ] [ \textbf{u}' \textbf{Y}_i \{ \boldsymbol{\Sigma}(\textbf{u})^{-1} +  \boldsymbol{\Sigma}_n(\textbf{u})^{-1} \} \textbf{Y}_i' \textbf{u} ] \right| \\
  \leq& \frac{1}{n}\sum_{i=1}^n \| \textbf{Y}_i \|_2^4 \| \boldsymbol{\Sigma}(\textbf{u})^{-1} -  \boldsymbol{\Sigma}_n(\textbf{u})^{-1} \|_2 (\| \boldsymbol{\Sigma}(\textbf{u})^{-1} \|_2 + \|  \boldsymbol{\Sigma}_n(\textbf{u})^{-1} \|_2),
\end{align*}
which converges to zero almost surely, thus proving the almost sure uniform convergence of $\psi_{n\textbf{X}}$ to $\psi_\textbf{X}$.

\end{proof}

\begin{lemma}\label{lem:composition_uniform_convergence}
Let $(X,d_X)$, $(Y,d_Y)$ be compact metric spaces and let $(Z,d_Z)$ a metric space. Let $f_n:Y\to Z$ and $g_n:X\to Y$ be sequences of functions such that $f_n \rightrightarrows f$ and $g_n \rightrightarrows g$ for some $f:Y \to Z$ and $g:X \to Y$ such that $f$ is continuous. Then $f_n \circ g_n \rightrightarrows f\circ g$.
\end{lemma}
\begin{proof}[Proof of Lemma \ref{lem:composition_uniform_convergence}]
We have,
\begin{align}\label{eq:composition_uniform_convergence_1}
\begin{split}
    &\sup_{x\in X} \|f_n\circ g_n(x)-f\circ g(x)\|\\
    \leq& \sup_{x\in X} \|f_n\circ g_n(x)-f\circ g_n(x)\|+\sup_{x\in X} \|f\circ g_n(x) -f\circ g(x)\|
\end{split}
\end{align}
The first supremum in \eqref{eq:composition_uniform_convergence_1} has
$$
\sup_{x\in X} \|f_n\circ g_n(x)-f\circ g_n(x)\|\leq \sup_{y\in Y} \|f_n(y)-f(y)\|\rightarrow 0,
$$
since $f_n \rightrightarrows f$. Now, we show that the second supremum in \eqref{eq:composition_uniform_convergence_1} converges to $0$. First observe that $f$ is a continuous function defined on a compact set and therefore uniformly continuous. Hence, for every $\varepsilon>0$ there exists $\delta>0$ such that, for every $y_1, y_2 \in Y$, $d_Y(y_1,y_2)<\delta $ implies $ d_Z(f(y_1),f(y_2))<\varepsilon$. Moreover, having $g_n \rightrightarrows g$ implies that, for any $\delta>0$, there exists $n_0\in\mathbb{N}$, such that, for $n>n_0$, we have $d_Y(g_n(x),g(x))<\delta$ for all $x\in X$. Taking now $n>n_0$, we have for all $x \in X$ that $d_Y(g_n(x),g(x))<\delta$, implying that $d_Z(f(g_n(x)),f(g(x))) < \varepsilon$ for all $x\in X$. Thus, $(f\circ g_n) \rightrightarrows (f\circ g)$ and we have hence shown that $\sup_{x\in X} \|f_n\circ g_n(x)-f\circ g(x)\|\rightarrow 0$.
\end{proof}

\begin{lemma}\label{lem:convergence_of_maximizer}
Let $K\subset\mathbb{R}^d$ be a compact set and let $f_n:K\to\mathbb{R}$ be a sequence of random functions such that $f_n \rightrightarrows f$ a.s. for some continuous function $f:K\to \mathbb{R}$. Let $K_1\subset K$ be such that each $f_n$ has a.s. maximizer (not necessarily unique) in $K_1$, and $f$ has a unique maximizer $x_0$ in $K_1$. Then, for any sequence of maximizers $x_n$, we have $x_n\to_{a.s.}x$.
\end{lemma}

\begin{proof}[Proof of Lemma \ref{lem:convergence_of_maximizer}]
There exists a set $N \subset \Omega$ with $\mathbb{P}(N) = 1$ such that $f_n \rightrightarrows f$ in $N$ (that is, the uniform convergence is point-wise instead of almost sure) and each $f_n $ has a maximizer in $N$.

Consider then, for a fixed $\omega \in N$, an aribtrary sequence $(x_n)$ of maximizers of $f_n$ in $K_1$. As a compact set $K$ is bounded, making $(x_n)$ a bounded sequence. Hence, $(x_n)$ has a convergent subsequence, $(x_{n_k})$ say, such that $x_{n_k} \rightarrow x'$ as $k \rightarrow \infty$. Then, as,
\begin{align*}
    |f_{n_k}(x_{n_k}) - f(x') | \leq \sup_{x \in K} |f_{n_k}(x) - f(x)| + |f(x_{n_k}) - f(x')|,
\end{align*}
the continuity of $f$ shows that $f_{n_k}(x_{n_k}) \rightarrow f(x')$ as $k \rightarrow \infty$. Now, since for any $x\in K_1$, $f_{n_k}(x_{n_k})\geq f_{n_k}(x)$, then the same holds in the limit as well That is, $f(x') \geq f(x)$, for all $x\in K_1$, and, by the uniqueness of the maximizer of $f$ in $K_1$, we must thus have $x' = x_0$. Now, since the subsequence $(x_{n_k})$ was chosen arbitrarily, the statement holds for any convergent subsequence of $(x_n)$. Therefore, we conclude that every convergent subsequence of the bounded sequence $(x_n)$ converges to $x_0$, the global maximizer of $f$ in $K_1$. Hence, $x_n \rightarrow x_0 $. Since this holds for all $\omega \in N$, the claim is true.
\end{proof}

\begin{proof}[Proof of Theorem \ref{theo:strong_consistency}]
    We show the claim only for $\kappa_{n\textbf{X}}$ as the result for $\psi_{n\textbf{X}}$ follows after a straightforward adaptation of the following technique to single-argument functions.
    
    Denote the non-zero singular values of $ \textbf{A}^{-1/2} (\textbf{T}_2 - \textbf{T}_1) \textbf{B}^{-1/2} $ by $\sigma_1, \ldots , \sigma_d$ and fix $(\textbf{u}_{0j}, \textbf{v}_{0j}) $, $j = 1, \ldots, d$ as any collection of the corresponding singular pairs. Then, by the proof of Theorem \ref{theo:matrix_lda_optimal_higher_rank}, the sequences of $(\mathcal{G}_{1,\textbf{X}}, \mathcal{G}_{2,\textbf{X}})$-maximizers of $\kappa_\textbf{X}$ are precisely all collections of pairs
    \begin{align*}
	(\textbf{u}_j, \textbf{v}_j) = \left( s_{\textbf{u}j} \frac{\textbf{A}^{-1/2} \textbf{u}_{0j}}{\|\textbf{A}^{-1/2} \textbf{u}_{0j}\|}, s_{\textbf{v}j} \frac{\textbf{B}^{-1/2} \textbf{v}_{0j}}{\|\textbf{B}^{-1/2} \textbf{v}_{0j}\|} \right),
	\end{align*}
	where $ s_{\textbf{u}j}, s_{\textbf{v}j} \in \{ -1, 1 \} $, $j = 1, \ldots , d$. To obtain a unique representatives for the sequences of maximizers, we restrict ourselves to the following subset of $\mathcal{U}_0$.
    \begin{align*}
    \mathcal{U}_1 := \{ (\textbf{u}, \textbf{v}) \in \mathcal{U}_0 \mid & \mbox{the first non-zero element of } \textbf{u} \mbox{ is positive, }\\
    & \mbox{the first non-zero element of } \textbf{v} \mbox{ is positive} \}.
    \end{align*}
    In $\mathcal{U}_1$ there is thus a unique sequence of $(\mathcal{G}_{1,\textbf{X}}, \mathcal{G}_{2,\textbf{X}})$-maximizers of $\kappa_\textbf{X}$, denoted hereafter as $(\textbf{u}_1, \textbf{v}_1), \ldots, (\textbf{u}_d, \textbf{v}_d) $.
    
    We first show that there exists a.s., for all $n > pq$, a sequence $(\textbf{u}_{n1}, \textbf{v}_{n1}), \ldots , (\textbf{u}_{nd}, \textbf{v}_{nd})$ of $(\mathcal{G}_{n1,\textbf{X}}, \mathcal{G}_{n2,\textbf{X}})$-maximizers of $\kappa_\textbf{X}$ in $\mathcal{U}_1$ (note that any maximizer in $\mathcal{U}_0$ can be brought to $\mathcal{U}_0$ with a suitable change of signs). Consider, for a fixed $n$, the first pair $(\textbf{u}_{n1}, \textbf{v}_{n1})$: the function $\kappa_{nX}:\mathcal{U}_1 \rightarrow \mathbb{R} $ is continuous everywhere on its domain as long as $ (1/n) \sum_{i=1}^n (\textbf{X}_i - \bar{\textbf{X}}) \textbf{v} \textbf{v}' (\textbf{X}_i - \bar{\textbf{X}})' $ is a positive definite matrix for all $\textbf{v} \in \mathbb{R}^{q - 1}$, see Section \ref{sec:well_defined}. To see that this holds almost surely, we write,
    \begin{align*}
        \textbf{u}' \frac{1}{n} \sum_{i=1}^n (\textbf{X}_i - \bar{\textbf{X}}) \textbf{v} \textbf{v}' (\textbf{X}_i - \bar{\textbf{X}})' \textbf{u} = (\textbf{v} \otimes \textbf{u})' \left\{ \frac{1}{n} \sum_{i=1}^n \{ (\textbf{z}_i - \bar{\textbf{z}}) (\textbf{z}_i - \bar{\textbf{z}})' \right\} (\textbf{v} \otimes \textbf{u}),
    \end{align*}
    where $\textbf{z}_i := \mathrm{vec}(\textbf{X}_i)$. Now, since the distribution of $\textbf{X}_i$ admits a density (w.r.t. the Lebesgue measure), the sample covariance matrix $(1/n) \sum_{i=1}^n (\textbf{z}_i - \bar{\textbf{z}}) (\textbf{z}_i - \bar{\textbf{z}})'$ is a.s. positive definite as soon as $n > pq$. 
    
    Thus (for large enough $n$), $\kappa_{nX}:\mathcal{U}_1 \rightarrow \mathbb{R} $ is an a.s. continuous function on a compact domain and, hence, there exists a.s. a maximizing pair $(\textbf{u}_{n1}, \textbf{v}_{n1})$. The second pair is obtained through maximization of the restriction of $\kappa_{n\textbf{X}} $ to $\mathcal{U}_1 \cap \mathcal{C}_{n1} $, where
    \begin{align*}
        \mathcal{C}_{n1} := \{ (\textbf{u}, \textbf{v}) \in \mathbb{R}^{p + q} \mid \textbf{u}'\textbf{G}_{n11,\textbf{X}} \textbf{u}_{n1} = 0, \textbf{v}'\textbf{G}_{n21,\textbf{X}} \textbf{v}_{n1} = 0  \}.
    \end{align*}
    Now, the set $\mathcal{U}_1 \cap \mathcal{C}_{n1}$ is compact as the intersection of a closed and a compact set, and reasoning as in the first step, there exists (for the given $n$) a.s. a pair $(\textbf{u}_{n2}, \textbf{v}_{n2})$ maximizing $\kappa_\textbf{X}$ in $\mathcal{U}_1 \cap \mathcal{C}_{n1} $. The a.s. existence of the remaining pairs (for the given $n$) is shown similarly. Finally, as countable intersections of almost sure sets are almost sure, the sequence of $(\mathcal{G}_{n1,\textbf{X}}, \mathcal{G}_{n2,\textbf{X}})$-maximizers exists for all $n > pq$ a.s.
    
    The almost sure convergence $\textbf{u}_{n1} \rightarrow \textbf{u}_1 $ follows now from Lemmas \ref{lem:uniform_convergence_objective_functions} and \ref{lem:convergence_of_maximizer} by observing that $\kappa_\textbf{X}$ is now a continuous function because $\textbf{X}$ admits a density, see the proof of Lemma \ref{lem:existence} and Appendix \ref{sec:well_defined}. Since an arbitrary element of $\mathcal{U}_0$ can be brought to $\mathcal{U}_1$ by changing the signs of its $\textbf{u}$- and \textbf{v}-parts suitably, we have thus shown the claim for the first pair of maximizers.
    
    We prove the remaining almost sure convergences inductively. That is, assume that $(\textbf{u}_{nj},\textbf{v}_{nj})\rightarrow (\textbf{u}_{j},\textbf{v}_{j})$ a.s., for $j = 1, \ldots , k - 1$. The pair is found as a maximizer of $\kappa_{n\textbf{X}}$ in $\mathcal{U}_1 \cap \mathcal{C}_{nk} $, where
    \begin{align*}
        \mathcal{C}_{nk} := \{ (\textbf{u}, \textbf{v}) \in \mathbb{R}^{p + q} \mid \textbf{u}'\textbf{G}_{n1j} \textbf{u}_{nj} = 0, \textbf{v}'\textbf{G}_{n2j} \textbf{v}_{nj} = 0, \quad j = 1, \ldots , k - 1  \}.
    \end{align*}
    Define also the population counterpart of $\mathcal{C}_{nk}$ as
    \begin{align*}
        \mathcal{C}_{k} := \{ (\textbf{u}, \textbf{v}) \in \mathbb{R}^{p + q} \mid \textbf{u}'\textbf{G}_{1j} \textbf{u}_{j} = 0, \textbf{v}'\textbf{G}_{2j} \textbf{v}_{j} = 0, \quad j = 1, \ldots , k - 1  \}.
    \end{align*}
    Observe still that, for any $j = 1, \ldots , k - 1$, the law of large numbers gives $\textbf{G}_{n1j} \rightarrow \textbf{G}_{1j}$ a.s. and, hence, that $\textbf{G}_{n1j} \textbf{u}_{nj} \rightarrow \textbf{G}_{1j} \textbf{u}_j$ a.s., and similarly for $\textbf{v}$.
    
    We being by constructing explicit bases for the subspaces $\mathcal{C}_{nk}$ and $\mathcal{C}_{k}$ of $\mathbb{R}^{p + q}$. By the proof of Theorem \ref{theo:matrix_lda_optimal_higher_rank}, we have
    \begin{align*}
        \textbf{G}_{1j} \textbf{u}_{j} = s_{\textbf{u}j} \|\textbf{A}^{-1/2} \textbf{u}_{0j}\|^{-1} \|\textbf{B}^{-1/2} \textbf{v}_{0j}\|^{-2} (1 + \beta \sigma_j^2) \textbf{A}^{1/2} \textbf{u}_{0j},
    \end{align*}
    showing that the vectors $\textbf{G}_{11,\textbf{X}} \textbf{u}_{1}, \ldots, \textbf{G}_{1(k-1)} \textbf{u}_{k-1} \in \mathbb{R}^p$ are linearly independent. Let now $\textbf{J}_1 := (\textbf{G}_{11,\textbf{X}} \textbf{u}_{1} \mid \cdots \mid \textbf{G}_{1(k-1)} \textbf{u}_{k-1} \mid \textbf{r}_k \mid \cdots \mid \textbf{r}_p) \in \mathbb{R}^{p \times p}$ where are $\textbf{r}_k, \ldots , \textbf{r}_p$ are some fixed vectors which make $\textbf{J}_1$ have full rank. Apply now Gram-Schmidt orthogonalization to $\textbf{J}_1$ and denote the orthonormal matrix of the last $p - k + 1$ obtained vectors as $\textbf{K}_1 \in \mathbb{R}^{p \times (p - k + 1)}$. Then any $\textbf{u} \in \mathrm{col}(\textbf{K}_1) $ has $ \textbf{u} \in \mathcal{C}_{1k} := \{ \textbf{u} \in \mathbb{R}^p \mid \textbf{u}'\textbf{G}_{1j} \textbf{u}_{j} = 0, \quad  j = 1, \ldots , k - 1 \}$ and the columns of $\textbf{K}_1$ form a basis of $\mathcal{C}_{1k}$. A basis for $\mathcal{C}_{k}$ is now obtained by carrying out the above construction also for $\textbf{v}$ as
    \begin{align*}
        \textbf{J} = \begin{pmatrix} \textbf{K}_1 & \textbf{0} \\
        \textbf{0} & \textbf{K}_2
        \end{pmatrix} \in \mathbb{R}^{(p + q) \times (p + q - 2k + 2)}.
    \end{align*}
    
    For the remainder of the proof, we work with a fixed probability element $\omega \in N$ where the set $N$ has $\mathbb{P}(N) = 1$ and is such that all the almost sure convergences in the previous paragraph hold point-wise in $N$, the maximizer $(\textbf{u}_{nk}, \textbf{v}_{nk})$ exists for all $n \in \mathbb{N}$ in $N$ and the sample objective function $\kappa_n$ is continuous for all $n > pq$ in $N$.
    
    We next construct bases for the subspaces $\mathcal{C}_{nk} $. For $n$ large enough, the vectors $\textbf{G}_{n11,\textbf{X}} \textbf{u}_{n1}, \ldots, \textbf{G}_{n1(k-1)} \textbf{u}_{n(k-1)}, \textbf{r}_k, \ldots, \textbf{r}_p \in \mathbb{R}^p$ form a linearly independent set (this happens as a consequence of the continuity of determinant) and we thus construct $\textbf{K}_{n1}$ analogously to its population counterpart. Constructing $\textbf{K}_{n2}$ similarly, we obtain the orthonormal matrices whose columns give (for $n$ large enough) bases for the $ \mathcal{C}_{nk}$,
    \begin{align*}
        \textbf{J}_{n} = \begin{pmatrix} \textbf{K}_{n1} & \textbf{0} \\
        \textbf{0} & \textbf{K}_{n2}
        \end{pmatrix} \in \mathbb{R}^{(p + q) \times (p + q - 2k + 2)}.
    \end{align*}
    Now, as the Gram-Schmidt process constitutes of additions and other basic vector space operations, we have $\textbf{J}_n \rightarrow \textbf{J}$.
    
    Denote $\mathcal{W}_0 := \mathbb{S}^{p - k + 1} \times \mathbb{S}^{q - k + 1}$ and let $g:\mathcal{W}_0 \rightarrow (\mathcal{U}_0 \cap \mathcal{C}_k) $ and $g_n:\mathcal{W}_0 \rightarrow (\mathcal{U}_0 \cap \mathcal{C}_k) $ be defined, respectively, as $g(\textbf{s}, \textbf{t}) = \textbf{J} (\textbf{s}', \textbf{t}')'$ and  $g_n(\textbf{s}, \textbf{t}) = \textbf{J}_n (\textbf{s}', \textbf{t}')'$ (observe that $g$ indeed maps the Cartesian product of unit spheres into the Cartesian product of unit spheres). Define the compositions $f := (\kappa_\textbf{X} \circ g): \mathcal{W}_0 \rightarrow \mathbb{R} $ and $f_n := (\kappa_{n\textbf{X}} \circ g_n): \mathcal{W}_0 \rightarrow \mathbb{R} $. Now, $g_n \rightrightarrows g$ as
    \begin{align*}
        \sup_{(\textbf{s}, \textbf{t}) \in \mathcal{W}_0} \| g(\textbf{s}, \textbf{t}) - g_n(\textbf{s}, \textbf{t}) \| \leq \| \textbf{J}_n - \textbf{J} \|_2 \rightarrow 0.
    \end{align*}
    Applying now Lemma \ref{lem:composition_uniform_convergence}, we see that $f_n \rightrightarrows f$.
    
    Let next,
    \begin{align*}
        \mathcal{W}_1 := \{ (\textbf{s}, \textbf{t}) \in \mathcal{W}_0 \mid & \mbox{the first non-zero elements of the $p$ and $q$-dimensional}\\
        &\mbox{subvectors of } g(\textbf{s}, \textbf{t}) \mbox{ are positive} \}.
    \end{align*}
    Now, since $(\textbf{u}_k, \textbf{v}_k) $ is the unique maximizer of $\kappa_\textbf{X}$ in $\mathcal{U}_1 \cap \mathcal{C}_{k} $, the function $f$ is maximized in $\mathcal{W}_1$ uniquely by the vector $(\textbf{s}, \textbf{t}) \in \mathcal{W}_1$ satisfying $g(\textbf{s}, \textbf{t}) = (\textbf{u}_k, \textbf{v}_k)$. Similarly, to the sample maximizer $(\textbf{u}_{nk}, \textbf{v}_{nk})$ there corresponds $(\textbf{s}_n, \textbf{t}_n)$ such that $g_n(\textbf{s}_n, \textbf{t}_n) = (\textbf{u}_{nk}, \textbf{v}_{nk})$ and such that $(s_{n{u}k} \textbf{s}_n, s_{n{v}k} \textbf{t}_n)$ is, for some signs $s_{n{u}k}, s_{n{v}k} \in \{ -1, 1 \}$, a maximizer of $f_n$ in $\mathcal{W}_1$ (the signs are needed to guarantee that the maximizer in indeed a member of $\mathcal{W}_1$). Reasoning now as in the proof of Theorem \ref{lem:convergence_of_maximizer}, we obtain $(s_{n{u}k} \textbf{s}_n, s_{n{v}k} \textbf{t}_n) \rightarrow (\textbf{s}, \textbf{t}) $ which, in turn, implies that
    \begin{align*}
        \begin{pmatrix}
        s_{n{u}k} \textbf{u}_{nk} \\
        s_{n{v}k} \textbf{v}_{nk}
        \end{pmatrix} 
        =
        \textbf{J}_n \begin{pmatrix}
        s_{n{u}k} \textbf{s}_n \\
        s_{n{v}k} \textbf{t}_n
        \end{pmatrix} \rightarrow
        \textbf{J} \begin{pmatrix}
        \textbf{s} \\
        \textbf{t}
        \end{pmatrix} = 
        \begin{pmatrix}
        \textbf{u}_{k} \\
        \textbf{v}_{k}
        \end{pmatrix}.
    \end{align*}
    Since the above holds for all $\omega \in N$, we have 
    \begin{align*}
        s_{n{u}j} \textbf{u}_{nk} \rightarrow  \textbf{u}_{k} \quad \mbox{and} \quad s_{n{v}k} \textbf{v}_{nk} \rightarrow \textbf{v}_k,
    \end{align*}
    almost surely, completing the proof.
    
\end{proof}


\begin{proof}[Proof of Corollary \ref{cor:strong_consistency}]
    Let, without loss of generality, $j = 1$. We restrict to an almost sure set $N$ of elements of $\Omega$ where the uniform convergence of the sample objective function in Lemma \ref{lem:uniform_convergence_objective_functions} is point-wise instead of almost sure and where $(s_{n{u}1} \textbf{u}_{n1}, s_{n{v}1} \textbf{v}_{n1}) \rightarrow (\textbf{u}_{1}, \textbf{v}_{1})$ as $n \rightarrow \infty$. 
    
    Then,
    \begin{align*}
        &| \kappa_{n\textbf{X}}(\textbf{u}_{n1}, \textbf{v}_{n1}) - \kappa_\textbf{X}(\textbf{u}_{1}, \textbf{v}_{1}) |\\
        =& | \kappa_{n\textbf{X}}(s_{n{u}1} \textbf{u}_{n1}, s_{n{v}1} \textbf{v}_{n1}) - \kappa_\textbf{X}(\textbf{u}_{1}, \textbf{v}_{1}) |\\
        \leq& \sup_{(\textbf{u}, \textbf{v}) \in \mathcal{U}_0} | \kappa_{n\textbf{X}}(\textbf{u}, \textbf{v}) - \kappa_{\textbf{X}}(\textbf{u}, \textbf{v}) | + | \kappa_{\textbf{X}}(s_{n{u}1} \textbf{u}_{n1}, s_{n{v}1} \textbf{v}_{n1}) - \kappa_{\textbf{X}}(\textbf{u}_{1}, \textbf{v}_{1}) |,
    \end{align*}
    where the first term goes to zero by Lemma \ref{lem:uniform_convergence_objective_functions} and the second one by the continuity of $\kappa_\textbf{X}$. Hence, $\kappa_{n\textbf{X}}(\textbf{u}_{n1}, \textbf{v}_{n1}) \rightarrow \kappa_\textbf{X}(\textbf{u}_{1}, \textbf{v}_{1}) $, a.s., which further implies that $\theta_{n1} \rightarrow \theta_1$, a.s., and $\lambda_{n1} \rightarrow \lambda_1 = \sigma_1$, a.s. 
    
    By the law of large of numbers, we have $z_{n12} \rightarrow \mathrm{E} ( [ \textbf{u}_1 \{ \textbf{X} - \mathrm{E}(\textbf{X})  \} \textbf{v}_1  ]^2 )$, a.s., and $s_{nu1} s_{nv1} z_{n13} \rightarrow \mathrm{E} ( [ \textbf{u}_1 \{ \textbf{X} - \mathrm{E}(\textbf{X})  \} \textbf{v}_1  ]^3 )$, a.s. Hence, the sign of $s_{nu1} s_{nv1} (\alpha_1 - \alpha_2)^{-1} z_{n13}$ converges almost surely to the sign of $  (\alpha_1 - \alpha_2)^{-1} \mathrm{E} ( [ \textbf{u}_j' \{ \textbf{X} - \mathrm{E}(\textbf{X})  \} \textbf{v}_j  ]^3 ) $, implying that $s_{nu1} s_{nv1} s_{n1} \rightarrow s_1$, a.s., where $s_1$ is an in Theorem \ref{theo:matrix_lda_optimal_higher_rank}. Plugging now everything in to the definition of $\textbf{W}_{n\mathrm{LDA}}$ and invoking the decomposition of $\textbf{W}_{\mathrm{LDA}}$ given in the statement of Theorem~\ref{theo:matrix_lda_optimal_higher_rank} yields the claim.
\end{proof}

\section{Alternative algorithm for optimizing $\kappa_{n\textbf{X}}$}\label{sec:algorithm_appendix}

The following is an alternative, fixed-point algorithm for the optimization of $\kappa_{n\textbf{X}}$.

\begin{algorithm}[H]
\caption{Fixed-point algorithm for the optimization of $\kappa_{n\textbf{X}}$.}\label{alg::kappa_fixed_point}
\SetKwInOut{Input}{Input}
        \Input{$\textbf{X}_1,\dots \textbf{X}_n\in\mathbb{R}^{p \times q}$ centered observations;}
        \BlankLine
	Allocate $\textbf{U},\,\textbf{V},\,\textbf{G}_{\textbf{u}},\,\textbf{G}_\textbf{v}$;\\
	Initialize $\textbf{v}_0$, $\|\textbf{v}_0\|=1$;\\
 	Set the tolerance $\varepsilon>0$ and $e=\varepsilon+1$;\\
 	Initialize $\textbf{G}_{\textbf{u}}^{\perp}\leftarrow\textbf{I}_p$, $\textbf{G}_{\textbf{v}}^{\perp}\leftarrow\textbf{I}_q$;\\ 
 	Calculate the projections $\textbf{y}_{\textbf{v},i} = \textbf{X}_i\textbf{v}_0$, $i=1,\dots,n$;\\  
    Calculate $\textbf{u}_0$ as unit length mimimizer of $\kappa_{n\textbf{y}_\textbf{v}}$;\\	
    
	\While{$\min\{p,q\}>1$}{
	$p\leftarrow p-1\,$; $q\leftarrow q-1$;\\
	
	Set $\textbf{X}_{temp,i}\leftarrow \textbf{X}_i$, $i=1,\dots,n$;\\
	Project $\textbf{X}_i\leftarrow\textbf{G}_{\textbf{u}}^{\perp}\textbf{X}_i\textbf{G}_{\textbf{v}}^{\perp}$, $i=1,\dots,n$;\\

    \While{$e>\varepsilon$}{

	Calculate the projections $\textbf{y}_{\textbf{u},i}=\textbf{X}_i'\textbf{u}_0$, $i=1,\dots,n$;\\  
    Calculate $\textbf{v}_1$ as unit length optimizer of $\kappa_{n\textbf{y}_\textbf{u}}$;\\
    
    Calculate the projections $\textbf{y}_{\textbf{v},i}=\textbf{X}_i\textbf{v}_1$, $i=1,\dots,n$;\\  
    Calculate $\textbf{u}_1$ as unit length optimizer of $\kappa_{n\textbf{y}_\textbf{v}}$;\\	
    
    $e_\textbf{u}=\min\{\|\textbf{u}_0-\textbf{u}_1\|^2,\|\textbf{u}_0+\textbf{u}_1\|^2\}\,$; $e_\textbf{v}=\min\{\|\textbf{v}_0-\textbf{v}_1\|^2,\|\textbf{v}_0+\textbf{v}_1\|^2\}$;\\
    $e\leftarrow e_\textbf{u}+e_\textbf{v}$;\\
    
    $(\textbf{u}_0,\textbf{v}_0)\leftarrow (\textbf{u}_1,\textbf{v}_1)$;\\
    
	  }
	
    Append $\textbf{U}\leftarrow[\textbf{U},\textbf{G}_{\textbf{u}}^{\perp}\textbf{u}_1]$,  $\textbf{V}\leftarrow[\textbf{V},\textbf{G}_{\textbf{v}}^{\perp}\textbf{v}_1]$;\\
    
   Append $\textbf{G}_\textbf{u}\leftarrow[\textbf{G}_\textbf{u},\frac{1}{n}\sum_{i=1}^n\textbf{u}_1'\textbf{X}_i\textbf{v}_1 \cdot \textbf{X}_i\textbf{v}_1]$, $\textbf{G}_\textbf{v}\leftarrow[\textbf{G}_\textbf{v},\frac{1}{n}\sum_{i=1}^n\textbf{u}_1'\textbf{X}_i\textbf{v}_1 \cdot \textbf{X}_i' \textbf{u}_1]$;\\
    
    Calculate the orthogonal complements $\textbf{G}_\textbf{u}^\perp$, $\textbf{G}_\textbf{v}^\perp$;\\
    
    Set $\textbf{X}_i\leftarrow\textbf{X}_{temp,i}$, $i=1,\dots,n$;
    }
     
	Return $(\textbf{U},\,\textbf{V})$;
\end{algorithm}

\bibliographystyle{apalike}
\bibliography{refs}

\begin{thebibliography}{}

\bibitem[Barzilai and Borwein, 1988]{Barzilai1988}
Barzilai, J. and Borwein, J. (1988).
\newblock Two-point step size gradient methods.
\newblock {\em IMA Journal of Numerical Analysis}, 8(1):141--148.

\bibitem[Beckmann and Smith, 2005]{beckmann2005tensorial}
Beckmann, C.~F. and Smith, S.~M. (2005).
\newblock Tensorial extensions of independent component analysis for
  multisubject {fMRI} analysis.
\newblock {\em Neuroimage}, 25(1):294--311.

\bibitem[Bickel et~al., 2018]{bickel2018projection}
Bickel, P.~J., Kur, G., and Nadler, B. (2018).
\newblock Projection pursuit in high dimensions.
\newblock {\em Proceedings of the National Academy of Sciences},
  115(37):9151--9156.

\bibitem[Bolton and Krzanowski, 2003]{bolton2003projection}
Bolton, R. and Krzanowski, W. (2003).
\newblock Projection pursuit clustering for exploratory data analysis.
\newblock {\em Journal of Computational and Graphical Statistics},
  12(1):121--142.

\bibitem[Caussinus and Ruiz, 1990]{Ruiz1990}
Caussinus, H. and Ruiz, A. (1990).
\newblock Interesting projections of multidimensional data by means of
  generalized principal component analyses.
\newblock In Momirovi{\'{c}}, K. and Mildner, V., editors, {\em Compstat},
  pages 121--126, Heidelberg. Physica-Verlag HD.

\bibitem[Davies, 1987]{Davies1987}
Davies, P.~L. (1987).
\newblock Asymptotic behaviour of {S}-estimates of multivariate location
  parameters and dispersion matrices.
\newblock {\em Annals of Statistics}, 15(3):1269--1292.

\bibitem[Diaconis and Freedman, 1984]{diaconis1984asymptotics}
Diaconis, P. and Freedman, D. (1984).
\newblock Asymptotics of graphical projection pursuit.
\newblock {\em Annals of Statistics}, 12:793--815.

\bibitem[Fischer et~al., 2019]{FischerBerroNordhausenRuizGazen2019}
Fischer, D., Berro, A., Nordhausen, K., and Ruiz-Gazen, A. (2019).
\newblock {REPPlab}: An {R} package for detecting clusters and outliers using
  exploratory projection pursuit.
\newblock {\em Communications in Statistics - Simulation and Computation}, to
  appear:1--23.

\bibitem[Friedman and Tukey, 1974]{friedman1974projection}
Friedman, J.~H. and Tukey, J.~W. (1974).
\newblock A projection pursuit algorithm for exploratory data analysis.
\newblock {\em IEEE Transactions on Computers}, 100(9):881--890.

\bibitem[Gupta and Nagar, 1999]{gupta2018matrix}
Gupta, A.~K. and Nagar, D.~K. (1999).
\newblock {\em Matrix Variate Distributions}.
\newblock CRC Press.

\bibitem[Hu et~al., 2020]{Hu2020}
Hu, W., Shen, W., Zhou, H., and Kong, D. (2020).
\newblock Matrix linear discriminant analysis.
\newblock {\em Technometrics}, 62(2):196--205.

\bibitem[Hua et~al., 2007]{hua2007face}
Hua, G., Viola, P.~A., and Drucker, S.~M. (2007).
\newblock Face recognition using discriminatively trained orthogonal rank one
  tensor projections.
\newblock In {\em 2007 IEEE Conference on Computer Vision and Pattern
  Recognition}, pages 1--8. IEEE.

\bibitem[Huber, 1985]{huber1985projection}
Huber, P.~J. (1985).
\newblock Projection pursuit.
\newblock {\em The Annals of Statistics}, 13:435--475.

\bibitem[Lindquist, 2008]{lindquist2008statistical}
Lindquist, M.~A. (2008).
\newblock The statistical analysis of {fMRI} data.
\newblock {\em Statistical Science}, 23(4):439--464.

\bibitem[Liu et~al., 2011]{liu2011discriminant}
Liu, C., He, K., Zhou, J.-l., and Gao, C.-B. (2011).
\newblock Discriminant orthogonal rank-one tensor projections for face
  recognition.
\newblock In {\em Asian Conference on Intelligent Information and Database
  Systems}, pages 203--211. Springer.

\bibitem[Loperfido, 2015]{loperfido2015vector}
Loperfido, N. (2015).
\newblock Vector-valued skewness for model-based clustering.
\newblock {\em Statistics \& Probability Letters}, 99:230--237.

\bibitem[Lopuhaä, 1991]{tau_estimates}
Lopuhaä, H.~P. (1991).
\newblock Multivariate $\tau$-estimators for location and scatter.
\newblock {\em The Canadian Journal of Statistics / La Revue Canadienne de
  Statistique}, 19(3):307--321.

\bibitem[Lu et~al., 2011]{lu2011survey}
Lu, H., Plataniotis, K.~N., and Venetsanopoulos, A.~N. (2011).
\newblock A survey of multilinear subspace learning for tensor data.
\newblock {\em Pattern Recognition}, 44(7):1540--1551.

\bibitem[Mardia, 1970]{mardia1970measures}
Mardia, K.~V. (1970).
\newblock Measures of multivariate skewness and kurtosis with applications.
\newblock {\em Biometrika}, 57(3):519--530.

\bibitem[Maronna, 1976]{maronna1976robust}
Maronna, R.~A. (1976).
\newblock Robust {M}-estimators of multivariate location and scatter.
\newblock {\em Annals of Statistics}, 4(1):51--67.

\bibitem[Mathai and Provost, 1992]{mathai1992quadratic}
Mathai, A.~M. and Provost, S.~B. (1992).
\newblock {\em Quadratic Forms in Random Variables: Theory and Applications}.
\newblock Dekker.

\bibitem[Miettinen et~al., 2017]{Miettinen2017TheSS}
Miettinen, J., Nordhausen, K., Oja, H., Taskinen, S., and Virta, J. (2017).
\newblock The squared symmetric {FastICA} estimator.
\newblock {\em Signal Processing}, 131:402--411.

\bibitem[Mohammadi et~al., 2020]{EvolutionaryAlgorithms2020}
Mohammadi, F.~G., Amini, M.~H., and Arabnia, H. (2020).
\newblock {\em Evolutionary Computation, Optimization, and Learning Algorithms
  for Data Science}, pages 37--65.
\newblock Springer International Publishing.

\bibitem[Nordhausen et~al., 2021]{R_ICTEst}
Nordhausen, K., Oja, H., Tyler, D.~E., and Virta, J. (2021).
\newblock {\em ICtest: Estimating and Testing the Number of Interesting
  Components in Linear Dimension Reduction}.
\newblock R package version 0.3-3.

\bibitem[Pena and Prieto, 2001]{pena2001cluster}
Pena, D. and Prieto, F.~J. (2001).
\newblock Cluster identification using projections.
\newblock {\em Journal of the American Statistical Association},
  96(456):1433--1445.

\bibitem[Pena et~al., 2017]{pena2017clustering}
Pena, D., Prieto, F.~J., and Rendon, C. (2017).
\newblock Clustering big data by extreme kurtosis projections.
\newblock DES - Working Papers. Statistics and Econometrics. WS 24522,
  Universidad Carlos III de Madrid. Departamento de Estadística.

\bibitem[Pe{\~n}a et~al., 2010]{pena2010eigenvectors}
Pe{\~n}a, D., Prieto, F.~J., and Viladomat, J. (2010).
\newblock Eigenvectors of a kurtosis matrix as interesting directions to reveal
  cluster structure.
\newblock {\em Journal of Multivariate Analysis}, 101:1995--2007.

\bibitem[Pfeiffer et~al., 2012]{pfeiffer2012sufficient}
Pfeiffer, R.~M., Forzani, L., and Bura, E. (2012).
\newblock Sufficient dimension reduction for longitudinally measured
  predictors.
\newblock {\em Statistics in Medicine}, 31(22):2414--2427.

\bibitem[Pires and Branco, 2019]{pires2019high}
Pires, A. and Branco, J. (2019).
\newblock High dimensionality: The latest challenge to data analysis.
\newblock {\em arXiv preprint arXiv:1902.04679}.

\bibitem[Radojicic et~al., 2020]{RadojicicNordhausnenOja2020}
Radojicic, U., Nordhausen, K., and Oja, H. (2020).
\newblock Notion of information and independent component analysis.
\newblock {\em Applications of Mathematics}, 65:311--330.

\bibitem[Radojicic et~al., 2021]{radojicic2021}
Radojicic, U., Nordhausen, K., and Virta, J. (2021).
\newblock Large-sample properties of blind estimation of the linear
  discriminant using projection pursuit.
\newblock {\em arXiv preprint arXiv:2103.04678}.

\bibitem[Rousseeuw, 1985]{Rousseeuw1985}
Rousseeuw, P. (1985).
\newblock Multivariate estimation with high breakdown point.
\newblock In Grossmann, W., Pflug, G., Vincze, I., and Wertz, W., editors, {\em
  Mathematical Statistics and Applications Vol. B}, pages 283--297. Springer.

\bibitem[Salomon, 2004]{Salomon2004}
Salomon, R. (2004).
\newblock The curse of high-dimensional search spaces: observing premature
  convergence in unimodal functions.
\newblock In {\em Proceedings of the 2004 Congress on Evolutionary
  Computation}, volume~1, pages 918--923.

\bibitem[Scrucca et~al., 2016]{mclust}
Scrucca, L., Fop, M., Murphy, T.~B., and Raftery, A.~E. (2016).
\newblock {mclust} 5: {c}lustering, classification and density estimation using
  {G}aussian finite mixture models.
\newblock {\em The {R} Journal}, 8(1):289--317.

\bibitem[Tyler et~al., 2009]{Tyler2009}
Tyler, D.~E., Critchley, F., D\"umbgen, L., and Oja, H. (2009).
\newblock Invariant co-ordinate selection.
\newblock {\em Journal of the Royal Statistical Society: Series B (Statistical
  Methodology)}, 71(3):549--592.

\bibitem[Virta et~al., 2021]{rtensorBSS}
Virta, J., Koesner, C.~L., Li, B., Nordhausen, K., Oja, H., and Radojicic, U.
  (2021).
\newblock {\em tensorBSS: Blind Source Separation Methods for Tensor-Valued
  Observations}.
\newblock R package version 0.3.8.

\bibitem[Virta et~al., 2017]{virta2017independent}
Virta, J., Li, B., Nordhausen, K., and Oja, H. (2017).
\newblock Independent component analysis for tensor-valued data.
\newblock {\em Journal of Multivariate Analysis}, 162:172--192.

\bibitem[Wu et~al., 2011a]{wu2011local}
Wu, S., Li, W., Wei, Z., and Yang, J. (2011a).
\newblock Local discriminative orthogonal rank-one tensor projection for image
  feature extraction.
\newblock In {\em The First Asian Conference on Pattern Recognition}, pages
  367--371. IEEE.

\bibitem[Wu et~al., 2011b]{wu2011rank}
Wu, X., Lai, J., and Chen, X. (2011b).
\newblock Rank-1 tensor projection via regularized regression for action
  classification.
\newblock {\em International Journal of Wavelets, Multiresolution and
  Information Processing}, 9(06):1025--1041.

\bibitem[Ye, 2005]{ye2005generalized}
Ye, J. (2005).
\newblock Generalized low rank approximations of matrices.
\newblock {\em Machine Learning}, 61(1-3):167--191.

\bibitem[Zhang and Zhou, 2005]{zhang20052d}
Zhang, D. and Zhou, Z.-H. (2005).
\newblock (2d)$^2${PCA}: Two-directional two-dimensional {PCA} for efficient
  face representation and recognition.
\newblock {\em Neurocomputing}, 69(1-3):224--231.

\bibitem[Zhong et~al., 2015]{zhong-2015}
Zhong, W., Xing, X., and Suslick, K. (2015).
\newblock Tensor sufficient dimension reduction.
\newblock {\em WIREs Computational Statistics}, 7:178--184.

\end{thebibliography}
\newpage
\end{document}